\documentclass{amsart}
\usepackage{amssymb}
\usepackage{amsfonts}
\usepackage{hyperref}
\usepackage{mathrsfs}
\usepackage{tikz-cd}

\newcommand{\bk}{\Bbbk}
\newcommand{\F}{\mathbb{F}}
\newcommand{\Z}{\mathbb{Z}}
\newcommand{\R}{\mathbb{R}}
\newcommand{\Q}{\mathbb{Q}}

\newcommand{\Gm}{{\mathbb{G}_{\mathrm{m}}}}
\newcommand{\Ga}{\mathbb{G}_{\mathrm{a}}}

\newcommand{\Gv}{\check G}
\newcommand{\Tv}{\check T}
\newcommand{\Bv}{\check B}
\newcommand{\Uv}{\check U}

\newcommand{\Lie}{\mathrm{Lie}}
\newcommand{\bX}{\mathbf{X}}
\newcommand{\bXv}{\check\bX}
\newcommand{\fR}{\mathfrak{R}}
\newcommand{\fRv}{\check\fR}
\newcommand{\fRs}{\fR_{\mathrm{s}}}
\newcommand{\fRvs}{\check\fR_{\mathrm{s}}}
\newcommand{\Waff}{W_{\mathrm{aff}}}
\newcommand{\Wext}{W_{\mathrm{ext}}}
\newcommand{\oWext}{{}^{\bo}\Wext}

\newcommand{\spWext}{{}^{\mathit{sp}}\Wext}
\newcommand{\oWaff}{{}^{\bo}\Waff}
\newcommand{\spWaff}{{}^{\mathit{sp}}\Waff}

\newcommand{\ba}{\mathbf{a}}
\newcommand{\bo}{\mathbf{o}}
\newcommand{\bba}{\overline{\ba}}
\newcommand{\bbf}{\mathbf{f}}
\newcommand{\para}{\mathrm{P}}
\newcommand{\Iw}{\mathrm{Iw}}
\newcommand{\Iwu}{\Iw_{\mathrm{u}}}
\newcommand{\Iwul}{\Iw_{p,\mathrm{u}}}
\newcommand{\Iwup}{\Iw_{\mathrm{u}}^+}
\newcommand{\Iwupl}{\Iw_{p,\mathrm{u}}^+}
\newcommand{\ev}{\mathrm{ev}}
\newcommand{\IW}{{\mathcal{IW}}}
\newcommand{\sph}{{\mathrm{sph}}}
\newcommand{\AS}{\mathrm{AS}}
\newcommand{\St}{\mathrm{St}}
\DeclareMathOperator{\ch}{ch}
\newcommand{\coh}{\mathsf{H}}

\newcommand{\cA}{\mathcal{A}}
\newcommand{\cE}{\mathcal{E}}
\newcommand{\cF}{\mathcal{F}}
\newcommand{\cG}{\mathcal{G}}
\newcommand{\cH}{\mathcal{H}}
\newcommand{\cL}{\mathcal{L}}
\newcommand{\Db}{D^{\mathrm{b}}}
\newcommand{\Dbc}{D^{\mathrm{b}}_{\mathrm{c}}}
\newcommand{\Sh}{\mathrm{Sh}}
\newcommand{\Sm}{\mathrm{Sm}}
\newcommand{\perf}{{\mathrm{perf}}}
\newcommand{\Psm}{\mathbf{Psm}}
\newcommand{\cRHom}{R\mathcal{H}\mathit{om}}
\newcommand{\cX}{\mathcal{X}}
\newcommand{\rQ}{\mathsf{Q}}

\newcommand{\dotl}{\mathbin{\bullet_p}}
\newcommand{\boxl}{\mathbin{\square_p}}

\newcommand{\Loop}{\mathrm{L}}
\newcommand{\Gr}{\mathrm{Gr}}
\newcommand{\Fl}{\mathrm{Fl}}
\newcommand{\pt}{\mathrm{pt}}
\newcommand{\id}{\mathrm{id}}

\newcommand{\bi}{\mathbf{i}}
\newcommand{\bit}{\tilde{\mathbf{\imath}}}
\newcommand{\thin}{{\mathrm{thin}}}
\newcommand{\bz}{\mathbf{z}}
\newcommand{\rL}{\mathrm{L}}
\newcommand{\rS}{\mathrm{S}}

\newcommand{\aff}{{\mathrm{aff}}}
\newcommand{\ext}{{\mathrm{ext}}}

\newcommand{\uH}{\underline{H}}
\newcommand{\cM}{\mathcal{M}}
\newcommand{\cN}{\mathcal{N}}

\newcommand{\uN}{\underline{N}}
\newcommand{\Ko}{\mathrm{K}_\oplus}
\newcommand{\triv}{\mathrm{triv}}
\newcommand{\sgn}{\mathrm{sgn}}
\newcommand{\length}{\ell}

\newcommand{\Parity}{\mathrm{Parity}}
\newcommand{\SmParity}{\mathrm{SmParity}}
\newcommand{\D}{\mathbb{D}}
\newcommand{\For}{\mathrm{For}}
\newcommand{\Av}{\mathrm{Av}}
\newcommand{\Perv}{\mathrm{Perv}}
\newcommand{\Rep}{\mathrm{Rep}}

\DeclareMathOperator{\Spec}{Spec}
\newcommand{\Hom}{\mathrm{Hom}}
\newcommand{\bHom}{\mathbb{H}\mathrm{om}}
\newcommand{\bEnd}{\mathbb{E}\mathrm{nd}}
\newcommand{\bHomev}{\mathbb{H}\mathrm{om}^{\mathrm{ev}}}
\newcommand{\Ext}{\mathrm{Ext}}
\newcommand{\End}{\mathrm{End}}
\DeclareMathOperator{\cok}{cok}
\DeclareMathOperator{\rank}{rank}

\newcommand{\la}{\langle}
\newcommand{\ra}{\rangle}
\newcommand{\simto}{\overset{\sim}{\to}}
\newcommand{\can}{\mathrm{can}}

\newcommand{\tboxtimes}{\mathbin{\widetilde{\boxtimes}}}

\numberwithin{equation}{section}
\numberwithin{figure}{section}
\newtheorem{thm}{Theorem}[section]
\newtheorem{lem}[thm]{Lemma}
\newtheorem{prop}[thm]{Proposition}
\newtheorem{cor}[thm]{Corollary}

\theoremstyle{definition}
\newtheorem{defn}[thm]{Definition}

\theoremstyle{remark}
\newtheorem{rmk}[thm]{Remark}

\title{Steinberg quotients and Smith--Treumann localization}

\author{Pramod N. Achar}
\address{Department of Mathematics\\
  Louisiana State University\\
  Baton Rouge, LA 70803\\
  U.S.A.}
\email{pramod.achar@math.lsu.edu}
\thanks{P.A. was supported by NSF Grant No.~DMS-2202012.}

\author{Paul Sobaje}
\address{Department of Mathematics\\
  Georgia Southern University\\
  Statesboro, GA 30460\\
  U.S.A.}
\email{psobaje@georgiasouthern.edu}

\begin{document}

\begin{abstract}
Smith--Treumann localization for sheaves on the affine Grassmannian of a reductive group has previously been studied by Leslie--Lonergan (for spherical sheaves) and by Riche--Williamson (for Iwahori--Whittaker sheaves).  In this paper, we show that the two versions are related by a commutative diagram that involves ``convolution with the Steinberg module.''

As an application, we ``categorify'' certain formal characters of a reductive group called \emph{Steinberg quotients}, previously introduced and studied by the second author.  Specifically, we show that Steinberg quotients describe the stalks of spherical parity sheaves on the $\Z/p\Z$-fixed-locus of the affine Grassmannian.
\end{abstract}

\maketitle

\section{Introduction}

\subsection{Overview: Steinberg quotients}
\label{ss:over-stquot}

Let $\bk$ denote an algebraically closed field of characteristic $p > 0$, and let $G$ be a connected reductive group over $\bk$.  Let $\bX^+$ be the set of dominant weights, and consider the character ring $\Z[\bX]^W$ (see Section~\ref{sec:weyl} for details on notation).  In~\cite{sob1,sob2}, the second author introduced and studied certain elements $t(\lambda) \in \Z[\bX]^W$, called \emph{Steinberg quotients}.  For $\lambda \in \bX^+$, $t(\lambda)$ is defined by
\[
t(\lambda) = \frac{\ch T(\lambda + (p-1)\rho)}{\ch \St},
\]
where $T(\lambda + (p-1)\rho)$ is the indecomposable tilting module of highest weight $\lambda + (p-1)\rho$, and $\St = T((p-1)\rho)$ is the Steinberg module.\footnote{For the claim that $\ch T(\lambda + (p-1)\rho)$ is actually divisible by $\ch \St$ in $\Z[\bX]^W$, see~\cite[\S1]{sob2}.  A related divisibility result was established by Lusztig~\cite{lus:dpm} in the setting of finite Chevalley groups.}
These elements enjoy remarkable combinatorial properties, including positivity and monotonicity phenomona.  However, in general, $t(\lambda)$ is not the character of a $G$-module---\emph{so what is it the character of}?

\subsection{Overview: Smith--Treumann localization}

Let $\Gr = \Loop \Gv /\Loop^+\Gv$ denote the affine Grassmannian for the Langlands dual group to $G$ over an algebraically closed field $\F$ of characteristic $\ell \ne p$, and let $\varpi \subset \F^\times$ be the group of $p$-th roots of unity.  This group acts on $\Gr$ via ``loop rotation,'' and one can consider the fixed-point locus $\Gr^\varpi$.  The operation of \emph{Smith--Treumann localization}, introduced by Treumann in~\cite{treu:smgha} as a sheaf-theoretic counterpart of Smith theory~\cite{smi:fp}, sends $\varpi$-equivariant sheaves on $\Gr$ to sheaves (or at least sheaf-like objects) on $\Gr^\varpi$.  Treumann predicted in~\cite{treu:smgha} that this operation, denoted by $\Psm$, should have applications in characteristic-$p$ representation theory, via the geometric Satake equivalence.

Treumann's prediction has been realized in two ways.  First, Leslie--Lonergan~\cite{ll:psst} showed that when applied to $\Loop^+\Gv$-equivariant sheaves on $\Gr$ (i.e., in the usual setting for the geometric Satake equivalence from~\cite{mv:gld}), Smith--Treumann localization gives a geometric counterpart of the ``Frobenius contraction functor'' of~\cite{gk}.

Second, Riche--Williamson~\cite{rw:st} showed that in the setting of Iwahori--Whit\-ta\-ker sheaves on $\Gr$ (which were shown to give another incarnation of the Satake category in~\cite{bgmrr}), Smith--Treumann localization leads to very slick proofs of both the linkage principle and the tilting character formula, thanks to the combinatorics of the connected components of $\Gr^\varpi$.

These works left open the following question: \emph{how are the two versions of Smith--Treumann localization on $\Gr$ related to one another?}

\subsection{Results of this paper}
\label{ss:results}

In this paper, we obtain a geometric interpretation of the Steinberg quotients $t(\lambda)$, as a consequence of a result relating the two versions of Smith--Treumann localization.  The following is (a slightly imprecise version of) the main geometric theorem of this paper.  

\begin{thm}\label{thm:main-intro}
Assume that $p$ is good for $G$.  The following diagram commutes up to natural isomorphism:
\[
\begin{tikzcd}[column sep=large]
\Parity_{\Loop^+\Gv}(\Gr,\bk) \ar[d, "\cE^\sph_{(p-1)\rho} \star ({-})"'] 
  \ar[rrr, "\Psm", "\text{\normalfont as in Leslie--Lonergan~\cite{ll:psst}}"']
  &&&
  \SmParity_{\sph}(\Gr^\varpi,\bk) \ar[dd, "\Av_\IW"] \\
\Parity_{\Loop^+\Gv}(\Gr,\bk) \ar[d, "\Av_\IW"'] \\
 \Parity_{\IW}(\Gr,\bk) \ar[rrr, "\Psm", "\text{\normalfont as in Riche--Williamson~\cite{rw:st}}"'] &&&
  \SmParity_{\IW}(\Gr^\varpi,\bk)
\end{tikzcd}
\]
\end{thm}

Here, both horizontal arrows are instances of Smith--Treumann localization, and both arrows labelled ``$\Av_\IW$'' are Iwahori--Whittaker averaging functors.  The categories denoted ``$\SmParity$'' consist of \emph{Smith-parity sheaves}.  (For further details, see Section~\ref{sec:convolution}.) The arrow labelled ``$\cE^\sph_{(p-1)\rho} \star ({-})$'' indicates convolution with the parity sheaf corresponding to the Steinberg module under the geometric Satake equivalence.

The fixed-point locus $\Gr^\varpi$ is acted on by the ``$p$-scaled arc group'' $\Loop^+_p\Gv$ (see Section~\ref{ss:scaled}), and the orbits of this group action are naturally in bijection with $\bX^+$.  For each $\lambda \in \bX^+$, let $\Gr^{\varpi}_\lambda$ be the corresponding orbit, and let
\[
\cE^{\varpi,\sph}_\lambda \in \Parity_{\Loop^+_p\Gv}(\Gr^\varpi,\bk)
\]
be the (suitably normalized) indecomposable parity sheaf supported on $\overline{\Gr^{\varpi}_\lambda}$.  In the following theorem, which describes the stalks of $\cE^{\varpi,\sph}_\lambda$, we use the notation $s(\mu) = \sum_{\nu \in W\mu} e^\nu  \in \Z[\bX]^W$.

\begin{thm}\label{thm:stquot-intro}
Assume that $p$ is good for $G$.  We have
\[
t(\lambda) = \sum_{\substack{\mu \in \bX^+ \\ i \in \Z}} \rank \coh^i(\cE^{\varpi,\sph}_\lambda|_{\Gr^{\varpi}_\mu}) s(\mu).
\]
\end{thm}

This theorem is essentially a corollary of Theorem~\ref{thm:main-intro}.  Here is a brief outline of the argument.  First, one checks that both horizontal arrows in Theorem~\ref{thm:main-intro} preserve characters.  In the left-hand column, $\Av_\IW$ preserves characters by~\cite{bgmrr}.  It follows that in the right-hand column, $\Av_\IW$ multiplies characters by $\ch \St$.  Finally, a result from~\cite{rw:scf} implies that both instances of $\Av_\IW$ preserve indecomposability.  Combining these observations, we see that the right-hand instance of $\Av_\IW$ sends $\cE^{\varpi,\sph}_\lambda$ to an indecomposable object corresponding to the tilting $G$-module of highest weight $\lambda + (p-1)\rho$.  Theorem~\ref{thm:stquot-intro} follows.

In addition to Theorem~\ref{thm:stquot-intro}, we record in Section~\ref{ss:applications} two more consequences of Theorem~\ref{thm:main-intro}.  First, we obtain a new proof (when $p$ is good) of the monotonicity property of $t(\lambda)$ that was previously proved by the second author~\cite{sob1,sob2}: see Theorem~\ref{thm:mono}.  Second, we give a representation-theoretic interpretation of decomposition multiplicities for the Leslie--Lonergan version of Smith--Treumann localization: see Theorem~\ref{thm:dmu}.

\subsection{Further comments on the main theorems and their proofs}
\label{ss:further}

The statement of Theorem~\ref{thm:main-intro} above is somewhat imprecise; the actual result we prove can be found in Theorem~\ref{thm:main}.  Here are most salient points:
\begin{itemize}
\item In the main body of the paper, we do \emph{not} assume that $\rho$ lies in $\bX$; instead, we assume that $\bX$ contains an element $\varsigma$ such that $\varsigma - \rho$ is orthogonal to all coroots.  This means that our results apply, for instance, to $\mathrm{GL}_n$.

\item Throughout the paper, we use the formalism for Smith--Treumann localization developed in~\cite{rw:st}, rather than Treumann's original version from~\cite{treu:smgha}.  This version takes as input $\Gm$-equivariant sheaves (e.g., for the loop rotation action of $\Gm$), so all categories in the left-hand column should impose $\Gm$-equivariance.

\item The usual definition of Iwahori--Whittaker averaging doesn't carry over to the Smith category if the latter is defined using the loop rotation $\Gm$-action.  To solve this problem, the category in the upper-right corner of Theorem~\ref{thm:main-intro} must be modified: the loop rotation $\Gm$-action needs to be twisted by the cocharacter $p\varsigma$.

\item As a consequence of the preceding two points, the top arrow in Theorem~\ref{thm:main-intro} does not literally coincide with the functor studied in~\cite{ll:psst}.  Nevertheless, at a combinatorial level (e.g., dimensions of stalks or of $\Hom$-groups), the categories in the top row of Theorem~\ref{thm:main-intro} have identical behavior to those in~\cite{ll:psst}.
\end{itemize}

Let us also comment on the restrictions on $p$.  In the first few sections of the paper, $p$ is arbitrary.  Starting in Section~\ref{ss:bhom}, we require $p$ to not be a torsion prime for $\Gv$.  This condition is needed for some equivariant cohomology calculations, and for some aspects of the theory of parity sheaves.

Finally, the main results of the paper require $p$ to be good for $G$, mainly so that we can invoke a theorem of Mautner--Riche~\cite[Corollary~1.6]{mr:etsps} which says that all (suitably normalized) indecomposable spherical parity sheaves on $\Gr$ are perverse.

\subsection{Organization of the paper}

We begin in Section~\ref{sec:weyl} with preliminaries on the affine Weyl group and the extended affine Weyl group associated to a reductive group, including the ``dot'' and ``box'' actions, alcoves, and coset representatives.  Section~\ref{sec:hecke} introduces the (extended) affine Hecke algebra as well as spherical and antispherical modules.  Sections~\ref{sec:weyl} and~\ref{sec:hecke} do not impose any restriction on $p$.

In Section~\ref{sec:affinegr}, we introduce the affine Grassmannian and affine flag variety; we record parametrizations of orbits; and we introduce various categories of sheaves on these spaces, including parity sheaves and the Smith category.  Starting in Section~\ref{ss:bhom}, we assume that $p$ is not a torsion prime for $\Gv$.

Section~\ref{sec:convolution} is devoted to the study of convolution on $\Gr$ and on $\Gr^\varpi$.  In particular, this section contains technical results on making sense of convolution in the Smith category.  Section~\ref{sec:groth} studies the (split) Grothendieck groups of all categories appearing in Theorem~\ref{thm:main-intro}.

The main results are proved in Section~\ref{sec:geom}: this includes Theorems~\ref{thm:main-intro} and~\ref{thm:stquot-intro}, as well as the two other results mentioned at the end of Section~\ref{ss:results}.  Finally, Appendix~\ref{app:st} contains general background on Smith--Treumann theory in the version developed in~\cite{rw:st}.

\subsection{Acknowledgements}

We are grateful to the University of Georgia for inviting us both to speak in the Algebra Seminar on April 14, 2025, where we began discussing this project.  We would also like to thank George Lusztig for drawing our attention to the paper~\cite{lus:dpm}.

\section{Weyl group combinatorics}
\label{sec:weyl}

\subsection{Preliminaries}

Let $\bk$ be an algebraically closed field of characteristic $p > 0$.  Let $G$ be a connected, reductive group over $\bk$.  Fix a a Borel subgroup $B \subset G$ and a maximal torus $T \subset B$.  Let $\bX$, resp.~$\bXv$, be the character lattice, resp.~cocharacter lattice, of $B$.  We also let $B^+$ denote the opposite Borel subgroup to $B$ with respect to $T$.  Let $\fR \subset \bX$, resp.~$\fRv \subset \bXv$, be the set of roots, resp.~coroots. 

As the notation suggests, we intend $B$ to be the ``negative'' Borel subgroup: we define the set of positive roots $\fR^+ \subset \fR$ to be the roots occuring in the $T$-action on $\Lie(B^+)$.  This choice determines a set $\fRv^+ \subset \fRv$ of positive coroots.  We denote by $\fRs^+$, resp.~$\fRvs^+$, the set of simple roots, resp.~simple coroots.  Let $\bX^+ \subset \bX$ be the set of dominant weights determined by $\fR^+$, i.e.,
\[
\bX^+ = \{ \lambda \in \bX \mid \text{$\la \lambda, \check \alpha \ra \ge 0$ for all $\check \alpha \in \fRv^+$} \}.
\]
As usual, we let
\[
\rho = \frac{1}{2}\sum_{\alpha \in \fR^+} \alpha.
\]
This is an element of $\Q \otimes_\Z \bX$; it may or may not lie in $\bX$.

We assume throughout this paper that $G$ satisfies the following condition:
\begin{equation}\label{eqn:varsig-exist}
\text{There exists a weight $\varsigma \in \bX^+$ such that $\la \varsigma, \check\alpha \ra = 1$ for all $\check\alpha \in \fRvs^+$.}
\end{equation}
For example, if $G$ is semisimple and simply connected, then $\varsigma = \rho$ is the unique weight that satisfies this condition.  However, it is possible for~\eqref{eqn:varsig-exist} to be satisfied even when $\rho$ does not lie in $\bX^+$: for instance, the group $\mathrm{GL}_n$ satisfies~\eqref{eqn:varsig-exist} for all $n$, but $\rho$ lies in $\bX$ for $\mathrm{GL}_n$ only when $n$ is odd.

Let $W$ be the Weyl group of $(G,T)$, and let
\[
\Waff = W \ltimes \Z\fR
\qquad\text{and}\qquad
\Wext = W \ltimes \bX
\]
be the affine Weyl group and the extended affine Weyl group, respectively.  For $\lambda \in \bX$, we let $t_\lambda$ denote the corresponding element of $\Wext$.  There are bijections
\begin{equation}\label{eqn:x-wext}
\begin{aligned}
\bX &\overset{\sim}{\longleftrightarrow} \Wext/W, \\
\bX^+ &\overset{\sim}{\longleftrightarrow} W \backslash\Wext/W
\end{aligned}
\end{equation}
given by $\lambda \mapsto t_\lambda W$ and $\lambda \mapsto W t_\lambda W$, respectively.

Let $S$ be the set of simple reflections of $W$, and let $S_\aff$ be the set of simple reflections of $\Waff$.

We will consider the following two $p$-dilated actions of $\Wext$ on $\bX$: for $w = vt_\lambda$ with $v \in W$ and $\lambda \in \Z\fR$, we set
\[
w \dotl \mu = v(\mu + p\lambda + \rho) - \rho
\qquad\text{and}\qquad
w \boxl \mu = v(\mu + p\lambda).
\]
Of course, these two actions are very closely related.  The $\dotl$ action occurs more commonly in representation theory, but the $\boxl$ action is better suited to the combinatorics that arises in this paper.

Let $\Z[\bX]$ denote the group ring of $\bX$.  For $\lambda \in \bX$, let $e^\lambda$ denote the corresponding element of $\Z[\bX]$.  The group $W$ acts on $\Z[\bX]$ by ring automorphisms, and we may consider the subring $\Z[\bX]^W$ of $W$-invariants.  For $\lambda \in \bX^+$, let
\[
s(\lambda) = \sum_{\mu \in W\lambda} e^\mu
\qquad\text{and}\qquad
\chi(\lambda) = \frac{\sum_{w \in W} (-1)^{\length(w)} e^{w(\lambda + \rho)}}{\sum_{w \in W} (-1)^{\length(w)} e^{w\rho}}.
\]
Both
\[
\{ s(\lambda) : \lambda \in \bX^+\}
\qquad\text{and}\qquad
\{ \chi(\lambda) : \lambda \in \bX^+\}
\]
are $\Z$-bases for $\Z[\bX]^W$.

\subsection{Facets and cosets}
\label{ss:facets}

Consider the real vector space $\bX_\R = \R \otimes_\Z \bX$.  Let
\[
\ba = \{ \lambda \in \bX_\R \mid \text{$-1 < \la \lambda, \check \alpha \ra < 0$ for all $\check \alpha \in \fRv^+$} \}.
\]
We also consider its closure
\[
\bba = \{ \lambda \in \bX_\R \mid \text{$-1 \le \la \lambda, \check \alpha \ra \le 0$ for all $\check \alpha \in \fRv^+$} \}.
\]
Suppose we choose two disjoint subsets $P_{-1}, P_0 \subset \fRv^+$, and then set
\[
\bbf = \left\{ \lambda \in \bX_\R \,\Big|\,
\begin{array}{c}
\text{$\la \lambda, \check \alpha \ra = -1$ for all $\alpha \in P_{-1}$, $\la \lambda, \check \alpha  \ra= 0$ for all $\alpha \in P_0$,} \\
\text{and $-1 < \la \lambda, \check \alpha \ra < 0 $ for all $\alpha \in \fRv^+ \smallsetminus (P_{-1} \cup P_0)$}
\end{array}
\right\}.
\]
If $\bbf$ is nonempty, it is called a \emph{facet} of $\bba$.  Each facet is a locally closed convex subset of $\bba$, and $\bba$ is the union of its facets.  We will denote by $\bo$ the facet $\{0\}$.

It is well known that $p\bba$ is a fundamental domain for the action of $\Waff$ on $\bX_\R$ via $\boxl$.  (This can be deduced from the corresponding statement for $\dotl$, which can be found, for instance, in~\cite[\S II.6.2.(4)]{jan:rag}.)  For each facet $\bbf$, let
\[
W_\bbf = \{ w \in \Waff \mid \text{for all $x \in \bbf$, $w \boxl (-p x) = -p x$} \}
\]
Note that $W_\bbf$ is independent of $p$.  By~\cite[\S II.6.3 and \S II.6.11]{jan:rag}, for any fixed element $x_0 \in \bbf$, we have
\begin{equation}\label{eqn:wf-altdefn}
W_\bbf = \{ w \in \Waff \mid w \boxl (-p x_0) = -p x_0 \}
\end{equation}

By~\cite[\S II.6.3]{jan:rag} again, $W_\bbf$ is a finite Coxeter group, and a parabolic subgroup of $\Waff$. Its set of simple reflections $S_\bbf$ is given by
\[
S_\bbf = W_\bbf \cap S_\aff.
\]
Let $w_\bbf$ be the longest element of $W_\bbf$.  
In particular, $W_\bo = W$ is the (finite) Weyl group of $G$, and $w_\bo$ is its longest element. Let
\[
\Wext^\bbf = \{ w \in \Wext \mid \text{$w$ has maximal length in  $w W_\bbf$} \}.
\]
We also let
\[
\oWext^\bbf = \{ w \in \Wext \mid \text{$w$ has maximal length in $W w W_\bbf$} \}.
\]

In the special case where $\bbf = \ba$, we have $\Wext^\ba = \Wext$.  For $\oWext^\ba$, we often omit the ``$\ba$'' from the notation and simply write
\[
\oWext = \oWext^\ba =  \{ w \in \Wext \mid \text{$w$ has maximal length in $W w$} \}.
\]
The sets $\Waff^\bbf$, $\oWaff^\bbf$, and $\oWaff$ are defined similarly, but replacing ``$\Wext$'' in the definitions above with ``$\Waff$.''

There are of course bijections
\begin{equation}\label{eqn:min-wext}
\begin{aligned}
\Wext^\bbf &\overset{\sim}{\longleftrightarrow} \Wext/W_\bbf, \\
\oWext^\bbf & \overset{\sim}{\longleftrightarrow} W \backslash \Wext/W_\bbf
\end{aligned}
\end{equation}

Let $\bbf \subset \bba$ be a facet.  We say that an element $w \in \Wext$ is:
\begin{itemize}
\item \emph{regular} with respect to $\bbf$ if for all $v_1, v_2 \in W$, we have
\[
v_1 w W_\bbf = v_2 w W_\bbf \qquad\text{if and only if} \qquad
v_1 = v_2.
\]
\item \emph{minimal-regular} with respect to $\bbf$ if it is regular, and minimal in $Ww$.
\end{itemize}
The set of minimal-regular elements that are also \emph{maximal} in their left $W_\bbf$-coset is denoted by
\[
\spWext^\bbf = 
\{ w \in \Wext^\bbf \mid \text{$w$ is minimal-regular with respect to $\bbf$} \}.
\]
When $\bbf = \ba$, the ``regular'' condition is vacuous.  In this case, we often omit the ``$\ba$'' and write
\[
\spWext = \spWext^\ba = \{ w \in \Wext \mid \text{$w$ has minimal length in $W w$} \}.
\]
It is obvious from the definition that
\begin{equation}\label{eqn:spwext-contain}
\spWext^\bbf \subset \spWext \cap \Wext^\bbf.
\end{equation}
See Corollary~\ref{cor:spwext-contain} for a refinement of this.  According to~\cite[Lemma~2.4]{mr:etspc},
\begin{equation}\label{eqn:dom-minimal}
\lambda \in \bX^+
\qquad\text{implies}\qquad
t_\lambda \in \spWext.
\end{equation}

There is a bijection
\[
\spWext \simto \oWext
\qquad\text{given by}\qquad
w \mapsto w_\bo w.
\]

\begin{lem}\label{lem:regular-coset}
Let $\bbf \subset \bba$ be a facet.  The following conditions on an element $w \in \Wext$ are equivalent:
\begin{enumerate}
\item Some element in the double coset $W w W_\bbf$ is regular with respect to $\bbf$.
\item Every element in the double coset $W w W_\bbf$ is regular with respect to $\bbf$.
\item We have $|W w W_\bbf| = |W||W_\bbf|$.
\end{enumerate}
\end{lem}
\begin{proof}
Choose an element $y \in W w W_\bbf$.  Then the double coset $W w W_\bbf = W y W_\bbf$ can be written as
\begin{equation}\label{eqn:regco-eqn}
W w W_\bbf = \bigcup_{v \in W} vy W_\bbf.
\end{equation}
By definition, $y$ is regular with respect to $\bbf$ if and only if the terms on the right-hand side of~\eqref{eqn:regco-eqn} are distinct.  On the other hand, we also have $|W w W_\bbf| \le |W| |W_\bbf|$, with equality if and only if the terms on the right-hand side of~\eqref{eqn:regco-eqn} are distinct.  To summarize, $|W w W_\bbf| = |W||W_\bbf|$ if and only if $y$ is regular with respect to $\bbf$.  As this holds for every $y \in W w W_\bbf$, the lemma follows.
\end{proof}

\begin{lem}\label{lem:regular-left}
Let $\bbf \subset \bba$ be a facet, and let $w \in \Wext$.  If $w$ is minimal-regular with respect to $\bbf$, then all elements of the coset $wW_\bbf$ are minimal-regular with respect to~$\bbf$.
\end{lem}
\begin{proof}
Lemma~\ref{lem:regular-coset} already tells us that all elements of $w W_\bbf$ are at least regular with respect to $\bbf$.  It remains to show that for all $v \in W_\bbf$, the element $wv$ is minimal in $Wwv$.  By induction on the length of $v$, it suffices to consider the case where $v$ is a simple reflection $t \in S_\bbf$.

Suppose instead that $wt$ is \emph{not} minimal in $Wwt$. Then there is some simple reflection $s \in S$ such that $swt < wt$.  Since $w$ is minimal in $Ww$, we have
\begin{equation}\label{eqn:special-test}
sw > w.
\end{equation}
According to~\cite[Proposition~5.9]{hum}, the fact that $swt < wt$ implies that either $sw \le wt$ or $sw \le w$.  The latter contradicts~\eqref{eqn:special-test}, so we must have $sw \le wt$.  Next,~\eqref{eqn:special-test} implies that $\length(sw) = \length(w) + 1$.  On the other hand, $\length(wt) = \length(w) \pm 1$.  Since $sw \le wt$, we must have $\length(sw) = \length(wt)$, and we deduce that $sw = wt$.  We therefore have $sw \in w W_\bbf$, and hence $sw W_\bbf = w W_\bbf$.  But this contradicts the assumption that $w$ is regular with respect to $\bbf$.
\end{proof}

\begin{lem}\label{lem:regular-right}
Let $\bbf \subset \bba$ be a facet, and let $w \in \Wext^\bbf$.  The following are equivalent:
\begin{enumerate}
\item The element $w$ is regular with respect to $\bbf$.\label{it:rr-reg}
\item We have $Ww \subset \Wext^\bbf$.\label{it:rr-contain}
\item We have $Ww = Ww W_\bbf \cap \Wext^\bbf$.\label{it:rr-cap}
\end{enumerate}
\end{lem}
\begin{proof}
\eqref{it:rr-contain}${}\Longrightarrow{}$\eqref{it:rr-reg}.
Suppose that $Ww \subset \Wext^\bbf$.  If $v_1, v_2 \in W$ and $v_1 \ne v_2$, then $v_1w \ne v_2w$.  Since each left $W_\bbf$-coset contains a unique maximal element, and since $v_1w$ and $v_2w$ are both maximal in their respective left $W_\bbf$-cosets, we conclude that $v_1w W_\bbf \ne v_2 w W_\bbf$.  Thus, $w$ is regular with respect to $\bbf$.

\eqref{it:rr-reg}${}\Longrightarrow{}$\eqref{it:rr-contain}.
Suppose that $w$ is regular with respect to $\bbf$.  We must show that for all $v \in W$, we have $vw \in \Wext^\bbf$.  By Lemma~\ref{lem:regular-coset}, all elements $vw$ are again regular with respect to $\bbf$.  By induction on the length of $v$, it suffices to treat the case where $v$ is a simple reflection $s \in S$.  If $sw \notin \Wext^\bbf$, there is a simple reflection $t \in S_\bbf$ such that $sw < swt$.  Since $w \in \Wext^\bbf$, we have
\begin{equation}\label{eqn:regular-test}
wt < w.
\end{equation}
By~\cite[Proposition~5.9]{hum}, $sw < swt$ implies that either $w \le swt$ or $w \le wt$.  The latter contradicts~\eqref{eqn:regular-test}, so we must have $w \le swt$.  By~\eqref{eqn:regular-test}, $\length(wt) = \length(w) - 1$, so $\length(swt) = \length(w) -1 \pm 1$.  Since $w \le swt$, we must have $\length(w) = \length(swt)$, and we deduce that $w = swt$.  We therefore have $sw \in w W_\bbf$, and hence $sw W_\bbf = w W_\bbf$.  But this contradicts the assumption that $w$ is regular with respect to $\bbf$.

\eqref{it:rr-cap}${}\Longrightarrow{}$\eqref{it:rr-contain}.
Obvious.

\eqref{it:rr-reg}, \eqref{it:rr-contain}${}\Longrightarrow{}$\eqref{it:rr-cap}.
We have already shown that conditions~\eqref{it:rr-reg} and~\eqref{it:rr-contain} are equivalent; now assume that both of these hold.  The set $Ww W_\bbf \cap \Wext^\bbf$ contains one representative from each left $W_\bbf$-coset contained in $W w W_\bbf$, so by Lemma~\ref{lem:regular-coset}, $|W w W_\bbf \cap \Wext^\bbf| = |W|$.  Since $Ww \subset Ww W_\bbf \cap \Wext^\bbf$, and since the two sets have the same cardinality, they coincide.
\end{proof}

\subsection{More on the box action}

Recall from Section~\ref{ss:facets} that $p\bba$ (and hence also $-p\bba$) is a fundamental domain for the $\boxl$-action of $\Waff$ on $\bX_\R$.  It follows that $-p\bba \cap \bX^+$ is a fundamental domain for the $\boxl$-action of $\Waff$ on $\bX$.  That is,
\[
\bX = \bigsqcup_{\lambda \in -p\bba \cap \bX^+} \Waff \boxl \lambda.
\]

For $\lambda \in -p\bba \cap \bX^+$, let $\bbf_\lambda$ be the facet of $\ba$ containing $-\frac{1}{p}\lambda$.  By~\eqref{eqn:wf-altdefn}, we have
\[
W_{\bbf_\lambda} = \{ w \in \Waff \mid w \boxl \lambda = \lambda \}
\]
so
\begin{equation}\label{eqn:bx-waff-bij}
\bX \cong \bigsqcup_{\lambda \in -p\bba \cap \bX^+} \Waff/W_{\bbf_\lambda} \cong \bigsqcup_{\lambda \in -p\bba \cap \bX^+} \Waff^{\bbf_\lambda}.
\end{equation}
In concrete terms, for $\lambda \in -p\bba \cap \bX^+$, the bijection above sends $w \in \Waff^{\bbf_\lambda}$ to $w \boxl \lambda$.  By Lemma~\ref{lem:boxl-wext} below,~\eqref{eqn:bx-waff-bij} restricts to a bijection
\begin{equation}\label{eqn:xpp-spwaff-bij}
\bX^{++} \cong \bigsqcup_{\lambda \in -p\bba \cap \bX^+} \spWaff^{\bbf_\lambda}.
\end{equation}

Lemma~\ref{lem:boxl-wext} also says that~\eqref{eqn:bx-waff-bij} yields a bijection $-\bX^+ \cong \bigsqcup_{\lambda \in -p\bba \cap \bX^+} \oWaff^{\bbf_\lambda}$.  We can modify this to obtain a bijection
\begin{equation}\label{eqn:xp-owaff-bij}
\bX^+ \cong \bigsqcup_{\lambda \in -p\bba \cap \bX^+} W \backslash \Waff/W_{\bbf_\lambda} \cong \bigsqcup_{\lambda \in -p\bba \cap \bX^+} \oWaff^{\bbf_\lambda}
\end{equation}
given by sending $w \in \oWaff^{\bbf_\lambda}$ to $w_\bo(w \boxl \lambda)$.

\begin{lem}\label{lem:boxl-regular}
Let $\lambda \in -p\bba \cap \bX^+$, and let $w \in \Waff^{\bbf_\lambda}$.  We have that $w$ is regular with respect to $\bbf_\lambda$ if and only if $\la w \boxl \lambda, \check \beta \ra \ne 0$ for all coroots $\check\beta$.
\end{lem}
\begin{proof}
For $v \in W$, consider the coset $vwW_{\bbf_\lambda}$.  Let $m(v)$ be the unique element of maximal length in this coset, i.e., $m(v) \in vwW_{\bbf_\lambda} \cap \Waff^{\bbf_\lambda}$.  The image of the map $m: W \to \Waff^{\bbf_\lambda}$ consists of representatives for the left $W_{\bbf_\lambda}$-cosets contained in $W w W_{\bbf_\lambda}$, and hence
\[
|W w W_{\bbf_\lambda}| = |(\text{image of $m$})| |W_{\bbf_\lambda}|.
\]
By Lemma~\ref{lem:regular-coset}, we see that $w$ is regular if and only $m$ is injective.

On the other hand, if we compose $m$ with the bijection~\eqref{eqn:bx-waff-bij}, we get a map
\[
m': W \to \bX
\qquad\text{given by}\qquad
m'(v) = m(v) \boxl \lambda = vw \boxl \lambda.
\]
The map $m$ is injective if and only if $m'$ is injective, and $m'$ is injective if and only if the stabilizer in $W$ of the weight $w \boxl \lambda$ is trivial.

The stabilizer in $W$ of $w \boxl \lambda$ is a parabolic subgroup of $W$.  In particular, it is generated by the reflections that stabilize $w \boxl \lambda$.  Thus, the stabilizer is trivial if and only if no reflection stabilizes $w \boxl \lambda$.  Of course, a reflection corresponding to a root $\beta$ stabilizes $w \boxl \lambda$ if and only if $\la w \boxl \lambda, \check \beta \ra = 0$.

Combining the observations above, we find that $w$ is regular if and only if $\la w \boxl \lambda, \check \beta \ra \ne 0$ for all coroots $\check\beta$.
\end{proof}

\begin{lem}\label{lem:boxl-wext}
Let $\lambda \in -p\bba \cap \bX^+$, and let $w \in \Waff^{\bbf_\lambda}$.
\begin{enumerate}
\item We have $w \in \oWaff^{\bbf_\lambda}$ if and only if $w \boxl \lambda \in -\bX^+$.\label{it:boxl-max}
\item We have $w \in \spWaff^{\bbf_\lambda}$ if and only if $w \boxl\lambda \in \bX^{++}$.\label{it:boxl-spec}
\end{enumerate}
\end{lem}
\begin{proof}
Write $w = vt_\mu$ with $v \in W$ and $\mu \in \bX$.  Let $\beta$ be a simple root, and let $s_\beta$ be the corresponding simple reflection.  We then have
\begin{equation}\label{eqn:boxl-expand}
\la w \boxl \lambda, \check \beta\ra 
= \la v(\lambda + p\mu), \check\beta \ra
= \underbrace{\la \lambda, v^{-1}(\check\beta)\ra}_{\text{(a)}} + \underbrace{p \la \mu, v^{-1}(\check\beta)\ra}_{\text{(b)}}
\end{equation}
We will study the terms (a) and (b) in more detail below.

\textit{Step 1. We have $\la w \boxl \lambda, \check \beta \ra = 0$ if and only if $s_\beta w W_{\bbf_\lambda} = wW_{\bbf_\lambda}$.  Moreover, if these conditions hold, then $s_\beta w < w$.} We have
\[
\la w \boxl \lambda, \check \beta \ra = 0
\qquad\text{if and only if}\qquad
s_\beta(w \boxl \lambda) = w \boxl \lambda.
\]
Since $s_\beta(w \boxl \lambda) = (s_\beta w) \boxl \lambda$, the conditions above are equivalent to all of the following:
\[
(s_\beta w) \boxl \lambda = w \boxl \lambda,
\qquad
w^{-1}s_\beta w \in W^{{\bbf_\lambda}},
\qquad
s_\beta w W_{\bbf_\lambda} = w W_{\bbf_\lambda}.
\]
By assumption, $w$ lies in $\Waff^{\bbf_\lambda}$, i.e., it is the largest element in the coset $wW_{\bbf_\lambda}$, so if the conditions above hold, we must have $s_\beta w < w$.

\textit{Step 2. Assume $s_\beta v > v$. We have $s_\beta w < w$ if and only if $\la w \boxl \lambda, \check \beta \ra \le 0$.}  Our assumption implies that $v^{-1}(\beta) \in \fR^+$.  From the formula for length in $\Waff$ (see~\cite[Proposition~1.23]{im:sbd}), we have
\[
\length(s_\beta w) = \length(w) - |\la \mu, v^{-1}\check\beta \ra| + |1 + \la \mu, v^{-1}\check\beta\ra|.
\]
Now, $s_\beta w < w$ if and only if $\length(s_\beta w) < \length(w)$.  From the formula above, we obtain
\begin{equation}\label{eqn:svv-cond1}
s_\beta w < w
\qquad\text{if and only if}\qquad
\la \mu, v^{-1}\check\beta \ra < 0.
\end{equation}

If $s_\beta w < w$, then~\eqref{eqn:svv-cond1} says that $\la \mu, v^{-1}\check\beta \ra \le -1$, so~\eqref{eqn:boxl-expand}(b) is${}\le -p$.  On the other hand, since $\lambda \in -p\bba$ and $v^{-1}(\check\beta) \in \fRv^+$, we see that~\eqref{eqn:boxl-expand}(a) is${}\le p$.  We conclude that $\la w \boxl \lambda, \check\beta \ra \le 0$.

Conversely, suppose $\la w \boxl \lambda, \check\beta \ra \le 0$.  If in fact $\la w \boxl \lambda, \check\beta \ra = 0$, then $s_\beta w < w$ by Step~1.  If $\la w \boxl \lambda, \check\beta \ra < 0$, then 
\[
\la v(p \mu), \check \beta\ra = \la w \boxl \lambda, \check\beta \ra - \la \lambda, v^{-1}(\check\beta) \ra < 0.
\]
It follows that $\la \mu, v^{-1}(\check\beta)\ra < 0$, and by~\eqref{eqn:svv-cond1}, $s_\beta w < w$.

\textit{Step 3. Assume $s_\beta v < v$. We have $s_\beta w < w$ if and only if $\la w \boxl \lambda, \check \beta \ra \le 0$.}  This time, we have $-v^{-1}(\beta) \in \fR^+$, and the length formula yields
\[
\length(s_\beta w) = \length(w) - |1 + \la \mu, -v^{-1}\check\beta \ra| + |\la \mu, -v^{-1}\check\beta\ra|.
\]
We see that
\begin{equation}\label{eqn:svv-cond2}
s_\beta w < w
\qquad\text{if and only if}\qquad
\la \mu, -v^{-1}(\check\beta) \ra \ge 0.
\end{equation}

If $s_\beta w < w$, then~\eqref{eqn:svv-cond2} implies that~\eqref{eqn:boxl-expand}(b) is${}\le 0$.  On the other hand, since $v^{-1}(\check\beta) \in - \fRv^+$, we see that~\eqref{eqn:boxl-expand}(a) is also${}\le 0$.  We conclude that $\la w \boxl \lambda, \check\beta \ra \le 0$.

Conversely, suppose $\la w \boxl \lambda, \check\beta \ra \le 0$.  As in Step~2, it is enough to treat the case where $\la w \boxl \lambda, \check\beta \ra < 0$.
Since $\la \lambda, -v^{-1}(\check\beta)\ra \le p$, we have
\[
\la v(p \mu), \check \beta\ra = \la w \boxl \lambda, \check\beta \ra + \la \lambda, -v^{-1}(\check\beta) \ra < p.
\]
But $\la v(p \mu), \check \beta\ra$ is an integer that is divisble by $p$, so if it is${}<p$, it must in fact be${}\le 0$.  By~\eqref{eqn:svv-cond2}, we conclude that $s_\beta w < w$.

\textit{Step 4. Proof of part~\eqref{it:boxl-max}.}  The element $w \in \Waff^{\bbf_\lambda}$ lies in $\oWaff^{\bbf_\lambda}$ if and only if $s_\beta w < w$ for all simple roots $\beta$.  Similarly, $w \boxl \lambda$ lies in $-\bX^+$ if and only if $\la w \boxl \lambda, \check\beta \ra \le 0$ for all simple roots $\beta$.  The claim thus follows from Steps~2 and~3.

\textit{Step 5. We have $w \in \spWaff \cap \Waff^{\bbf_\lambda}$ if and only if $w \boxl\lambda \in \bX^{++}$.}  This is essentially identical to Step 4: just observe that Steps~2 and~3 say that $s_\beta w > w$ if and only if $\la w \boxl \lambda, \check\beta \ra > 0$.

\textit{Step 6. Proof of part~\eqref{it:boxl-spec}.} In view of Step~5 and~\eqref{eqn:spwext-contain}, what remains to be shown is that if $w \in \spWaff \cap \Waff^{\bbf_\lambda}$, then $w$ is regular with respect to $\bbf_\lambda$.  Since $w \boxl \lambda \in \bX^{++}$, we see that $\la w \boxl \lambda, \check \beta\ra \ne 0$ for all coroots $\check\beta$, so $w$ is regular by Lemma~\ref{lem:boxl-regular}.
\end{proof}

\begin{cor}\label{cor:spwext-contain}
Let $\bbf \subset \bba$ be a facet.
\begin{enumerate}
\item We have $\spWaff^\bbf = \spWaff \cap \Waff^\bbf$.
\item If $w \in \Waff$ is regular, the double coset $W w W_\bbf$ contains a unique element of $\spWaff^\bbf$.  Thus, there is a bijection
\[
\left\{
\begin{array}{c}
\text{double cosets $Ww W_\bbf$ where} \\
\text{$w \in \Waff$ is regular with respect to $\bbf$}
\end{array}
\right\}
\overset{\sim}{\leftrightarrow}
\spWaff^\bbf.
\]
\end{enumerate}
\end{cor}
\begin{proof}
The first statement was established in Steps~5 and~6 of the preceding proof.  For the second statement, the double coset $W w W_\bbf$ contains some element of $\Waff^\bbf$, so we may assume without loss of generality that $w \in \Waff^\bbf$.  By Lemma~\ref{lem:regular-right}, we have $Ww = Ww W_\bbf \cap \Waff^\bbf$, so we are reduced to showing that $Ww$ contains a unique element of $\spWaff^\bbf$.  

Choose $p$ and $\lambda \in \bX$ such that $-\frac{1}{p}\lambda \in \bbf$.  By Lemma~\ref{lem:boxl-regular} (and its proof), the map $W \to \bX$ given by $v \mapsto vw \boxl \lambda$ is injective.  The image of this map contains a unique dominant weight (which necessarily lies in $\bX^{++}$), say $\mu = vw \boxl \lambda$.  By Lemma~\ref{lem:boxl-wext}, $vw$ is the unique element of $Ww \cap \spWaff^\bbf$.
\end{proof}

\section{Hecke algebras and modules}
\label{sec:hecke}

\subsection{The extended affine Hecke algebra}

Let $v$ be an indeterminate, and consider the ring of Laurent polynomials $\Z[v,v^{-1}]$.  Let $\cH_\ext$ be the extended affine Hecke algebra associated to $\Wext$ over $\Z[v,v^{-1}]$, defined as in, say,~\cite[\S4.1]{rw:tmpcb} or~\cite[\S3.1]{rw:scf}.  This algebra comes with two $\Z[v,v^{-1}]$-bases, denoted by
\[
\{ H_w : w \in \Wext \}
\qquad\text{and}\qquad
\{ \uH_w : w \in \Wext \},
\]
and called the \emph{standard} and \emph{Kazhdan--Lusztig} bases, respectively.  More generally, for each prime number $p$, there is a basis
\[
\{ {}^p\! \uH_w : w \in \Wext \},
\]
called the $p$-canonical basis.  

We can also consider the group ring $\Z[\Wext]$.  For $w \in \Wext$, let $h_w$ denote the corresponding element of $\Z[\Wext]$.  There is a canonical ring isomorphism
\[
\cH_\ext|_{v=1} \simto \Z[\Wext]
\qquad\text{given by}\qquad
H_w \mapsto h_w.
\]

There is also the (nonextended) affine Hecke algebra $\cH_\aff$, defined by
\[
\cH_\aff = \mathrm{span}_{\Z[v,v^{-1}]}\  \{ H_w : w \in \Waff \}.
\]

\subsection{Hecke algebras associated to facets}

For each facet $\bbf \subset \bba$, we can consider the Hecke algebra $\cH_\bbf$ of the finite Coxeter group $W_\bbf$. This is identified with a subalgebra of $\cH_\ext$ (indeed, of $\cH_\aff$):
\[
\cH_\bbf = \mathrm{span}_{\Z[v,v^{-1}]}\ \{ H_w \mid w \in W_\bbf \}.
\]
For $\bbf = \bo$, the ring $\cH_\bo$ is the Hecke algebra of the Weyl group $W$.  On the other hand, $\cH_\ba$ is the just the ring $\Z[v,v^{-1}]$ itself.

Let $w_\bbf$ be the longest element of $W_\bbf$.  It is well known that
\begin{equation}\label{eqn:uh-long-formula}
{}^p\!\uH_{w_\bbf} = \sum_{w \in W_\bbf} v^{\length(w_\bbf) - \length(w)} H_w = \sum_{w \in W_\bbf} v^{\length(w)} H_{ww_\bbf}
\end{equation}
In particular, this element is independent of $p$.  Furthermore, for any $w \in W_\bbf$, we have
\begin{equation}\label{eqn:uh-long}
H_w \cdot \uH_{w_\bbf} = \uH_{w_\bbf} \cdot H_w = v^{-\length(w)} \uH_{w_\bbf}.
\end{equation}

\subsection{Spherical modules}

The right $\cH_\ext$-module $\cM$ given by
\[
\cM = \uH_{w_\bo} \cH_\ext
\]
is called the \emph{spherical module}.  It is easily deduced from~\eqref{eqn:uh-long} that an element $\sum a_w H_w \in \cH_\ext$ lies in $\cM$ if and only if we have
\[
a_{yw} = v^{\length(y)} a_w
\qquad\text{for all $y \in W$ and $w \in \oWext$.}
\]
For $w \in \oWext$, set
\[
M_w = \uH_{w_\bo} H_{w_\bo w} = \sum_{y \in W} v^{\length(y)} H_{y w} = \sum_{u \in Ww} v^{\length(w) - \length(u)} H_u.
\]
Then $\{ M_w : w \in \oWext \}$ is a $\Z[v,v^{-1}]$-basis for $\cM$.

\begin{rmk}
In some sources (e.g., \cite{rw:scf}), the standard basis for the spherical module is parametrized by \emph{minimal} coset representatives for $W \backslash \Wext$, rather than maximal coset representatives as we use here.
\end{rmk}

\begin{rmk}\label{rmk:m-alt-defn}
Let $\Z[v,v^{-1}]_\triv$ denote the ``trivial'' $\cH_\bo$-module: that is, the module whose underlying set is $\Z[v,v^{-1}]$, on which $H_w \in \cH_\bo$ acts by $v^{-\length(w)}$.  Then, by~\eqref{eqn:uh-long}, there is a right $\cH_\ext$-module homomorphism
\[
\Z[v,v^{-1}]_\triv \otimes_{\cH_\bo} \cH_\ext \to \cM
\qquad\text{given by}\qquad
1 \otimes h \mapsto \uH_{w_\bo} h.
\]
In fact, this map is an isomorphism.  In some sources, the spherical module is defined to be $\Z[v,v^{-1}]_\triv \otimes_{\cH_\bo} \cH_\ext$.
\end{rmk}

We generalize this construction as follows: given a facet $\bbf \subset \bba$, let
\[
\cM^\bbf = \uH_{w_\bo}\cH_\ext \cap \cH_\ext\uH_{w_\bbf}.
\]
(In general, this is no longer an $\cH_\ext$-module.)  For $w \in \oWext^\bbf$, set
\begin{equation}\label{eqn:mfw-defn}
M_w^\bbf = \sum_{u \in WwW_\bbf} v^{\length(w) - \length(u)} H_u = \sum_{u \in \oWext \cap W w W_\bbf} v^{\length(w) - \length(u)} M_u
\end{equation}
Then $\{ M_w^\bbf : w \in \oWext^\bbf \}$ is a $\Z[v,v^{-1}]$-basis for $\cM^\bbf$.  We also let $\cM^\bbf_\aff$ be the submodule of $\cM^\bbf$ spanned by $\{ M_w^\bbf : w \in \oWaff^\bbf \}$.

\begin{lem}\label{lem:mfw-formula}
Let $\bbf \subset \bba$ be a facet, and let $w \in \oWext^\bbf$.  We have
\[
M_w^\bbf = \Big( \sum_{y \in \Wext^\bbf \cap W w W_\bbf} v^{\length(w) - \length(y)} H_{yw_\bbf} \Big) \uH_{w_\bbf}.
\]
\end{lem}
\begin{proof}
Using~\eqref{eqn:uh-long-formula}, the statement of the lemma can be rewritten as the claim that
\[
M_w^\bbf = \sum_{\substack{ y \in \Wext^\bbf \cap W w W_\bbf \\ v \in W_\bbf}} v^{\length(w) - \length(y) + \length(w_\bbf) - \length(v)} H_{yw_\bbf} H_v
\]
Comparing this with the definition of $M^\bbf_w$, we see that the lemma follows from the following three observations:
\begin{enumerate}
\item Each $u \in W w W_\bbf$ can be written uniquely as $u = yw_\bbf v$, where $y \in \Wext^\bbf$ and $v \in W_\bbf$.
\item For $u = yw_\bbf v$ as above, we have $\length(u) = \length(y) - \length(w_\bbf) + \length(v) = \length(y w_\bbf) + \length(v)$.
\item For $u = yw_\bbf v$ as above, we have $H_{y w_\bbf} H_v = H_u$.\qedhere
\end{enumerate}
\end{proof}

\begin{lem}\label{lem:theta-sph}
There is a surjective map of abelian groups
\[
\theta^\sph: \bigoplus_{\lambda \in -p\bba \cap \bX^+} \cM^{\bbf_\lambda}|_{v = 1}  \to \Z[X]^W
\qquad\text{given by}\qquad
M_w^{\bbf_\lambda} \mapsto s(w_\bo(w \boxl \lambda)).
\]
This map restricts to an isomorphism of abelian groups
\[
\theta^\sph_\aff: \bigoplus_{\lambda \in -p\bba \cap \bX^+} \cM_\aff^{\bbf_\lambda}|_{v = 1}  \to \Z[X]^W
\]
\end{lem}
\begin{proof}
From the bijection~\eqref{eqn:xp-owaff-bij}, we see that $\theta^\sph_\aff$ sends the basis $\{ M_w^{\bbf_\lambda} : \lambda \in -p\bba \cap \bX^+,\ w \in \oWaff^\bbf \}$ of its domain to the basis $\{ s(\mu) : \mu \in \bX^+ \}$ of its codomain.
\end{proof}

\subsection{Antispherical modules}

Let $\Z[v,v^{-1}]_\sgn$ denote the ``sign'' representation of $\cH_\bo$: that is, the module whose underlying set is $\Z[v,v^{-1}]$, and on which $H_w \in \cH_\bo$ acts by $(-v)^{\length(w)}$.  The \emph{antispherical module} is the right $\cH_\ext$-module $\cN$ given by
\[
\cN = \Z[v,v^{-1}]_\sgn \otimes_{\cH_\bo} \cH_\ext.
\]
Let
\[
\pi_\cN: \cH_\ext \to \cN
\]
be the right $\cH_\ext$-module map given by $\pi_\cN(h) = 1 \otimes h$.  For $w \in \spWext$, we set
\begin{equation}\label{eqn:nw-defn}
N_w = \pi_\cN(H_w)
\qquad\text{and}\qquad
\uN_w = \pi_\cN(\uH_w).
\end{equation}
The sets $\{ N_w : w \in \spWext\}$ and $\{ \uN_w : w \in \spWext\}$ are each $\Z[v,v^{-1}]$-bases for $\cN$.  By~\eqref{eqn:dom-minimal}, $t_\varsigma$ lies in $\spWext$, and so the elements $N_{t_\varsigma}$ and $\uN_{t_\varsigma}$ are defined.

\begin{lem}\label{lem:un-sigma}
We have
\[
\uN_{t_\varsigma} = N_{t_\varsigma w_\bo} \uH_{w_\bo}.
\]
\end{lem}
\begin{proof}
According to~\cite[Lemmas~2.2 and 2.4]{mr:etspc}, for all $z \in W$, the element $t_\varsigma z = z t_{z^{-1}(\varsigma)}$ also lies in $\spWext$, and we have
\[
\length(t_\varsigma z) = \length(t_\varsigma) - \length(z).
\]
It follows that
\[
\length(t_\varsigma w_\bo z) = \length(t_\varsigma w_\bo) + \length(z),
\]
and hence that $H_{t_\varsigma w_\bo z} = H_{t_\varsigma w_\bo} H_z$.  Applying $\pi_\cN$, we obtain $N_{t_\varsigma w_\bo z} = N_{t_\varsigma w_\bo} H_z$.  Using~\cite[Lemma~3.4]{rw:scf} and~\eqref{eqn:uh-long-formula} we find that
\[
\uN_{t_\varsigma} = \sum_{z \in W} v^{\length(w_\bo) - \length(z)} N_{t_\varsigma w_\bo z} 
= \sum_{z \in W} v^{\length(w_\bo) - \length(z)} N_{t_\varsigma w_\bo} H_z = N_{t_\varsigma w_\bo} \uH_{w_\bo}. \qedhere
\]
\end{proof}

\begin{lem}\label{lem:nw-irreg}
Let $\bbf \subset \bba$ be a facet, and let $w \in \Wext$.  If $w$ is not regular with respect to $\bbf$, then
\[
\pi_\cN(H_w) \uH_{w_\bbf} = 0.
\]
\end{lem}
\begin{proof}
Choose a prime $p$ and a weight $\lambda \in \bX$ such that $-\frac{1}{p}\lambda \in \bbf$.  In other words, $\bbf = \bbf_\lambda$.  To prove the lemma, we will first consider various special cases.

Let us first treat the special case where $w \in \oWaff^\bbf$.  Let $\mu = w \boxl \lambda$.  By Lemma~\ref{lem:boxl-wext}, $\mu \in -\bX^+$, and by Lemma~\ref{lem:boxl-regular}, there exists a coroot $\check\beta$ such that $\la \mu, \check\beta\ra = 0$.  Since $\mu$ is antidominant, we may assume that $\check\beta$ is a \emph{simple} coroot.  Let $s$ be the corresponding simple reflection.  By Step~1 in the proof of Lemma~\ref{lem:boxl-wext}, $sw W_\bbf = w W_\bbf$.  That is, there exists a $y \in W_\bbf$ such that $sw = wy$.  We have $\length(sw) = \length(w) \pm 1$, but since $w$ is maximal in $wW_\bbf$, we must actually have $\length(sw) = \length(wy) = \length(w) - 1$.  This also shows that $y$ must be a simple reflection in $W_\bbf$. It follows that $H_s H_{sw} = H_w = H_{wy} H_y$.  Applying $\pi_\cN$, we obtain $-v \pi_\cN(H_{wy}) = \pi_\cN(H_w) = \pi_\cN(H_{wy})H_y$.  Since $\uH_y = H_y + v$, we deduce that $\pi_\cN(H_{wy}) \uH_y = 0$, and hence that
\[
\pi_\cN(H_w) \uH_y = 0.
\]
Next, recalling that $\uH_y \uH_{w_\bbf} = (v+v^{-1}) \uH_{w_\bbf}$, we obtain
\[
(v+v^{-1}) \pi_\cN(H_w) \uH_{w_\bbf} = 0.
\]
Since $\cN$ is free over the integral domain $\Z[v,v^{-1}]$, it follows that $\pi_\cN(H_w) \uH_{w_\bbf} = 0$.

Now let $w$ be an arbitrary element of $\Waff$.  Let $w'$ be the unique element of maximal length in the double coset $W w W_\bbf$.  Then $w' \in \oWaff^\bbf$, and $w'$ is also not regular with respect to $\bbf$, by Lemma~\ref{lem:regular-coset}.  Write
\[
w' = u w y
\]
where $u \in W$, $y \in W_\bbf$, and $\length(w') = \length(u) + \length(w) + \length(y)$.  Then $H_{w'} = H_u H_w H_y$, so
\[
\pi_\cN(H_{w'}) = (-v)^{\length(u)} \pi_\cN(H_w) H_y.
\]
Using~\eqref{eqn:uh-long}, we have
\[
\pi_\cN(H_{w'})\uH_{w_\bbf} = (-v)^{\length(u)} v^{-\length(y)} \pi_\cN(H_w) \uH_{w_\bbf}.
\]
The left-hand side vanishes by the previous paragraph, and we conclude that $\pi_\cN(H_w) \uH_{w_\bbf} = 0$.

Finally, let $w$ be an arbitrary element of $\Wext$.  We can write $w = w'\omega$ where $w' \in \Waff$, and $\length(\omega) = 0$.  Under the $\boxl$-action, $\omega$ preserves $\bba$, but it may send $\bbf$ to another facet $\bbf'$.  In this case, we have $\omega W_\bbf \omega^{-1} = W_{\bbf'}$ and $\omega w_\bbf \omega^{-1} = w_{\bbf'}$, as well as $H_\omega \uH_{w_\bbf} H_{\omega}^{-1} = \uH_{w_{\bbf'}}$.  We deduce that
\[
\pi_\cN(H_w) \uH_{w_\bbf} = \pi_\cN(H_{w'}) \uH_{w_{\bbf'}} H_\omega.
\]
A short calculation shows that if $w$ is not regular with respect to $\bbf$, then $w'$ is not regular with respect to $\bbf'$, so by the previous paragraph, the right-hand side above vanishes, and we conclude that $\pi_\cN(H_w) \uH_{w_\bbf} = 0$.
\end{proof}

\begin{lem}\label{lem:nw-mult}
Let $\bbf \subset \bba$ be a facet.  Let $w \in \spWext$, and write
\[
w = w'u
\qquad\text{with}\qquad
\text{$w'$ minimal in $w W_\bbf$,\quad $u \in W_\bbf$.}
\]
If $w$ is not regular with respect to $\bbf$, then
\[
N_w \uH_{w_\bbf} = 0.
\]
If $w$ is regular with respect to $\bbf$, then for all $y \in W_\bbf$, $w'y$ lies in $\spWext$, and
\[
N_w \uH_{w_\bbf} = v^{-\length(u)}\sum_{y \in W_\bbf} v^{\length(w_\bbf) - \length(y)} N_{w'y}.
\]
\end{lem}
\begin{proof}
If $w$ is not regular, this is immediate from Lemma~\ref{lem:nw-irreg}.  We assume for the rest of the proof that $w$ is regular (and thus minimal-regular) with respect to $\bbf$.  Lemma~\ref{lem:regular-left} says that all elements of $wW_\bbf = w'W_\bbf$ are minimal-regular with respect to $\bbf$, and hence (a fortiori) that $w'y \in \spWext$ for all $y \in W_\bbf$. 

We have $\length(w) = \length(w') + \length(u)$ and $H_w = H_{w'} H_u$.  Recall from~\eqref{eqn:uh-long} that $H_u \uH_{w_\bbf} = v^{-\length(u)}\uH_{w_\bbf}$, so $H_w \uH_{w_\bbf} = v^{-\length(u)}H_{w'}  \uH_{w_\bbf}$, and hence
\begin{equation}\label{eqn:nwm-reduce}
N_w \uH_{w_\bbf} = v^{-\length(u)} N_{w'}\uH_{w_\bbf}.
\end{equation}
Next, since $w'$ is minimal in $w'W_\bbf$, we have $H_{w'y} = H_{w'}H_y$ for all $y \in W_\bbf$, so
\begin{equation}\label{eqn:nwm-reduce2}
H_{w'}\uH_{w_\bbf} = \sum_{y \in W_\bbf} v^{\length(w_\bbf) - \length(y)} H_{w'y}.
\end{equation}
Applying $\pi_\cN$ to~\eqref{eqn:nwm-reduce2} and combining with~\eqref{eqn:nwm-reduce}, we obtain the desired formula for $N_w \uH_{w_\bbf}$.
\end{proof}

Next, given a facet $\bbf \subset \bba$, let
\[
\cN^\bbf = \cN \uH_{w_\bbf} = \Z[v,v^{-1}]_\sgn \otimes_{\cH_\bo} (\cH_\ext\uH_{w_\bbf}).
\]
For $w \in \spWext^\bbf$, define
\begin{equation}\label{eqn:nfw-defn}
N^\bbf_w = \sum_{u \in wW_\bbf} v^{\length(w) - \length(u)} N_u = N_{ww_\bbf} \uH_{w_\bbf}.
\end{equation}
Lemma~\ref{lem:nw-mult} implies that $\{ N^\bbf_w : w \in \spWext^\bbf \}$ is a $\Z[v,v^{-1}]$-basis for $\cN^\bbf$.  We define $\cN^\bbf_\aff$ to be the submodule of $\cN^\bbf$ spanned by $\{ N_w^\bbf : w \in \spWaff^\bbf \}$.

More generally, suppose $w \in \spWext$ is minimal-regular with respect to $\bbf$, and write it as $w = w'u$ as in Lemma~\ref{lem:nw-mult}.  Then $w'w_\bbf$ is the maximal element of the coset $w W_\bbf$, i.e., $w' w_\bbf \in \spWext^\bbf$.  Lemma~\ref{lem:nw-mult} implies that
\begin{equation}\label{eqn:nfw-general}
N^\bbf_{w'w_\bbf} = v^{\length(u)} N_w \uH_{w_\bbf}.
\end{equation}

\begin{lem}\label{lem:theta-iw}
There is a surjective map of abelian groups
\[
\theta^\IW: \bigoplus_{\lambda \in -p\bba \cap \bX^+} \cN^{\bbf_\lambda}|_{v = 1}  \to \Z[X]^W
\qquad\text{given by}\qquad
N_w^{\bbf_\lambda} \mapsto \chi(w \boxl \lambda - \varsigma).
\]
This map restricts to an isomorphism of abelian groups
\[
\theta^\IW_\aff: \bigoplus_{\lambda \in -p\bba \cap \bX^+} \cN_\aff^{\bbf_\lambda}|_{v = 1}  \to \Z[X]^W
\]

In addition, given $\lambda \in -p\bba \cap \bX^+$, for arbitrary $w \in \Wext$, regard $\pi_\cN(H_w)\uH_{w_{\bbf_\lambda}}$ as an element of $\cN^{\bbf_\lambda}$.  We have
\[
\theta^\IW((\pi_\cN(H_w)\uH_{w_{\bbf_\lambda}})|_{v=1}) =
\begin{cases}
\chi(w \boxl \lambda - \varsigma) & \text{if $w$ is  regular with respect to $\bbf_\lambda$,} \\
0 & \text{otherwise.}
\end{cases}
\]
\end{lem}
\begin{proof}
The proof that $\theta^\IW_\aff$ is an isomorphism (and hence that $\theta^\IW$ is surjective) is similar to that of Lemma~\ref{lem:theta-sph}, using the bijection~\eqref{eqn:xpp-spwaff-bij} combined with the bijection
\[
\bX^{++} \simto \bX^+
\qquad\text{given by}\qquad
\mu \mapsto \mu - \varsigma.
\]

Let us now compute $\theta^\IW(\pi_\cN(H_w)\uH_{w_\bbf})$.  Write $w = yw'$ where $w'$ is minimal in $Ww$, and $y \in W$.  Then $H_w = H_y H_{w'}$, and $\pi_\cN(H_w) = (-v)^{\length(y)} N_{w'}$.  By Lemma~\ref{lem:regular-coset}, $w'$ is regular with respect to $\bbf_\lambda$ if and only if $w$ is.  By Lemma~\ref{lem:nw-mult}, if $w$ is \emph{not} regular, we have
\[
\pi_\cN(H_w)\uH_{w_{\bbf_\lambda}} = (-v)^{\length(y)} N_{w'} \uH_{w_{\bbf_\lambda}} = 0.
\]
On the other hand, if $w$ is regular, then $w'$ is minimal-regular with respect to $\bbf_\lambda$.  Write $w' = w'' u$ with $w''$ minimal in $w'W_\bbf$ and $u \in W_\bbf$.  We have $w''w_{\bbf_\lambda} \in \spWext^{\bbf_\lambda}$, and by~\eqref{eqn:nfw-general},  $N^{\bbf_\lambda}_{w''w_\bbf} = v^{\length(u)} N_{w'} \uH_{w_{\bbf_\lambda}}$, and hence
\[
\pi_\cN(H_w)\uH_{w_{\bbf_\lambda}} = (-v)^{\length(y)} v^{-\length(u)} N^{\bbf_\lambda}_{w''w_\bbf}.
\]
Specializing at $v = 1$ and then applying $\theta^\IW$, we obtain
\[
(-1)^{\length(y)} \chi( w''w_\bbf \boxl \lambda - \varsigma).
\]
Observe that $w'' w_\bbf \boxl \lambda = w' \boxl \lambda$.  Recall also that for any $y \in W$ and any $\mu \in \bX^+$, we have $\chi(y \dotl \mu) = (-1)^{\length(y)} \chi(\mu)$: see~\cite[\S II.5.9(1)]{jan:rag}.  With these observations, we rewrite the above as
\[
\chi(y \dotl (w' \boxl \lambda - \varsigma)) = \chi((yw') \boxl \lambda - \varsigma) = \chi(w \boxl \lambda - \varsigma). \qedhere
\]
\end{proof}

\subsection{The Riche--Williamson map}

Let $\phi: \cM \to \cN$ be the map given by
\[
\phi(M_w) = \uN_{t_\varsigma} H_{w_\bo w} \qquad\text{for $w \in \oWext$.}
\]
We can extend this to a map $\cH_\ext \to \cN$ in two ways, by regarding $\cM$ as either a quotient (see Remark~\ref{rmk:m-alt-defn}) or a subset of $\cH_\ext$.  The former perspective is the one taken in~\cite{rw:scf}: in that paper, they define a map of right $\cH_\ext$-modules
\[
\tilde\phi: \cH_\ext \to \cN
\qquad\text{by}\qquad
\tilde\phi(h) = \uN_{t_\varsigma} \cdot h.
\]
According to~\cite[\S3.3]{rw:scf}, $\tilde\phi$ is equal to the composition
\[
\cH_\ext \twoheadrightarrow \Z[v,v^{-1}]_\triv \otimes_{\cH_\bo} \cH_\ext \xrightarrow[\sim]{\text{Remark~\ref{rmk:m-alt-defn}}} \cM \xrightarrow{\phi} \cN.
\]

On the other hand, we can also define
\[
\hat\phi: \cH_\ext \to \cN
\qquad\text{by}\qquad
\hat\phi(h) = N_{t_\varsigma w_\bo} \cdot h.
\]
Lemma~\ref{lem:un-sigma} implies that $\phi$ is equal to the composition
\[
\cM = \uH_{w_\bo} \cH_\ext \hookrightarrow \cH_\ext \xrightarrow{\hat\phi} \cN.
\]

Now let $\bbf \subset \bba$ be a facet, and let $w \in \oWext^\bbf$.  The formula for $M_w^\bbf$ in Lemma~\ref{lem:mfw-formula} shows that $\phi(M_w^\bbf) = N_{t\varsigma w_\bo} M_w^\bbf$ lies in $\cN \uH_{w_\bbf} = \cN^\bbf$, so $\phi$ restricts to a map
\[
\phi: \cM^\bbf \to \cN^\bbf.
\]

\begin{lem}\label{lem:theta-compare}
The following diagram commutes:
\[
\begin{tikzcd}
\displaystyle\bigoplus_{\lambda \in -p\bba \cap \bX^+} \cM^{\bbf_\lambda}|_{v = 1} \ar[r, "\theta^\sph"] \ar[d, "\phi|_{v=1}"'] &
\Z[\bX]^W \ar[d, "\text{\normalfont multiply by $\chi((p-1)\varsigma)$}"] \\
\displaystyle\bigoplus_{\lambda \in -p\bba \cap \bX^+} \cN^{\bbf_\lambda}|_{v = 1} \ar[r, "\theta^\IW"] &
\Z[\bX]^W
\end{tikzcd}
\]
\end{lem}
\begin{proof}
Let $\lambda \in -p\bba \cap \bX^+$, and let $w \in \oWext^{\bbf_\lambda}$.  Using Lemma~\ref{lem:mfw-formula} and the formula for $\hat\phi$, we obtain
\begin{align*}
\phi(M_w^{\bbf_\lambda}) &= \hat\phi\Big(\Big( \sum_{y \in \Wext^{\bbf_\lambda} \cap W w W_\bbf} v^{\length(w) - \length(y)} H_{yw_{\bbf_\lambda}} \Big) \uH_{w_{\bbf_\lambda}}\Big) \\
&= \Big( \sum_{y \in \Wext^{\bbf_\lambda} \cap W w W_{\bbf_\lambda}} v^{\length(w) - \length(y)} N_{t_\varsigma w_\bo} H_{yw_{\bbf_\lambda}} \Big) \uH_{w_{\bbf_\lambda}}.
\end{align*}
Now specialize at $v = 1$. We have
\[
N_{t_\varsigma w_\bo} H_{yw_{\bbf_\lambda}}|_{v=1} = 
\pi_\cN(H_{t_\varsigma w_\bo} H_{yw_{\bbf_\lambda}})|_{v=1} =
\pi_\cN(H_{t_\varsigma w_\bo y w_{\bbf_\lambda}})|_{v=1}.
\]
By Lemma~\ref{lem:theta-iw}, we obtain
\begin{equation}\label{eqn:theta-phi1}
\theta^\IW(\phi(M_w^{\bbf_\lambda})) = \sum_{y \in \Wext^{\bbf_\lambda}  \cap W w W_{\bbf_\lambda}} \chi(t_\varsigma w_\bo y w_{\bbf_\lambda} \boxl \lambda - \varsigma)
\end{equation}
Since $w_{\bbf_\lambda} \boxl \lambda = \lambda$, we can omit the $w_{\bbf_\lambda}$ from this expression.  Next, let us determine the possible values of $y \boxl \lambda$.  For any $u \in WwW_{\bbf_\lambda}$, we have $u \boxl \lambda \in W\cdot(w \boxl \lambda)$, and there is a bijection
\[
\{ \text{left $W_{\bbf_\lambda}$-cosets contained in $WwW_{\bbf_\lambda}$} \} \overset{\sim}{\longleftrightarrow} W\cdot(w \boxl \lambda).
\]
Since the sum in~\eqref{eqn:theta-phi1} involves exactly one element $y$ from each left $W_{\bbf_\lambda}$-coset contained in $W w W_{\bbf_\lambda}$, we can rewrite it as
\[
\theta^\IW(\phi(M_w^{\bbf_\lambda})) = \sum_{\mu \in W \cdot (w \boxl \lambda)} \chi(t_\varsigma w_\bo \boxl \mu - \varsigma).
\]
Since the action of $w_\bo$ just permutes the $W$-orbit $W \cdot (w \boxl \lambda)$, we can omit it, and we obtain
\[
\theta^\IW(\phi(M_w^{\bbf_\lambda})) = \sum_{\mu \in W \cdot (w \boxl \lambda)}
\chi(t_\varsigma \boxl \mu - \varsigma)
=\sum_{\mu \in W \cdot (w \boxl \lambda)} \chi((p -1) \varsigma + \mu).
\]
By Brauer's formula~\cite[Lemma~II.5.8(b)]{jan:rag}, the last expression is equal to $\chi((p -1) \varsigma) s(w_\bo(w \boxl \lambda))$.
\end{proof}

\section{Affine Grassmannians, affine flag varieties}
\label{sec:affinegr}

\subsection{Loop groups and affine flag varieties}

Let $\Gv$ be the Langlands dual group to $G$ over an algebraically closed field $\F$ of characteristic $\ell$, where $\ell > 0$ and $\ell \ne p$.  Let $\Tv$ be a maximal torus, and let $\Bv, \Bv^+ \subset \Gv$ be a pair of opposite Borel subgroups containing $\Tv$.  Identify the root system of $\Gv$ with $\fRv$ in such a way that roots of the $\Tv$-action on $\Lie(\Bv^+)$ are those in $\fRv^+$.

Let $\Loop \Gv$ and $\Loop^+\Gv$ be the loop group and arc group of $\Gv$, respectively.  For an $\F$-algebra $R$, they are defined by
\[
\Loop \Gv(R) = \Gv(R((z))),
\qquad
\Loop^+\Gv(R) = \Gv(R[[z]]),
\]
where $z$ is a formal variable.  For each facet $\bbf \subset \bba$, we have a parahoric subgroup
\[
\para_\bbf \subset \Loop \Gv.
\]
As a special case, the parahoric subgroup associated to $\ba$ itself is an \emph{Iwahori subgroup} that will be denoted by
\[
\Iw = \para_\ba.
\]
There is also an opposite Iwahori subgroup $\Iw^+ \subset \Loop^+\Gv$.
At the other extreme, for the singleton facet $\bo$ we have
\[
\Loop^+\Gv = \para_{\bo}.
\]

For each facet $\bbf$ we have a partial affine flag variety $\Fl_\bbf$ defined by
\[
\Fl_\bbf = \Loop\Gv/\para_\bbf.
\]
In the special case $\bbf = \ba$, we obtain the \emph{full} affine flag variety, which we denote by
\[
\Fl = \Loop\Gv/\Iw.
\]
At the other extreme, for $\bbf = \bo$, we obtain the \emph{affine Grassmannian}, denoted by
\[
\Gr = \Loop \Gv / \Loop^+\Gv.
\]
For any $\bbf \subset \bba$, we have $\Iw \subset \para_\bbf$, so there is a natural quotient map
\begin{equation}\label{eqn:pif-defn}
\pi_\bbf: \Fl \to \Fl_\bbf.
\end{equation}
This is smooth morphism of relative dimension $\length(w_\bbf) = \dim \para_\bbf/\Iw$.

For $\nu \in \bX$, let $\bz^\nu$ be corresponding $\F$-point of $\Loop\Gv$, i.e., the $\F$-point obtained as the composition
\[
\Spec \F((z)) \to \Spec \F[z,z^{-1}] = \Gm \xrightarrow{\nu} \Tv \hookrightarrow \Gv,
\]
and then let
\[
\rL_\lambda = \bz^\lambda \Loop^+\Gv \quad \in \quad \Gr.
\]

\subsection{Loop rotation and twisted loop rotation}

The group $\Gm$ acts on $\F((z))$ by scaling the formal variable $z$: that is, for $t \in \Gm$ and $f(z) \in \F((z))$, we set
\begin{equation}\label{eqn:looprot-defn}
t \cdot f(z) = f(tz).
\end{equation}
This induces an action of $\Gm$ (called the \emph{loop rotation} action) on $\Loop\Gv$ and on its parahoric subgroups, so we can form the semidirect products $\Gm \ltimes \Loop\Gv$ and $\Gm \ltimes \para_\bbf$.  For $t \in \Gm$ and $\mu \in \bX$, we have the following equalities in $\Gm \ltimes \Loop\Gv$:
\begin{equation}\label{eqn:tz-conj}
t \bz^\mu t^{-1} = \mu(t)\bz^\mu = \bz^\mu \mu(t).
\end{equation}

Next, we can identify
\[
\Fl_\bbf = (\Gm \ltimes \Loop\Gv)/(\Gm \ltimes \para_\bbf).
\]
The action of $\Gm \subset \Gm \ltimes \Loop\Gv$ on $\Fl_\bbf$ by left multiplication is again called the \emph{loop rotation} action, or sometimes the \emph{untwisted loop rotation} action.  

Now let $\nu \in \bX$, regarded as a map $\nu: \Gm \to \Tv$, and let
\begin{equation}\label{eqn:hatnu-defn}
\hat\nu = (\id,\nu): \Gm \to \Gm \times \Tv.
\end{equation}
In an abuse of notation, we sometimes treat $\hat \nu$ as a map from $\Gm$ to $\Gm \times \Gv$, or to $\Gm \ltimes \Loop^+\Gv$, or to $\Gm \ltimes \Iw$.  

Then one can consider the $\Gm$-action on $\Fl_\bbf$ by composing $\hat\nu$ with the left-multiplication action of $\Gm \times \Tv \subset \Gm \ltimes \Loop^+\Gv$.  This action is called the \emph{$\nu$-twisted loop rotation} action.   Of course, if $\nu = 0$, this coincides with the untwisted loop rotation action.

In the special case of $\Gr$, for all $\mu, \nu \in \bX$, the point $L_\mu \in \Gr$ is fixed by the $\nu$-twisted loop rotation action.

The image of $\hat\nu$ in $\Loop^+\Gv$ normalizes $\Iwu$ and $\Iwup$.  We define
\[
\Gm \ltimes_{\hat\nu} \Iwu = (\text{image of $\hat \nu$}) \ltimes \Iwu,
\]
and likewise for $\Gm \ltimes_{\hat\nu} \Iwup$.

\subsection{Orbits}

The groups $\Loop^+\Gv$, $\Iw$, and $\Iw^+$ all act on $\Fl_\bbf$, and there are bijections
\begin{equation}\label{eqn:fl-orbits}
\begin{aligned}
\{\text{$\Loop^+\Gv$-orbits on $\Fl_\bbf$}\} &\overset{\sim}{\longleftrightarrow} \oWext^\bbf, \\
\{\text{$\Iw$-orbits on $\Fl_\bbf$}\} &\overset{\sim}{\longleftrightarrow} \Wext^\bbf, \\
\{\text{$\Iw^+$-orbits on $\Fl_\bbf$}\} &\overset{\sim}{\longleftrightarrow} \Wext^\bbf.
\end{aligned}
\end{equation}
We will sometimes refer to $\Loop^+\Gv$-orbits as \emph{spherical orbits}\footnote{In the literature, the term ``spherical'' is often reserved for the case of $\Loop^+\Gv$-orbits on $\Gr$, but in this paper, it will be convenient to use it in this more general setting.}.  Given a element $w$ in $\oWext^\bbf$ or in $\Wext^\bbf$ (as appropriate), we denote by $\Fl^\sph_{\bbf,w}$, resp.~$\Fl_{\bbf,w}$, resp.~$\Fl^+_{\bbf,w}$ the corresponding $\Loop^+\Gv$-orbit, resp.~$\Iw$-orbit, resp.~$\Iw^+$-orbit, and we write
\[
j^\sph_w: \Fl^\sph_{\bbf,w} \hookrightarrow \Fl_\bbf,
\qquad
j_w: \Fl_{\bbf,w} \hookrightarrow \Fl_\bbf,
\qquad
j^+_w: \Fl^+_{\bbf,w} \hookrightarrow \Fl_\bbf
\]
for the inclusion maps.

We now record some variations and special cases of these bijections:

\subsubsection{Orbits in identity components}
The ind-scheme $\Fl_\bbf$ may be disconnected; let $\Fl^\circ_\bbf$ denote the connected component containing the coset of the identity element.  To classify the $\Loop^+\Gv$- or $\Iw^+$-orbits contained in $\Fl^\circ_\bbf$, replace ``$\Wext$'' in~\eqref{eqn:fl-orbits} with ``$\Waff$.''

\subsubsection{Orbits and loop rotation}
The classification of $\Gm \ltimes \Loop^+\Gv$-orbits (resp.~$\Gm \ltimes \Iw$-orbits, $\Gm \ltimes \Iw^+$-orbits) on $\Fl_\bbf$ is identical to that of $\Loop^+\Gv$-orbits (resp.~$\Iw$-orbits, $\Iw^+$-orbits), because the latter orbits each contain a $\Gm$-fixed point.

Moreover, for any $\nu \in \bX$, the $\nu$-twisted loop rotation action of $\Gv$ also preserves each orbit for $\Loop^+\Gv$, $\Iw$, or $\Iw^+$.

\subsubsection{Orbits of a prounipotent radical}
Recall that $\Iw^+$ is the preimage of $\Bv^+$ under the evaluation map $\ev_0: \Loop^+\Gv \to \Gv$ that sends $z \mapsto 0$.  Let $\Uv^+$ be the unipotent radical of $\Bv^+$, and let $\Iwup = \ev_0^{-1}(\Uv^+)$.  It can be shown that each $\Iw^+$-orbit is also an $\Iwup$-orbit.

\subsubsection{Orbits in the affine Grassmannian}
\label{sss:orbits-gr}
In the special case where $\bbf = \bo$ and $\Fl_\bbf = \Gr$, using~\eqref{eqn:x-wext}, we obtain
\[
\begin{aligned}
\{\text{$\Loop^+\Gv$-orbits on $\Gr$}\} &\overset{\sim}{\longleftrightarrow} \bX^+, \\
\{\text{$\Iw^+$-orbits on $\Gr$}\} &\overset{\sim}{\longleftrightarrow} \bX.
\end{aligned}
\]
We thus usually denote spherical orbits on $\Gr$ as $\Gr_\lambda$ ($\lambda \in \bX^+$) and $\Iw^+$-orbits as $\Gr^+_\lambda$ ($\lambda \in \bX$).  Concretely, $\Gr_\lambda = \Loop^+\Gv \cdot \rL_\lambda$ and $\Gr^+_\lambda = \Iw^+ \cdot \rL_\lambda$.

\subsection{Derived categories of sheaves}
\label{ss:derivedsheaves}

Since $\bk$ has characteristic different from that of $\F$, it makes sense to consider \'etale $\bk$-sheaves on (ind-)schemes over $\F$.  We denote by
\[
\Db_{\Gm \ltimes \Iwu}(\Fl_\bbf,\bk),
\qquad
\Db_{\Gm \ltimes_{\hat\nu} \Iwu}(\Fl_\bbf,\bk),
\qquad
\Db_{\Gm \ltimes \Loop^+\Gv}(\Fl_\bbf,\bk)
\]
the equivariant derived categories of $\Fl_\bbf$ (in the sense of~\cite{bl:esf}) with respect to left multiplication actions of $\Gm \ltimes \Iwu$, $\Gm \ltimes_{\hat\nu} \Iwu$, and $\Gm \ltimes \Loop^+\Gv$, respectively.  (In the first and third cases, the $\Gm$ acts by untwisted loop rotation, whereas in the second case, it acts by $\nu$-twisted loop rotation for some $\nu \in \bX$.)

We say that a constructible complex $\cF$ of \'etale $\bk$-sheaves on $\Fl_\bbf$ is \emph{spherical} if for every $\Loop^+\Gv$-orbit $O$ in $\Fl_\bbf$, the restriction $\cF|_O$ has the property that 
\[
\text{$\mathcal{H}^i(\cF|_O)$ is a constant sheaf for all $i \in \Z$.}
\]
All objects of $\Db_{\Gm \ltimes \Loop^+\Gv}(\Fl_\bbf,\bk)$ are automatically spherical.  We denote by
\[
\Db_{\Gm,\sph}(\Fl_\bbf,\bk),
\qquad\text{resp.}\qquad
\Db_{\Gm,\sph|\nu}(\Fl_\bbf,\bk),
\]
the full subcategory of $\Db_{\Gm \ltimes \Iwu}(\Fl_\bbf,\bk)$, resp.~$\Db_{\Gm \ltimes_{\hat\nu} \Iwu}(\Fl_\bbf,\bk)$ consisting of spherical objects.

\subsection{Iwahori--Whittaker sheaves}
\label{ss:whittaker}

For each simple coroot $\check\alpha \in \fRvs$, let $\Uv_{\check\alpha}$ be the corresponding root subgroup of $\Gv$, and fix an isomorphism $\phi_{\check\alpha}: \Uv_{\check\alpha} \simto \Ga$.  Recall that the quotient $\Uv^+/[\Uv^+,\Uv^+]$ can be identified with the product of all the $\Uv_{\check\alpha}$.  Let
\begin{equation}\label{eqn:chi-defn}
\chi: \Iwup \to \Ga
\end{equation}
denote the following composition:
\[
\Iwup \xrightarrow{\ev_0} \Uv^+ \twoheadrightarrow \Uv^+/[\Uv^+,\Uv^+] \cong \prod_{\check\alpha \in \fRvs} \Uv_{\check\alpha} \xrightarrow[\sim]{\prod \phi_{\check\alpha}} \Ga^{|\fRvs|} \xrightarrow{\sum} \Ga.,
\]
If we let $\Gm$ act on $\Iwup$ by loop rotation and trivially on all other varieties above, then these maps (and their composition $\chi$) are $\Gm$-equivariant.

Let $\AS$ denote an Artin--Schreier $\bk$-local system on $\Ga$.  We regard it as $\Gm$-equivariant (for the trivial $\Gm$-action).  Then $\chi^*\AS$ is $\Gm$-equivariant for the loop rotation action.  We denote by
\[
\Db_{\Gm,\IW}(\Fl_\bbf,\bk)
\]
the bounded constructible $(\Gm \ltimes \Iwup, \chi^*\AS)$-equivariant category of $\bk$-sheaves on $\Fl_\bbf$.  For more details on the definition of these categories, see~\cite[Appendix~A]{ar}.  Objects in these categories are called ($\Gm$-equivariant) \emph{Iwahori--Whittaker complexes}.  

In contrast with the spherical case, not every $\Iwup$-orbit admits a nonzero Iwa\-ho\-ri--Whittaker local system.  Let $\bbf \subset \bba$ be a facet, and consider an element $w \in \Wext^\bbf$.  The orbit $\Fl^+_{\bbf,w}$ admits a nonzero Iwahori--Whittaker local system if and only there exists an $\Iwup$-equivariant map
\[
\Fl^+_{\bbf,w} \to \Ga,
\]
where $\Iwup$ acts on $\Ga$ via $\chi$. Such a map exists if and only if the stabilizer in $\Iwup$ of any point in $\Fl^+_{\bbf,w}$ is contained in the kernel of $\chi: \Iwup \to \Ga$.  In combinatorial terms, we have
\[
\begin{array}{c}
\text{$\Fl^+_{\bbf,w}$ admits a nonzero}\\
\text{Iwahori--Whittaker local system}
\end{array}
\qquad\Longleftrightarrow\qquad
w \in \spWext^\bbf.
\]
In the special case where $\bbf = \bo$, this condition becomes:
\[
\begin{array}{c}
\text{$\Gr^+_\lambda$ admits a nonzero}\\
\text{Iwahori--Whittaker local system}
\end{array}
\qquad\Longleftrightarrow\qquad
\lambda \in \bX^{++}
\]
For $\lambda \in \bX^{++}$, we denote by
\[
\cL^{\lambda}_\AS
\qquad\text{or simply}\qquad
\cL_\AS
\]
the unique irreducible Iwahori--Whittaker local system on $\Gr^+_\lambda$.  Recall that $j^+_\lambda: \Gr^+_\lambda \hookrightarrow \Gr$ denotes the inclusion map.  There is a natural transformation
\begin{equation}\label{eqn:iw-clean-defn}
j^+_{\lambda!}\cL^\lambda_\AS \to j^+_{\lambda*}\cL^\lambda_\AS.
\end{equation}
We say that the weight $\lambda$ is \emph{IW-clean} if~\eqref{eqn:iw-clean-defn} is an isomorphism.  For instance, the weight $\varsigma$ is IW-clean because $\Gr^+_\varsigma$ is minimal (with respect to the closure partial order) among $\Iwup$-orbits admitting a nonzero Iwahori--Whittaker local system (cf.~\cite[Eq.~(3.2)]{bgmrr}).

\subsection{Graded Hom groups and equivariant cohomology}
\label{ss:bhom}

We will assume from this section forward that $p$ is not a torsion prime for $\Gv$ (see~\cite[\S 2.6]{jmw:ps} for a discussion of this condition).  If $\cF$ and $\cG$ are two objects in a triangulated category, we can form the graded vector space
\[
\bHom(\cF,\cG) = \bigoplus_{n \in \Z}\Hom(\cF,\cG[n]).
\]
We will occasionally consider only the even-graded part: we write
\[
\bHomev(\cF, \cG) = \bigoplus_{n \in \Z}\Hom(\cF,\cG[2n]).
\]
We sometimes decorate ``$\bHom$'' with a subscript to indicate the category in which the $\Hom$-groups are to be computed, in case there is some risk of ambiguity.  We also use the notation
\[
\bEnd(\cF) = \bHom(\cF,\cF). 
\]

For instance, let $H$ be an algebraic group over $\F$, and consider the equivariant derived category $\Db_H(\pt,\bk)$.  Then
\[
\bEnd_{\Db_H(\pt,\bk)}(\underline{\bk}_\pt)
\]
is identified with the equivariant cohomology ring $\coh^\bullet_H(\pt,\bk)$. 

More generally, if $X$ is a variety over $\F$ equipped with an $H$-action, then for $\cF, \cG \in \Db_H(X,\bk)$, the graded vector space $\bHom(\cF,\cG)$ can naturally be regarded as a graded $\coh^\bullet_H(\pt,\bk)$-module.  Concretely, if $h \in \coh^i_H(\pt,\bk) = \Hom(\underline{\bk}_\pt, \underline{\bk}_\pt[i])$ and $f \in \Hom(\cF,\cG[j])$, then the product $hf$ is given by
\begin{equation}\label{eqn:eqcoh-module}
hf = a^*h \otimes^L f
\end{equation}
where $a: X \to \pt$ is the structure map of $X$.

The most commonly occurring equivariant cohomology rings in this paper are the following:
\[
\coh^\bullet_\Gm(\pt,\bk),
\qquad
\coh^\bullet_{\Tv}(\pt,\bk),
\qquad
\coh^\bullet_{\Gv}(\pt,\bk).
\]
Recall that $\coh^\bullet_\Gm(\pt,\bk)$ is the symmetric algebra on $\coh^2_\Gm(\pt,\bk)$, which is a $1$-di\-men\-sion\-al vector space.  (This $1$-di\-men\-sion\-al vector space is canonically identified with $\bk(-1)$, the dual of the Tate module: see Section~\ref{ss:pi-prelim}.)  Similarly, $\coh^\bullet_{\Tv}(\pt,\bk)$ is the symmetric algebra on $\coh^2_{\Tv}(\pt,\bk)$, and we have
\[
\coh^2_{\Tv}(\pt,\bk) \cong \bXv \otimes_\bk \bk(-1).
\]
The $W$-action on $\bXv$ induces a $W$-action on $\coh^\bullet_{\Tv}(\pt,\bk)$, and we denote by $\coh^\bullet_{\Tv}(\pt,\bk)^W$ the subring of $W$-invariant elements in this ring.

Finally, the inclusion map $\Tv \hookrightarrow \Gv$ induces a map of equivariant cohomology rings
\begin{equation}\label{eqn:cohgt-incl}
\coh^\bullet_{\Gv}(\pt,\bk) \to \coh^\bullet_{\Tv}(\pt,\bk).
\end{equation}
Since $p$ is not a torsion prime for $\Gv$, the map above is injective and yields an isomorphism
\[
\coh^\bullet_{\Gv}(\pt,\bk) \to \coh^\bullet_{\Tv}(\pt,\bk)^W.
\]
A related fact is that there is a canonical isomorphism
\begin{equation}\label{eqn:gmgv-coh}
\coh^\bullet_{\Gm \times \Gv}(\pt,\bk) \cong \coh^\bullet_\Gm(\pt,\bk) \otimes \coh^\bullet_{\Gv}(\pt,\bk).
\end{equation}

Let $\nu \in \bX$.  The group homomorphisms
\begin{align*}
\nu &: \Gm \to \Tv, \\
\hat\nu &: \Gm \to \Gm \times \Gv
\end{align*}
induce ring homomorphisms in equivariant cohomology, which we denote by
\begin{align}
\rS(\nu) &: \coh^\bullet_{\Tv}(\pt,\bk) \to \coh^\bullet_{\Gm}(\pt,\bk), \notag \\
\rS(\hat\nu) &: \coh^\bullet_{\Gm \times \Gv}(\pt,\bk) \to \coh^\bullet_{\Gm}(\pt,\bk). \label{eqn:s-hatnu-defn}
\end{align}
Concretely, $\rS(\nu)$ is determined by its restriction to $\coh^2_{\Tv}(\pt,\bk) \to \coh^2_\Gm(\pt,\bk)$, and this restriction is identified with $\la\nu,{-}\ra: \bXv(-1) \to \bk(-1)$.  The map $\rS(\hat\nu)$ can be described in similar terms.  We will write
\[
\coh^\bullet_{\Gm}(\pt,\bk)_{\rS(\nu)}
\]
to mean $\coh^\bullet_\Gm(\pt,\bk)$ regarded as a $(\coh^\bullet_{\Gm}(\pt,\bk), \coh^\bullet_{\Tv}(\pt,\bk))$-bimodule, where $\coh^\bullet_{\Tv}(\pt,\bk)$ acts by multiplication on the right through the homomorphism $\rS(\nu)$.  Via~\eqref{eqn:cohgt-incl}, we may also regard $\coh^\bullet_{\Gm}(\pt,\bk)_{\rS(\nu)}$ as a $(\coh^\bullet_{\Gm}(\pt,\bk), \coh^\bullet_{\Gv}(\pt,\bk))$-bimodule.  The latter perspective is equivalent to treating $\coh^\bullet_\Gm(\pt,\bk)$ as a $\coh^\bullet_{\Gm \times \Gv}(\pt,\bk)$-module via $\rS(\hat\nu)$.

\subsection{Parity sheaves}
\label{ss:parity}

In each of the derived categories introduced in the preceding sections, there is an additive category of parity sheaves (following~\cite{jmw:ps}).  These categories are denoted by:
\begin{gather*}
\Parity_{\Gm \ltimes \Iwu}(\Fl_\bbf,\bk), 
\qquad
\Parity_{\Gm \ltimes_{\hat\nu} \Iwu}(\Fl_\bbf,\bk), 
\qquad
\Parity_{\Gm \ltimes \Loop^+\Gv}(\Fl_\bbf,\bk), 
\\
\Parity_{\Gm,\sph}(\Fl_\bbf,\bk), 
\qquad
\Parity_{\Gm,\sph|\nu}(\Fl_\bbf,\bk), 
\qquad
\Parity_{\Gm,\IW}(\Fl_\bbf,\bk).
\end{gather*}
For $w \in \Wext^\bbf$, we let
\[
\cE^\bbf_w \in \Parity_{\Gm \ltimes \Iwu}(\Fl_\bbf,\bk)
\]
denote the unique (up to isomorphism) indecomposable parity sheaf whose support is $\overline{\Fl_{\bbf,w}}$, and which satisfies
\[
\cE^\bbf_w|_{\Fl_{\bbf,w}} \cong \underline{\bk}_{\Fl_{\bbf,w}}[\dim \Fl_{\bbf,w}].
\]
In the special case where $\bbf = \ba$, we typically omit the ``$\bbf$'' from the notation and simply write $\cE_w$.

We may also write $\cE^\bbf_w$ for the object of $\Parity_{\Gm \ltimes_{\hat\nu} \Iwu}(\Fl_\bbf,\bk)$ with the same characterization.  If $w$ belongs to $\oWext^\bbf$, then we may regard $\cE^\bbf_w$ as an object of $\Parity_{\Gm \ltimes \Loop^+\Gv}(\Fl_\bbf,\bk)$, $\Parity_{\Gm,\sph}(\Fl_\bbf,\bk)$, or $\Parity_{\Gm,\sph|\nu}(\Fl_\bbf,\bk)$.

Similarly, for $w \in \spWext^\bbf$, we let
\[
\cE^{\IW,\bbf}_w \in \Parity_{\Gm,\IW}(\Fl_\bbf,\bk)
\]
denote the indecomposable parity sheaf supported on $\overline{\Fl^+_{\bbf,w}}$ and satisfying
\[
\cE^{\IW,\bbf}_w|_{\Fl^+_{\bbf,w}} \cong \cL^w_\AS[\dim \Fl_{\bbf,w}].
\]

A key property of equivariant parity sheaves is that their $\bHom$-groups are free modules over the appropriate equivariant cohomology ring.  For instance, if $\cF, \cG \in \Parity_{\Gm \ltimes \Loop^+\Gv}(\Fl_\bbf,\bk)$, then
\[
\text{$\bHom(\cF,\cG)$ is a free $\coh^\bullet_{\Gm \times \Gv}(\pt,\bk)$-module.}
\]
Moreover, this module structure interacts well with forgetful functors.  To make this precise, let $\nu \in \bX$, and let
\[
\For_\nu: \Parity_{\Gm \ltimes \Loop^+\Gv}(\Fl_\bbf,\bk) \to \Parity_{\Gm,\sph|\nu}(\Fl_\bbf,\bk)
\]
be the forgetful functor.  For $\cF, \cG \in \Parity_{\Gm \ltimes \Loop^+\Gv}(\Fl_\bbf,\bk)$, we have a natural isomorphism
\[
\bHom(\For_\nu(\cF),\For_\nu(\cG)) \cong
\coh^\bullet_{\Gm}(\pt)_{\rS(\nu)} \otimes_{\coh^\bullet_{\Gm \times \Gv}(\pt,\bk)} \bHom(\cF,\cG).
\]

\subsection{Parity sheaves on the affine Grassmannian}
\label{ss:parity-gr}

We will use a slightly different notation for parity sheaves on $\Fl_\bo = \Gr$.  Recall from Section~\ref{sss:orbits-gr} that spherical orbits on $\Gr$ are in bijection with $\bX^+$.  For $\lambda \in \bX^+$, we denote by
\[
\cE^\sph_\lambda \in \Parity_{\Gm \ltimes \Loop^+\Gv}(\Gr,\bk)
\]
the unique indecomposable parity sheaf that is supported on $\overline{\Gr_\lambda}$ and that satisfies $\cE^\sph_\lambda|_{\Gr_\lambda} \cong \underline{\bk}_{\Gr_\lambda}[\dim \Gr_\lambda]$.  We have
\[
\cE^\sph_\lambda \cong \cE^{\bo}_w
\qquad\text{where}\qquad
w \in \oWext^\bo \cap W t_\lambda W.
\]

Similarly, recall from Section~\ref{ss:whittaker} that $\Gr^+_\lambda$ admits a nonzero Iwahori--Whittaker local system if and only if $\lambda \in \bX^{++}$.  We denote by
\[
\cE^{\IW,\bo}_\lambda = \cE^{\IW,\bo}_w
\qquad\text{where}\qquad
w \in \spWext^\bo \cap t_\lambda W.
\]

\subsection{Scaled versions of loop groups}
\label{ss:scaled}

The construction of loop groups (and their various subgroups and quotients) involve rings of the form $R((z))$ or $R[[z]]$, where $R$ is an $\F$-algebra.  Consider the element $z^p \in R[[z]]$.  We can treat $z^p$ as a formal variable in its own right, and repeat the constructions above in terms of this variable.  For instance, we have the group (ind-)schemes
\[
\Loop_p\Gv(R) = \Gv(R((z^p)))
\qquad
\Loop^+_p\Gv(R) = \Gv(R[[z^p]]).
\]
We similarly define
\[
\para_{p,\bbf},
\quad
\Iw_p,
\quad
\Iwul,
\quad
\Iw^+_p,
\quad
\Iwupl,
\quad
\Fl_{p,\bbf},
\quad
\Gr_p.
\]

The inclusion $R[[z^p]] \hookrightarrow R[[z]]$ lets us regard $\Loop_p\Gv$ as a subgroup of $\Loop\Gv$, and we likewise have $\Loop^+_p\Gv \subset \Loop^+\Gv$ and $\Iwupl \subset \Iwup$.  Let
\[
\chi_p: \Iwupl \to \Ga
\]
be the restriction of the map $\chi: \Iwup \to \Ga$ from~\eqref{eqn:chi-defn}.  

We define the loop rotation action of $\Gm$ on $\F((z^p))$ to be the restriction of the loop rotation action on $\F((z))$.  That is, for $t \in \Gm$ and $f(z^p) \in \F((z^p))$, we have
\[
t \cdot f(z^p) = f(t^p z^p).
\]
(Compare this with~\eqref{eqn:looprot-defn}.)  We can then define all the analogous categories to those in Sections \ref{ss:whittaker}, \ref{ss:derivedsheaves}, and \ref{ss:parity}, including
\[
\Db_{\Gm \ltimes \Loop^+_p\Gv}(\Fl_{p,\bbf},\bk), 
\qquad
\Parity_{\Gm,\IW}(\Fl_{p,\bbf},\bk), 
\qquad
\text{etc.}
\]

We remark that it is only the loop rotation action that distinguishes the ``$z$'' version of our set-up from the ``$z^p$'' version: if we omit the loop rotation, then, for instance, the categories
\[
\Db_{\Loop^+\Gv}(\Fl_\bbf,\bk)
\qquad\text{and}\qquad
\Db_{\Loop^+_p\Gv}(\Fl_{p,\bbf},\bk)
\]
are canonically equivalent via the substitution $z \mapsto z^p$.

\subsection{Fixed points under \texorpdfstring{$\varpi$}{varpi}}
\label{ss:fixed-pts}

Recall that $\varpi \subset \Gm$ denotes the group of $p$-th roots of unity.  Let $\varpi$ act on $\Gr$ via the (untwisted) loop rotation action of $\Gm$.  The fixed-point locus
\[
\Gr^\varpi \subset \Gr
\]
is a closed sub-(ind-)scheme of $\Gr$ that contains all points of the form $\rL_\lambda$.  The group $\Loop_p\Gv$ acts on $\Gr^\varpi$, and for $\lambda \in -p\bba \cap \bX^+$, the stabilizer of $\rL_\lambda$ in $\Loop_p\Gv$ is $\para_{p,\bbf_\lambda}$.  There is a map of ind-schemes
\begin{equation}\label{eqn:grpi-overcount}
\bigsqcup_{\lambda \in -p\bba \cap \bX^+} \Fl_{p,\bbf_\lambda} \to \Gr^\varpi
\qquad\text{given by}\qquad
(g\para_{p,\bbf_\lambda} \in \Fl_{p,\bbf_\lambda}) \mapsto g \cdot \rL_\lambda.
\end{equation}
According to~\cite[Proposition~4.7]{rw:st}, when we take just the identity component of each term in the domain of~\eqref{eqn:grpi-overcount}, we get an isomorphism of ind-schemes
\begin{equation}\label{eqn:grpi-isom}
\bigsqcup_{\lambda \in (-p\bba) \cap \bX^+} \Fl_{p,\bbf_\lambda}^\circ \simto \Gr^\varpi.
\end{equation}
In this decomposition, the component corresponding to $\lambda = 0$ is $\Gr^\circ_{p}$.  The union of components corresponding to $\lambda \in (-p\bba) \cap p \bX^+$ is $\Gr_p$.

The argument in~\cite[Proposition~4.7]{rw:st} shows more generally that the restriction of~\eqref{eqn:grpi-overcount} to \emph{any} connected component of its domain induces an isomorphism with its image, which is a connected component of $\Gr^\varpi$.  However,~\eqref{eqn:grpi-overcount} is not injective in general.

According to~\cite[Lemma~4.4]{rw:st}, each $\Loop^+\Gv$-orbit, resp.~$\Iwup$-orbit, in $\Gr$ meets exactly one $\Loop_p\Gv$-orbit, resp.~$\Iwupl$-orbit, in $\Gr^\varpi$.  There are thus two ways to label such orbits: one labelling inherited from $\Gr$, and the other coming from the left-hand side of~\eqref{eqn:grpi-isom}.  To make this explicit, let $\lambda \in -p\bba \cap \bX^+$. Let $w \in \Waff^\bbf$, and set $\mu = w \boxl \lambda \in \bX$.  Let $\Gr^{\varpi+}_\mu$ be given by
\[
\Gr^{\varpi+}_\mu = \Gr^+_\mu \cap \Gr^\varpi = \Iwupl \cdot \rL_\mu, \qquad\text{corresponding to}\qquad
\Fl^+_{p,\bbf_\lambda,w} \subset \Fl^\circ_{p,\bbf_\lambda},
\]
where ``corresponding to'' should be understood with respect to~\eqref{eqn:grpi-isom}.  Similarly, if $w \in \oWaff^\bbf$, set $\mu' = w_\bo(w \boxl \lambda)$.  By Lemma~\ref{lem:boxl-wext}\eqref{it:boxl-max}, $\mu' \in \bX^+$.  Let $\Gr^{\varpi}_{\mu'}$ be given by
\[
\Gr^{\varpi}_{\mu'} = \Gr_{\mu'} \cap \Gr^\varpi = \Loop_p^+\Gv \cdot \rL_{\mu'}, \qquad\text{corresponding to}\qquad
\Fl^\sph_{p,\bbf_\lambda,w} \subset \Fl^\circ_{p,\bbf_\lambda}.
\]

As in Sections~\ref{ss:parity}--\ref{ss:parity-gr}, we can consider parity sheaves on $\Gr^\varpi$.  We will use only the labelling from the right-hand side of~\eqref{eqn:grpi-isom}: we have
\begin{align*}
\cE^{\varpi,\sph}_\lambda &\in \Parity_{\Gm \times \Loop^+_p\Gv}(\Gr^\varpi,\bk) & \lambda &\in \bX^+, \\
\cE^{\varpi,\IW,\bo}_\lambda &\in \Parity_{\Gm,\IW}(\Gr^\varpi,\bk) & \lambda &\in \bX^{++}.
\end{align*}

\subsection{Fixed points under a twisted loop rotation action}

If we instead consider the $\nu$-twisted loop rotation action for some $\nu \in \bX$, the fixed-point locus of $\varpi \subset \Gm$ may be in general be different from $\Gr^\varpi$.  However, if $\nu$ is divisible by $p$, then $\varpi$ is contained in the kernel of $\nu: \Gm \to \Tv$, and it follows that
\[
\Gr^\varpi = 
\begin{array}{c}
\text{the fixed-point locus for $\varpi \subset \Gm$ acting on $\Gr$}\\
\text{by the $\nu$-twisted loop rotation action for any $\nu \in p\bX$.}
\end{array}
\]

\subsection{Smith categories and Smith--Treumann localization}

We briefly recall the construction of Smith categories and the Smith--Treumann localization functor.  However, some statements in the following subsections will reference more nuanced constructions that are only detailed in the Appendix.

Let $X$ be an (ind-)scheme over $\F$ with an action of $\Gm$.  A $\Gm$-equivariant complex on $X^{\varpi}$, also being $\varpi$-equivariant, has stalks that can be viewed as complexes of $\bk[\varpi]$-modules.  The category
\[
\Db_\Gm(X^{\varpi},\bk)_\perf
\]
denotes the full subcategory of complexes for which these stalks admit finite free resolutions as  $\bk[\varpi]$-modules (such complexes are called \textit{$\varpi$-perfect}).

The \emph{Smith category} of $X^{\varpi}$ is the Verdier quotient
\[
\Sm(X^{\varpi},\bk) = \Db_\Gm(X^{\varpi},\bk)/ \Db_\Gm(X^{\varpi},\bk)_\perf,
\]
and we denote by
\[
\rQ: \Db_\Gm(X^{\varpi},\bk) \to \Sm(X^{\varpi},\bk)
\]
the resulting quotient functor.

Let $i: X^\varpi \hookrightarrow X$ be the inclusion map.  By~\cite[Lemma~3.5]{rw:st}, for any $\cF \in \Db_\Gm(X,\bk)$, the cone of the natural map $i^!\cF \to i^*\cF$ is $\varpi$-perfect.  Thus, this natural map becomes an isomorphism after passage to $\Sm(X^\varpi,\bk)$.  The \emph{Smith--Treumann localization functor} is the functor
\[
\Psm: \Db_\Gm(X,\bk) \to \Sm(X^\varpi,\bk)
\]
arising from the composition of $\rQ$ with either $i^!$ or $i^*$.

\subsection{Smith categories for the affine Grassmannian}

Given a weight $p\nu \in p\bX$, we have the following categories of sheaves on $\Gr^\varpi$:
\[
\Db_{\Gm \ltimes_{p\nu} \Iwul}(\Gr^\varpi,\bk),
\qquad
\Db_{\Gm,\sph|p\nu}(\Gr^\varpi,\bk),
\qquad
\Db_{\Gm,\IW}(\Gr^\varpi,\bk).
\]
For each of these categories, we can take its image under the appropriate quotient functor $\rQ: \Db_\Gm(\Gr^\varpi,\bk) \to \Sm(\Gr^\varpi,\bk)$, and then take the full triangulated category generated by that image.  The resulting categories  are denoted by
\[
\Sm_{\Iwul|p\nu}(\Gr^\varpi,\bk),
\qquad
\Sm_{\sph|p\nu}(\Gr^\varpi,\bk),
\qquad
\Sm_{\IW}(\Gr^\varpi,\bk),
\]
respectively.  Note that $\Sm_{\Iwul|p\nu}(\Gr^\varpi,\bk)$ and $\Sm_{\IW}(\Gr^\varpi,\bk)$ both fit the general framework considered in Section~\ref{ss:tweq}.  In particular, Lemma~\ref{lem:smith-2defns} gives an alternative description of these two Smith categories.  In~\cite{rw:st}, that alternative description was taken as the definition of $\Sm_\IW(\Gr^\varpi,\bk)$ (see~\cite[\S6.1]{rw:st}), so Lemma~\ref{lem:smith-2defns} allows us to apply results from that paper.

In contrast, the category $\Sm_{\sph|p\nu}(\Gr^\varpi,\bk)$ does \emph{not} fit the framework of Section~\ref{ss:tweq}.  However, by construction, it is a full subcategory of $\Sm_{\Iwul|p\nu}(\Gr^\varpi,\bk)$.

The general theory from Section~\ref{ss:smithpar} lets us define Smith parity sheaves in the $\Iwul$-equivariant and Iwahori--Whittaker settings.  These categories are denoted by
\[
\SmParity_{\Iwul|p\nu}(\Gr^\varpi,\bk)
\qquad\text{and}\qquad
\SmParity_{\IW}(\Gr^\varpi,\bk),
\]
respectively.  By Corollary~\ref{cor:parity-full}, there are enough Smith parity sheaves in the $\Iwul$-equivariant and Iwahori--Whittaker settings.  We define
\[
\SmParity_{\sph|p\nu}(\Gr^\varpi,\bk) = \SmParity_{\Iwul|p\nu}(\Gr^\varpi,\bk) \cap \Sm_{\sph|p\nu}(\Gr^\varpi,\bk).
\]

Denote the inclusion map of $\Gr^\varpi$ into $\Gr$ by
\[
\bi: \Gr^\varpi \hookrightarrow \Gr.
\]
The functors $\bi^!$ and $\bi^*$ send spherical, resp.~Iwahori--Whittaker, objects on $\Gr$ to spherical, resp.~Iwahori--Whittaker, objects on $\Gr^\varpi$.  We can therefore consider the spherical and Iwahori--Whittaker versions of the Smith--Treumann localization functor (cf.~Section~\ref{ss:st-defn}) in the following two cases:
\begin{align*}
\Psm_{p\nu} &:\Db_{\Gm,\sph|p\nu}(\Gr,\bk) \to \Sm_{\sph|p\nu}(\Gr^\varpi,\bk), \\
\Psm_{\IW} &: \Db_{\Gm,\IW}(\Gr,\bk) \to \Sm_\IW(\Gr^\varpi,\bk).
\end{align*}
In both cases, the functor is defined to be $\rQ \circ \bi^! \cong \rQ \circ \bi^*$, where $\rQ$ is the quotient functor to the Smith category (see Section~\ref{ss:smith-defn}).  Of course, $\bi^!$ preserves the property of being $!$-even (or $!$-odd), while $\bi^*$ preserves the property of being $*$-even (or $*$-odd).  Since $\rQ$ sends parity objects to parity objects (see Section~\ref{ss:smithpar}), Smith--Treumann localization also sends parity objects to parity objects.  We obtain functors
\begin{align*}
\Psm_{p\nu} &:\Parity_{\Gm,\sph|p\nu}(\Gr,\bk) \to \SmParity_{\sph|p\nu}(\Gr^\varpi,\bk), \\
\Psm_{\IW} &: \Parity_{\Gm,\IW}(\Gr,\bk) \to \SmParity_\IW(\Gr^\varpi,\bk).
\end{align*}

\section{Convolution}
\label{sec:convolution}

Throughout this section, we retain the assumption from Section~\ref{ss:bhom} that $p$ is not a torsion prime for $\Gv$.

\subsection{Review of convolution}

Here is a brief informal review of the ``convolution'' operation for sheaves on $\Gr$ and $\Fl_\bbf$.  Consider the following diagram:
\begin{equation}\label{eqn:conv-diag}
\Gr \times \Fl_\bbf \xleftarrow{p} \Loop\Gv \times \Fl_\bbf \xrightarrow{q} \Loop\Gv \times^{\Loop^+\Gv} \Fl_\bbf \xrightarrow{m} \Fl_\bbf
\end{equation}
Here, $p$ is induced by the quotient map $\Loop\Gv \to \Gr$; $q$ is the quotient for the ``middle action'' of $\Loop^+\Gv$ on $\Loop\Gv \times \Fl_\bbf$ given by $g \cdot (h,x) = (hg^{-1}, gx)$; and $m$ is the multiplication map that sends $[h : x ]$ to $hx$.  

Let us now equip the diagram~\eqref{eqn:conv-diag} with a $\Gm$-action.  Let $\Gm$ act on $\Gr \times \Fl_\bbf$ by acting by loop rotation on both factors, and likewise for $\Loop\Gv \times \Fl_\bbf$.  Then $p$ is $\Gm$-equivariant.  Next, we can combine the $\Gm$-action on $\Loop\Gv \times \Fl_\bbf$ with the ``middle action'' of $\Loop^+\Gv$ to get an action of $\Gm \ltimes \Loop^+\Gv$.  Since the map $q$ in~\eqref{eqn:conv-diag} is the quotient by the action of $\Loop^+\Gv$, the quotient space $\Loop\Gv \times^{\Loop^+\Gv}\Fl_\bbf$ inherits an action of the quotient group $(\Gm \ltimes \Loop^+\Gv)/\Loop^+\Gv = \Gm$.  Finally, the map $m$ is $\Gm$-equivariant with respect to the aforementioned action on $\Loop\Gv \times^{\Loop^+\Gv} \Fl_\bbf$ and the loop rotation action on $\Fl_\bbf$.

Suppose now that we have
\[
\cF \in \Db_\Gm(\Gr,\bk)
\qquad\text{and}\qquad
\cG \in \Db_{\Gm \ltimes \Loop^+\Gv}(\Fl_\bbf,\bk).
\]
Then one can consider the object $\cF \boxtimes \cG \in \Db_\Gm(\Gr \times \Fl_\bbf)$.  Because $\cG$ is assumed to be $\Loop^+\Gv$-equivariant, the pullback $p^*(\cF \boxtimes \cG)$ is $(\Gm \ltimes \Loop^+\Gv)$-equivariant (where $\Loop^+\Gv$ acts by the middle action), and there is a canonical object
\begin{equation}\label{eqn:tboxtimes-defn}
\cF \tboxtimes \cG \in \Db_\Gm(\Loop\Gv \times^{\Loop^+\Gv} \Fl_\bbf,\bk)
\qquad\text{such that}\qquad
q^*(\cF \tboxtimes \cG) \cong p^*(\cF \boxtimes \cG).
\end{equation}
The convolution product $\cF \star \cG$ is defined by
\[
\cF \star \cG = m_*(\cF \tboxtimes \cG).
\]
This construction defines a bifunctor
\[
\star: \Db_\Gm(\Gr,\bk) \times \Db_{\Gm \ltimes \Loop^+\Gv}(\Fl_\bbf,\bk) \to \Db_\Gm(\Fl_\bbf,\bk).
\]
The following lemma follows from the definition of convolution and the definition of the $\coh^\bullet_\Gm(\pt,\bk)$-module structure (see~\eqref{eqn:eqcoh-module}).

\begin{lem}\label{lem:looprot-conv}
Let $\cF_1, \cF_2 \in \Db_\Gm(\Gr,\bk)$ and $\cG_1, \cG_2 \in \Db_{\Gm \ltimes \Loop^+\Gv}(\Fl_\bbf,\bk)$.  The map
\[
\bHom(\cF_1,\cF_2) \times \bHom(\cG_1,\cG_2) \to \bHom(\cF_1 \star \cG_1, \cF_2 \star \cG_2)
\]
is $\coh^\bullet_\Gm(\pt,\bk)$-bilinear.
\end{lem}

Below are some variants of this construction.

\subsubsection{Equivariance properties for $\cF$}
If the object $\cF$ starts out with any of the properties below, then these properties are inherited by $\cF \star \cG$:
\[
\text{$\Loop^+\Gv$-equivariant, \quad
spherical, \quad
Iwahori--Whittaker.}
\]
These properties can optionally be further combined with loop rotation equivariance.  For instance, we get functors
\begin{gather*}
\star: \Db_{\Loop^+\Gv}(\Gr,\bk) \times \Db_{\Loop^+\Gv}(\Fl_\bbf,\bk) \to \Db_{\Loop^+\Gv}(\Fl_\bbf,\bk), \\
\star: \Db_{\Gm,\IW}(\Gr,\bk) \times \Db_{\Gm \ltimes \Loop^+\Gv}(\Fl_\bbf,\bk) \to \Db_{\Gm,\IW}(\Fl_\bbf,\bk).
\end{gather*}

\subsubsection{Parity sheaves}\label{sss:conv-parity}
It is well known that convolution sends parity sheaves to parity sheaves.  In particular, we have functors
\begin{gather*}
\star: \Parity_{\Gm \ltimes \Loop^+\Gv}(\Gr,\bk) \times \Parity_{\Gm \ltimes \Loop^+\Gv}(\Fl_\bbf,\bk) \to \Parity_{\Gm \ltimes \Loop^+\Gv}(\Fl_\bbf,\bk), \\
\star: \Parity_{\Gm,\IW}(\Gr,\bk) \times \Parity_{\Gm \ltimes \Loop^+\Gv}(\Fl_\bbf,\bk) \to \Parity_{\Gm,\IW}(\Fl_\bbf,\bk)
\end{gather*}
The first case goes back to~\cite[Theorem~4.8]{jmw:ps} (at least for sheaves in the analytic topology; for the \'etale setting, see the discussion in~\cite[\S9.4]{rw:tmpcb}).  For the second case, see~\cite[Lemma~4.5]{rw:scf} and~\cite[Lemma~4.14]{bgmrr}.

\subsubsection{Scaled loop groups}
For simplicity, the results in Sections~\ref{ss:conv-wo} and~\ref{ss:indecomp} below are stated for ``ordinary'' loop groups and related varieties, but they also apply to the ``$p$-scaled'' versions from Section~\ref{ss:scaled}.  

However, there is one consideration that is relevant only to the $p$-scaled version: since the subgroup $\varpi \subset \Gm$ acts trivially in the loop rotation action on $\Gr_p$ and $\Fl_{p,\bbf}$, one can speak of $\varpi$-perfect objects on these varieties.  Suppose we have
\[
\cF \in \Db_\Gm(\Gr_p,\bk)
\qquad\text{and}\qquad
\cG \in \Db_{\Gm \ltimes \Loop^+_p\Gv}(\Fl_{p,\bbf},\bk).
\]
If at least one of $\cF$ or $\cG$ is $\varpi$-perfect, Lemma~\ref{lem:sheaf-perfect} tells us that $\cF \boxtimes \cG$ is also $\varpi$-perfect.  

\begin{lem}\label{lem:conv-perfect}
Let $\cF \in \Db_\Gm(\Gr_p,\bk)$ and $\cG \in \Db_{\Gm \ltimes \Loop^+_p\Gv}(\Fl_{p,\bbf},\bk)$.  If at least one of $\cF$ or $\cG$ is $\varpi$-perfect, then $\cF \star \cG$ is $\varpi$-perfect.  As a consequence, there is an induced bifunctor
\[
\star: \Sm(\Gr_p,\bk) \times \Sm_{\Loop^+\Gv}(\Fl_{p,\bbf},\bk) \to \Sm(\Fl_\bbf,\bk).
\]
\end{lem}
\begin{proof}
If at least one of $\cF$ or $\cG$ is $\varpi$-perfect, Lemma~\ref{lem:sheaf-perfect} tells us that $\cF \boxtimes \cG$ is also $\varpi$-perfect.  From~\eqref{eqn:conv-diag} we see that every stalk of $\cF \tboxtimes \cG$ is isomorphic to a stalk of $\cF \boxtimes \cG$, so $\cF \tboxtimes \cG$ is also $\varpi$-perfect.  Finally, Lemma~\ref{lem:sheaf-perfect} again tells us that $m_*(\cF \tboxtimes \cG)$ is $\varpi$-perfect.
\end{proof}

\subsection{The right action of equivariant cohomology}
\label{ss:rightact}

Let $\delta$ be the skyscraper sheaf on the basepoint of $\Gr$.  Then we have
\[
\bEnd_{\Db_{\Gm \ltimes \Loop^+\Gv}}(\delta) \cong \coh^\bullet_{\Gm \times \Gv}(\pt,\bk).
\]
Now let $\cF \in \Db_{\Gm}(\Gr,\bk)$.  We have $\cF \cong \cF \star \delta$, and hence
\[
\bEnd_{\Db_\Gm(\Gr,\bk)}(\cF) \cong \bEnd_{\Db_\Gm(\Gr,\bk)}(\cF \star \delta).
\]
Now the functor $\cF \star ({-})$ induces a map
\[
\bEnd_{\Db_{\Gm \ltimes \Loop^+\Gv}}(\delta) \to 
\bEnd_{\Db_\Gm(\Gr,\bk)}(\cF).
\]
In other words, it makes $\bEnd(\cF)$ into a module over $\coh^\bullet_{\Gm \times \Gv}(\pt,\bk)$.

Here is a more ``hands-on'' construction of the $\coh^\bullet_{\Gm \times \Gv}(\pt,\bk)$-action on $\bEnd(\cF)$.  Let $\Loop^{++}\Gv$ be the kernel of the evaluation map $\ev_0: \Loop^+\Gv \to \Gv$, and then let
\[
\widetilde{\Gr} = \Loop\Gv/\Loop^{++}\Gv.
\]
This is an ind-variety with a free (right) action of $\Gv$, and it induces an isomorphism
\[
\widetilde{\Gr}/\Gv \cong \Gr.
\]
We therefore have an equivalence of categories
\[
\Db_\Gm(\Gr,\bk) \cong \Db_{\Gm \times \Gv}(\widetilde{\Gr},\bk),
\]
where the right-hand side involves the right action of $\Gv$ on $\widetilde{\Gr}$.  All $\bHom$-groups on the right-hand side are naturally $\coh^\bullet_{\Gm \times \Gv}(\pt,\bk)$-modules, and transferring this structure across the equivalence gives the desired action on $\bEnd(\cF)$.

\begin{lem}\label{lem:jnu-coh}
Let $\nu \in \bX^{++}$.  There is a canonical isomorphism of $\coh^\bullet_{G \times \Gm}(\pt,\bk)$-modules
\[
\bEnd(j^+_{\nu!}\cL^\nu_\AS) \cong \coh^\bullet_\Gm(\pt,\bk)_{\rS(\nu)}.
\]
\end{lem}
(The lemma is stated for the only case that will actually be needed in this paper.  But the same result is true if we replace $\cL^\nu_\AS$ by the constant sheaf on either an $\Iwup$-orbit or an $\Iwu$-orbit (and in these cases, one could allow all $\nu \in \bX$).)
\begin{proof}
Since $j^+_{\nu!}$ is fully faithful, the $\bEnd$-group we wish to compute is the same as $\bEnd(\cL^\nu_\AS)$, computed in $\Db_{\Gm,\IW}(\Iwup \cdot \rL_\nu,\bk)$.  Taking the stalk at $\rL_\nu$ gives an equivalence
\[
\Db_{\Gm,\IW}(\Iwup \cdot \rL_\nu,\bk) \cong 
\Db_\Gm(\{\rL_\nu\},\bk),
\]
under which $\cL^\nu_\AS$ correspondings to the constant sheaf $\underline{\bk}_\pt$ on the point $\{\rL_\nu\}$.  As a left $\coh^\bullet_\Gm(\pt,\bk)$-module, $\bEnd(\cF)$ is identified with
\[
\bEnd_{\Db_\Gm(\{\rL_\nu\},\bk)}(\underline{\bk}_\pt) = \coh^\bullet_\Gm(\pt,\bk).
\]

Let us now determine its structure as a right $\coh^\bullet_{\Gv}(\pt,\bk)$-module.  Let $h: \widetilde{\Gr} \to \Gr$ be the projection map, and let $X_\nu = h^{-1}(\rL_\nu)$.  We must compute
\[
\bEnd_{\Db_{\Gm \times \Gv}(X_\nu,\bk)}(\underline{\bk}_{X_\nu}).
\]
Concretely, we have $X_\nu = \bz^\nu \Loop^+\Gv/\Loop^{++}\Gv$.  Consider the point $\bz^\nu \Loop^{++}\Gv \in X_\nu$.  Using~\eqref{eqn:tz-conj}, we compute the action of $(t,g) \in \Gm \times \Gv$ on this point as follows:
\[
(t,g) \cdot \bz^\nu \Loop^{++}\Gv = t\bz^\nu t^{-1}g^{-1} \Loop^{++}\Gv = \bz^\nu \nu(t)g^{-1} \Loop^{++}\Gv.
\]
In particular, the stabilizer of $\bz^\nu \Loop^{++}\Gv$ is the image of $\hat \nu: \Gm \to \Gm \times \Gv$, and $\hat \nu$ induces an equivalence of categories
\[
\Db_{\Gm \times \Gv}(X_\nu,\bk) \cong \Db_{\Gm}(\{ \bz^\nu \Loop^{++}\Gv\},\bk).
\]
We deduce that
\[
\bEnd_{\Db_{\Gm \times \Gv}(X_\nu,\bk)}(\underline{\bk}_{X_\nu}) \cong \coh^\bullet_{\Gm}(\pt,\bk),
\]
where the right-hand side is a $\coh^\bullet_{\Gm \times \Gv}(\pt,\bk)$-module via $\rS(\hat\nu)$.
\end{proof}

\subsection{Convolution without \texorpdfstring{$\Loop^+\Gv$}{L+ check G}-equivariance}
\label{ss:conv-wo}

As noted in Section~\ref{ss:further}, the right-hand vertical map in Theorem~\ref{thm:main-intro} does not a priori make sense.  In this subsection, we explain how to construct a version of that functor after suitably twisting the $\Gm$-action.

\begin{lem}\label{lem:nuconv-bhom}
Let $\nu \in \bX^{++}$.  For $\cF, \cG \in \Parity_{\Gm \ltimes \Loop^+\Gv}(\Fl_\bbf,\bk)$, there is a unique way to fill in the dotted arrow below such that the diagram commutes:
\[
\begin{tikzcd}
\bHom(\cF, \cG) \ar[dr, "j^+_{\nu!}\cL^\nu_\AS \star ({-})"'] \ar[r]
&
\coh^\bullet_\Gm(\pt,\bk)_{\rS(\nu)} \otimes_{\coh^\bullet_{\Gm \times \Gv}(\pt,\bk)} \bHom(\cF, \cG) \ar[d,dashed] \\
&\bHom(j^+_{\nu!}\cL^\nu_\AS \star \cF, j^+_{\nu!}\cL^\nu_\AS \star \cG)
\end{tikzcd}
\]
Moreover, the dotted arrow is a homomorphism of $\coh^\bullet_\Gm(\pt,\bk)$-modules.
\end{lem}
\begin{proof}
The top horizontal map is surjective (because it is induced by tensoring with the surjective map $\rS(\hat\nu)$ from~\eqref{eqn:s-hatnu-defn}), so if the dotted arrow exists, its uniqueness is clear.  Moreover, both solid arrows are homomorphisms of $\coh^\bullet_\Gm(\pt,\bk)$-modules (see Lemma~\ref{lem:looprot-conv} for the arrow induced by $j^+_{\nu!}\cL^\nu_\AS \star ({-})$), so if the dotted arrow exists, it is automatically an $\coh^\bullet_\Gm(\pt,\bk)$-homomorphism.

Let us now prove that the dotted arrow exists.  Let $J$ be the kernel of $\rS(\hat \nu)$.  Since $\bHom(\cF,\cG)$ is free over $\coh^\bullet_{\Gm \times \Gv}(\pt,\bk)$, we have a short exact sequence
\[
0 \to J \otimes \bHom(\cF,\cG) \to \bHom(\cF,\cG) \xrightarrow{\rS(\hat\nu) \otimes({-})} \coh^\bullet_\Gm(\pt,\bk)_{\rS(\nu)} \otimes \bHom(\cF, \cG) \to 0
\]
where all tensor products are over $\coh^\bullet_{\Gm \times \Gv}(\pt,\bk)$.  Thus, to prove the lemma, it is enough to show the following claim:
\begin{equation}\label{eqn:hom-j-claim}
\begin{minipage}{4in}
For any $j \in J$ and $f \in \bHom(\cF,\cG)$, the functor $j^+_{\nu!}\cL^\nu_\AS \star ({-})$ sends $j \otimes f \in \bHom(\cF,\cG)$ to $0$.
\end{minipage}
\end{equation}
Furthermore, it is enough to prove this when $j$ and $f$ are homogeneous elements, say of degrees $d$ and $e$ respectively.  In particular, $f$ can be regarded as a morphism $f: \cF \to \cG[e]$.  Similarly, since $J \subset \coh^\bullet_{\Gm \times \Gv}(\pt,\bk) \cong \bEnd(\delta)$, we can regard $j$ as a morphism $j: \delta \to \delta[d]$.  To prove~\eqref{eqn:hom-j-claim}, we must show that the morphism
\[
j^+_{\nu!}\cL^\nu_\AS \star j \star f: \cF \to \cG[d+e]
\]
is zero.  This follows from Lemma~\ref{lem:jnu-coh}, which implies that
\[
j^+_{\nu!}\cL^\nu_\AS \star j: j^+_{\nu!}\cL^\nu_\AS \star \delta \to j^+_{\nu!}\cL^\nu_\AS \star \delta[d]
\]
is already zero.
\end{proof}

\begin{lem}\label{lem:nuconv-parity}
Let $\nu \in \bX^{++}$, and suppose it is IW-clean.  There is a unique way (up to natural isomorphism) to fill in the dotted arrow below such that the diagram commutes up to natural isomorphism:
\[
\begin{tikzcd}
\Parity_{\Gm \ltimes \Loop^+\Gv}(\Fl_\bbf,\bk) \ar[r] \ar[dr, "j^+_{\nu!}\cL^\nu_\AS \star ({-})"'] &
\Parity_{\Gm,\sph|\nu}(\Fl_\bbf,\bk) \ar[d, dashed, "j^+_{\nu!}\cL^\nu_\AS \star ({-})"] \\
& \Parity_{\Gm,\IW}(\Fl_\bbf,\bk)
\end{tikzcd}
\]
\end{lem}
In this lemma, the assumption of IW-cleanness means that $j^+_{\nu!}\cL^\nu_\AS$ is a parity sheaf, so by Section~\ref{sss:conv-parity}, the functor $j^+_{\nu!}\cL^\nu_\AS \star ({-})$ does indeed take values in $\Parity_{\Gm,\IW}(\Fl_\bbf,\bk)$.
\begin{proof}
The category $\Parity_{\Gm,\sph|\nu}(\Fl_\bbf,\bk)$ admits the following description: its objects are the same as those of $\Parity_{\Gm \ltimes \Loop^+\Gv}(\Fl_\bbf,\bk)$, but its $\Hom$-groups are given by the rule
\begin{multline*}
\bHom_{\Parity_{\Gm,\sph|\nu}(\Fl_\bbf,\bk)}(\cF, \cG)
\cong \\
\coh^\bullet_\Gm(\pt,\bk)_{\rS(\nu)} \otimes_{\coh^\bullet_{\Gm \times \Gv}(\pt,\bk)} \Hom_{\Parity_{\Gm \ltimes \Loop^+\Gv}(\Fl_\bbf,\bk)}(\cF,\cG).
\end{multline*}
The result then follows from Lemma~\ref{lem:nuconv-bhom}.
\end{proof}

\begin{lem}\label{lem:nuconv-smith}
Let $p\nu \in p\bX^{++}$, and suppose it is IW-clean.  There is a unique way (up to natural isomorphism) to fill in the dotted arrow below such that the diagram commutes up to natural isomorphism:
\[
\begin{tikzcd}
\Parity_{\Gm,\sph|p\nu}(\Fl_{p,\bbf},\bk) \ar[r] \ar[d, "j^{p+}_{p\nu!}\cL^{p\nu}_\AS \star ({-})"', "\text{\normalfont Lem.~\ref{lem:nuconv-parity}}"] &
\SmParity_{\sph|p\nu}(\Fl_{p,\bbf},\bk) \ar[d, dashed, "j^{p+}_{p\nu!}\cL^{p\nu}_\AS \star ({-})"] \\
\Parity_{\Gm,\IW}(\Fl_{p,\bbf},\bk) \ar[r] &
\SmParity_\IW(\Fl_{p,\bbf},\bk)
\end{tikzcd}
\]
\end{lem}
\begin{proof}
The proof is very similar to that of Lemma~\ref{lem:nuconv-parity}: by Lemma~\ref{lem:sphom-gen} and Corollary~\ref{cor:parity-full}, $\SmParity_{\sph|p\nu}(\Fl_{p,\bbf},\bk)$ can be identified with the category whose objects are the same as those of $\Parity_{\Gm,\sph|p\nu}(\Fl_{p,\bbf},\bk)$, but with $\Hom$-groups given by
\begin{multline*}
\Hom_{\SmParity_{\sph|p\nu}(\Fl_{p,\bbf},\bk)}(\cF,\cG) 
\cong \\
\coh^\bullet_\Gm(\pt,\bk)/(\can^2 - 1) \otimes_{\coh^\bullet_\Gm(\pt,\bk)} \bHomev_{\Parity_{\Gm,\sph|p\nu}(\Fl_{p,\bbf},\bk)}(\cF,\cG).
\end{multline*}
Since the map on $\bHom$-groups in $\Parity_{\Gm,\sph|p\nu}(\Fl_{p,\bbf},\bk)$ is a $\coh^\bullet_\Gm(\pt,\bk)$-homo\-morphism (see Lemma~\ref{lem:nuconv-bhom}), the present lemma follows.
\end{proof}

We conclude this subsection by showing that the assumption of Lemma~\ref{lem:nuconv-smith} is not vacuous.

\begin{lem}
\phantomsection
\label{lem:steinberg-clean}
\begin{enumerate}
\item We have $j^+_{\varsigma!}\cL^\varsigma_\AS[\langle \varsigma, 2\check\rho\rangle] \star \cE^\sph_{(p-1)\varsigma} \cong j^+_{p\varsigma!}\cL^{p\varsigma}_\AS[\langle p\varsigma, 2\check\rho\rangle]$.
\item The weight $p\varsigma \in \bX^+$ is IW-clean.
\end{enumerate}
\end{lem}
\begin{proof}
Let $\Perv_{\Loop^+\Gv}(\Gr,\bk)$ and $\Perv_{\IW}(\Gr,\bk)$ be the abelian categories of perverse sheaves in $\Db_{\Loop^+\Gv}(\Gr,\bk)$ and in $\Db_\IW(\Gr,\bk)$, respectively.  According to~\cite[Theorem~3.9]{bgmrr}, the functor $j^+_{\varsigma!}\cL^\varsigma_\AS[\langle \varsigma, 2\check\rho\rangle] \star ({-}): \Db_{\Loop^+\Gv}(\Gr,\bk) \to \Db_{\IW}(\Gr,\bk)$ is $t$-exact, and it restricts to an equivalence of abelian categories
\[
j^+_{\varsigma!}\cL^\varsigma_\AS[\langle \varsigma, 2\check\rho\rangle]  \star ({-}): \Perv_{\Loop^+\Gv}(\Gr,\bk) \simto \Perv_{\IW}(\Gr,\bk).
\]
Composing this with the geometric Satake equivalence, we get an equivalence of abelian categories
\[
\Rep(G) \simto \Perv_{\IW}(\Gr,\bk).
\]
This equivalence has the property that for any $\lambda \in \bX^+$, it sends
\begin{align*}
\Delta(\lambda) &\mapsto j^+_{(\lambda+\varsigma)!}\cL^{\lambda+\varsigma}_\AS[\langle \lambda + \varsigma, 2\check\rho\rangle], \\
\nabla(\lambda) &\mapsto j^+_{(\lambda+\varsigma)*}\cL^{\lambda+\varsigma}_\AS[\langle \lambda + \varsigma, 2\check\rho\rangle],
\end{align*}
where $\Delta(\lambda)$, resp.~$\nabla(\lambda)$, is the Weyl module, resp.~dual Weyl module, of highest weight $\lambda$.  Now take $\lambda = (p-1)\varsigma$.  By~\cite[Remark~II.3.19]{jan:rag}, $\Delta((p-1)\varsigma) \cong \nabla((p-1)\varsigma)$, so that both are isomorphic to the indecomposable tilting module $T((p-1)\varsigma)$.  This tilting module corresponds under the geometric Satake equivalence to the parity sheaf $\cE^\sph_{(p-1)\varsigma}$.

We conclude that the functor $j^+_{\varsigma!}\cL^\varsigma_\AS[\langle \varsigma, 2\check\rho\rangle]  \star ({-})$ sends $\cE^\sph_{(p-1)\varsigma}$ to both
\[
j^+_{p\varsigma!}\cL^{p\varsigma}_\AS[\langle p\varsigma, 2\check\rho\rangle] 
\qquad\text{and}\qquad
j^+_{p\varsigma*}\cL^{p\varsigma}_\AS[\langle p\varsigma, 2\check\rho\rangle].
\]
Both parts of the lemma follow.
\end{proof}

\subsection{Indecomposability}
\label{ss:indecomp}

The following result is a slight generalization of~\cite[Proposition~4.6]{rw:scf}.

\begin{prop}\label{prop:diw-indecomp}
The functor
\[
j^+_{\varsigma!}\cL^\varsigma_\AS \star ({-}): \Parity_{\Gm,\sph|\varsigma}(\Fl_\bbf,\bk) \to \Parity_{\Gm,\IW}(\Fl_\bbf,\bk)
\]
sends indecomposable objects to indecomposable objects.
\end{prop}

\begin{proof}
In the special case where $\bbf = \ba$, this is~\cite[Proposition~4.6]{rw:scf}.  

Suppose now that $\bbf \ne \ba$.  We will reduce to the previous case using the map $\pi_\bbf$ from~\eqref{eqn:pif-defn}.  Let $\cF \in \Parity_{\Gm,\sph}(\Fl_\bbf,\bk)$ be indecomposable.  A routine argument shows that
\begin{equation}\label{eqn:diw-indecomp}
j^+_{\varsigma!}\cL^\varsigma_\AS \star \pi_\bbf^*\cF \cong \pi_\bbf^*(j^+_{\varsigma!}\cL^\varsigma_\AS \star \cF).
\end{equation}
(See, for instance,~\cite[Lemma~5.3(ii)]{ar:agsr} for a similar statement.)  By~\cite[Lemma~A.5]{acr:pets}, $\pi_\bbf^*\cF$ is indecomposable, and then by the previous paragraph (i.e., by~\cite[Proposition~4.6]{rw:scf}), the left-hand side of~\eqref{eqn:diw-indecomp} is indecomposable.  Since $\pi_\bbf^*$ kills no nonzero object, we see from the right-hand side of~\eqref{eqn:diw-indecomp} that $j^+_{\varsigma!}\cL^\varsigma_\AS \star \cF$ must also be indecomposable.
\end{proof}

Proposition~\ref{prop:diw-indecomp} applies to the $p$-scaled flag variety $\Fl_{p,\bbf}$, and says that
\[
j^{p+}_{p\varsigma!}\cL^\varsigma_\AS \star ({-}): \Parity_{\Gm,\sph|p\varsigma}(\Fl_{p,\bbf},\bk) \to \Parity_{\Gm,\IW}(\Fl_{p,\bbf},\bk)
\]
sends indecomposable objects to indecomposable objects.  By taking a sum over facets $\bbf_\lambda$ for $\lambda \in -p\bba \cap \bX^+$ as in~\eqref{eqn:grpi-overcount} and~\eqref{eqn:grpi-isom}, we obtain the following consequence.

\begin{cor}\label{cor:gv-indecomp}
The functor
\[
j^{p+}_{p\varsigma!}\cL^p_\AS \star ({-}): \Parity_{\Gm,\sph|p\varsigma}(\Gr^\varpi,\bk) \to \Parity_{\Gm,\IW}(\Gr^\varsigma,\bk)
\]
sends indecomposable objects to indecomposable objects.  In particular, we have
\[
j^{p+}_{p\varsigma!}\cL^p_\AS[\langle \varsigma, 2\check\rho\rangle] \star \cE^{\varpi,\sph}_\lambda \cong \cE^{\varpi,\IW,\bo}_{\lambda + p \varsigma}.
\]
\end{cor}

\section{Split Grothendieck groups}
\label{sec:groth}

For an additive category $\cA$, denote its split Grothendieck group by $\Ko(\cA)$.  Suppose now that $\cA$ is a full additive subcategory of a triangulated category that is stable under the shift operation $[1]$.  Then $[1]$ induces an automorphism of the abelian group $\Ko(\cA)$.  We can then make $\Ko(\cA)$ into a $\Z[v,v^{-1}]$-module by letting $v$ act by the shift operation $[1]$.

In this section, we will study the $\Z[v,v^{-1}]$-modules $\Ko(\cA)$ when $\cA$ is a category of parity sheaves or Smith parity sheaves on a partial affine flag variety.  We continue to retain the assumption from Section~\ref{ss:bhom} that $p$ is not a torsion prime for $\Gv$.

\subsection{Parity sheaves on partial affine flag varieties}
\label{ss:groth-parity}

We begin by studying the Grothendieck groups
\[
\Ko(\Parity_{\Gm \ltimes \Iw}(\Fl,\bk)),
\quad
\Ko(\Parity_{\Gm,\sph}(\Fl_\bbf,\bk)),
\quad
\Ko(\Parity_{\Gm,\IW}(\Fl_\bbf,\bk)).
\]
Note that in all three cases, forgetting the $\Gm$-action induces a bijection on isomorphism classes of objects, so we may identify
\begin{equation}\label{eqn:groth-isom}
\Ko(\Parity_{\Gm \ltimes \Iw}(\Fl,\bk)) = \Ko(\Parity_\Iw(\Fl,\bk)),
\end{equation}
and likewise for other categories involving minor changes in equivariance.

Here are descriptions of these Grothendieck groups in some cases.

\begin{lem}
\phantomsection
\label{lem:groth}
\begin{enumerate}
\item The map $\ch: \Ko(\Parity_{\Gm \ltimes \Iw}(\Fl,\bk)) \simto \cH_\ext$ given by
\[
\ch([\cF]) = \sum_{\substack{w \in \Wext \\ i \in \Z}} \rank \coh^{i- \dim \Fl_w}(j_w^*\cF) v^{-i} H_w
\]
is an isomorphism of $\Z[v,v^{-1}]$-algebras.\label{it:groth-hecke}
\item The map $\ch: \Ko(\Parity_{\Gm,\sph}(\Fl_\bbf,\bk)) \simto \cM^\bbf$
given by
\[
\ch([\cF]) = \sum_{\substack{w \in \oWext^\bbf \\ i \in \Z}} \rank \coh^{i- \dim \Fl^\sph_{\bbf,w}}((j^\sph_w)^*\cF) v^{-1} M^\bbf_w
\]
is an isomorphism of $\Z[v,v^{-1}]$-modules.  In the special case where $\bbf = \ba$, this is an isomorphism of right $\cH_\ext$-modules.\label{it:groth-sph}
\item The map $\ch: \Ko(\Parity_{\Gm,\IW}(\Fl_\bbf,\bk)) \simto \cN^\bbf$ given by
\[
\ch(\cF) = \sum_{\substack{w \in \spWext^\bbf \\ i \in \Z}} \rank \coh^{i- \dim \Fl^+_{\bbf,w}}((j^+_w)^*\cF) v^{-i} N^\bbf_w.
\]
is an isomorphism of $\Z[v,v^{-1}]$-modules.\label{it:groth-anti}
\end{enumerate}
\end{lem}
(When $\bbf = \ba$, the map in part~\eqref{it:groth-anti} can also be shown to be an isomorphism of right $\cH_\ext$-modules, but we will not need this fact.)
\begin{proof}[Proof sketch]
Part~\eqref{it:groth-hecke} is well known: see, for instance,~\cite[Eq.~(1.2.3) and Theorem~10.3.2]{rw:tmpcb}.

For part~\eqref{it:groth-sph}, observe that the collection $\{ [\cE^\bbf_w] : w \in \oWext^\bbf \}$ is a $\Z[v,v^{-1}]$-basis for $\Ko(\Parity_{\Gm,\sph}(\Fl_\bbf,\bk))$.  The formula for $\ch$ yields
\[
\ch([\cE^\bbf_w]) = M^\bbf_w + (\text{lower terms}),
\]
where ``lower terms'' means a linear combination of various $M^\bbf_y$ with $\dim \Fl_{\bbf,y} < \dim \Fl_{\bbf,w}$.  This description shows that the map $\ch$ takes a basis of the Grothendieck group $\Ko(\Parity_{\Gm,\sph}(\Fl_\bbf,\bk))$ to a basis of $\cM^\bbf$, so it is at least an isomorphism of $\Z[v,v^{-1}]$-modules.  When $\bbf = \ba$, $\cM$ is a $\Z[v,v^{-1}]$-submodule of $\cH_\ext$, and the fact that $\ch$ is an isomorphism of right $\cH_\ext$-modules follows from the fact that part~\eqref{it:groth-hecke} is a ring isomorphism.

For part~\eqref{it:groth-anti}, the proof that we have an isomorphism of $\Z[v,v^{-1}]$-modules is simlar to that in part~\eqref{it:groth-sph}.
\end{proof}

\begin{lem}\label{lem:groth-pull}
The following diagrams commute:
\[
\begin{tikzcd}
\Ko(\Parity_{\Gm,\sph}(\Fl_\bbf,\bk)) \ar[r, "\ch"] \ar[d, "\pi_\bbf^*{[\length(w_\bbf)]}"'] & \cM^\bbf \ar[d, hook] \\
\Ko(\Parity_{\Gm,\sph}(\Fl,\bk)) \ar[r, "\ch"] \ & \cM
\end{tikzcd}
\qquad
\begin{tikzcd}
\Ko(\Parity_{\Gm,\IW}(\Fl_\bbf,\bk)) \ar[r, "\ch"] \ar[d, "\pi_\bbf^*{[\length(w_\bbf)]}"'] & \cN^\bbf \ar[d, hook] \\
\Ko(\Parity_{\Gm,\IW}(\Fl,\bk)) \ar[r, "\ch"] \ & \cN
\end{tikzcd}
\]
\end{lem}
\begin{proof}
The map $\pi_\bbf$ is smooth of relative dimension $\length(w_\bbf)$.  For $w \in \Wext^\bbf$, the preimage of the $\Iwup$-orbit $\Fl^+_{\bbf,w}$ is given by
\[
\pi_\bbf^{-1}(\Fl^+_{\bbf,w}) = \bigcup_{u \in wW_\bbf} \Fl^+_u.
\]
If $\cF \in \Db_{\Gm,\IW}(\Fl_\bbf,\bk)$, then for $u \in wW_\bbf$, we have
\[
\rank \coh^i((j^+_u)^*\pi_\bbf^*\cF)
= \rank \coh^i((j^+_w)^*\cF),
\]
for any $i \in \Z$.  Recall that $\dim \Fl^+_u = \length(u)$, while $\dim \Fl^+_{\bbf,w} = \length(w) - \length(w_\bbf)$.   We can rewrite the preceding equation as
\[
\rank \coh^{i + \length(u) - \length(w) - \dim \Fl^+_u}((j^+_u)^*\pi_\bbf^*\cF[\length(w_\bbf)])
= \rank \coh^{i- \dim \Fl^+_{\bbf,w}}((j^+_w)^*\cF).
\]
The commutativity of the diagram involving $\cN^\bbf$ and $\cN$ follows from this equation, the definition of $\ch$, and~\eqref{eqn:nfw-defn}.

Now let $w \in \oWext^\bbf$.  Then we have
\[
\pi_\bbf^{-1}(\Fl^\sph_{\bbf,w}) = \bigcup_{u \in WwW_\bbf} \Fl_u = \bigcup_{u \in \oWext \cap WwW_\bbf} \Fl^\sph_u.
\]
The argument for the diagram involving $\cM^\bbf$ and $\cM$ proceeds similarly, using this observation along with~\eqref{eqn:mfw-defn}.
\end{proof}

\begin{lem}\label{lem:groth-whitav}
The following diagram commutes:
\[
\begin{tikzcd}
\Ko(\Parity_{\Gm,\sph}(\Fl_\bbf,\bk)) \ar[r, "\ch"] \ar[d, "j^+_{\varsigma!}\cL^\varsigma_\AS{[\langle \varsigma, 2\check\rho\rangle]} \star ({-})"'] & \cM^\bbf \ar[d, "\phi"] \\
\Ko(\Parity_{\Gm,\IW}(\Fl_\bbf,\bk)) \ar[r, "\ch"] & \cN^\bbf
\end{tikzcd}
\]
\end{lem}
\begin{proof}
Using the map~\eqref{eqn:pif-defn} along with Lemma~\ref{lem:groth-pull} and the formula~\eqref{eqn:diw-indecomp}, we can reduce to the special case where $\bbf = \ba$.  In this special case, the present lemma is~\cite[Lemma~4.5]{rw:scf}.
\end{proof}

\subsection{Parity sheaves on the affine Grassmannian}

Lemma~\ref{lem:groth} has already given us one description of the Grothendieck groups on $\Gr = \Fl_\bo$, but in this case there is additional information coming from the geometric Satake equivalence, summarized in the following statement.

\begin{lem}
\phantomsection
\label{lem:groth-gr}
\begin{enumerate}
\item The map $\ch: \Ko(\Parity_{\Gm,\sph}(\Gr,\bk))|_{v=1} \to \Z[\bX]^W$ given by\label{it:ggr-sph}
\[
\ch([\cF]) = \sum_{\substack{\mu \in \bX^+ \\ i \in \Z}} \rank \coh^i(\cF|_{\Gr_\mu}) s(\lambda)
\]
is an isomorphism of abelian groups.  Moreover, if the parity sheaf $\cF$ is also perverse and corresponds to $V \in \Rep(G)$ under the geometric Satake equivalence, then
\[
\ch([\cF]) = \ch V.
\]
\item The map $\ch: \Ko(\Parity_{\Gm,\IW}(\Gr,\bk))|_{v= 1} \to \Z[\bX]^W$ given by\label{it:ggr-iw}
\[
\ch([\cF]) = \sum_{\substack{\mu \in \bX^{++} \\ i \in \Z}} \rank \coh^i(\cF|_{\Gr^{+}_\mu}) \chi(\lambda - \varsigma)
\]
is an isomorphism of abelian groups.  Moreover, if the parity sheaf $\cF$ is also perverse and corresponds to $V \in \Rep(G)$ under the Iwahori--Whittaker model for the Satake equivalence from~\cite{bgmrr}, then
\[
\ch([\cF]) = \ch V.
\]
\end{enumerate}
\end{lem}
\begin{proof}
In both cases, the fact that the maps are isomorphisms of abelian groups is straightforward from a comparison of bases for both sides.  In part~\eqref{it:ggr-sph}, the claim that $\ch([\cF]) = \ch V$ holds by~\cite[Corollary~4.1]{jmw:pstm}.  In part~\eqref{it:ggr-iw}, that claim follows from the main results of~\cite{bgmrr}.
\end{proof}

\begin{rmk}
If $p$ is good for $G$, then by~\cite[Corollary~1.6]{mr:etsps}, every $\cE^\sph_\lambda$ is perverse, and the second part of Lemma~\ref{lem:groth-gr}\eqref{it:ggr-sph} applies to all $\cE^\sph_\lambda$.  This point will come up again when we apply Lemma~\ref{lem:groth-gr}\eqref{it:ggr-sph} in the proof of Theorem~\ref{thm:main}.
\end{rmk}

\subsection{Parity sheaves on the fixed-point locus}

We now turn to Grothendieck groups of parity sheaves on $\Gr^\varpi$.

\begin{lem}
\phantomsection
\label{lem:groth-varpi}
\begin{enumerate}
\item The map $\ch: \Ko(\Parity_{\Gm,\sph}(\Gr^\varpi,\bk))|_{v= 1} \to \Z[\bX]^W$ given by\label{it:gv-sph}
\[
\ch([\cF]) = \sum_{\substack{\mu \in \bX^+ \\ i \in \Z}} \rank \coh^i(\cF|_{\Gr^{\varpi}_\mu}) s(\mu)
\]
is an isomorphism of abelian groups.  Moreover, this isomorphism makes the following diagram commute:
\[
\begin{tikzcd}[column sep=huge]
\displaystyle\bigoplus_{\lambda \in (-p\bba) \cap \bX^+} \Ko(\Parity_{\Gm,\sph}(\Fl_{p,\bbf_\lambda}^\circ))|_{v=1} \ar[r, "\ch\ \text{\normalfont (Lemma~\ref{lem:groth})}", "\sim"'] \ar[d, "\eqref{eqn:grpi-isom}"', "\wr"] &
\displaystyle\bigoplus_{\lambda \in (-p\bba) \cap \bX^+} \cM^{\bbf_\lambda}_\aff|_{v=1} \ar[d, "\theta^\sph_\aff\ \text{\normalfont (Lemma~\ref{lem:theta-sph})}", "\wr"'] \\
\Ko(\Parity_{\Gm,\sph}(\Gr^\varpi,\bk))|_{v= 1} \ar[r, "\ch", "\sim"'] &
\Z[\bX]^W
\end{tikzcd}
\]
\item The map $\ch: \Ko(\Parity_{\Gm,\IW}(\Gr^\varpi,\bk))|_{v= 1} \to \Z[\bX]^W$ given by\label{it:gv-iw}
\[
\ch([\cF]) = \sum_{\substack{\mu \in \bX^{++} \\ i \in \Z}} \rank \coh^i(\cF|_{\Gr^{\varpi,+}_\mu}) \chi(\mu - \varsigma)
\]
is an isomorphism of abelian groups.  Moreover, this isomorphism makes the following diagram commute:
\[
\begin{tikzcd}[column sep=huge]
\displaystyle\bigoplus_{\lambda \in (-p\bba) \cap \bX^+} \Ko(\Parity_{\Gm,\IW}(\Fl_{p,\bbf_\lambda}^\circ))|_{v=1} \ar[r, "\ch\ \text{\normalfont (Lemma~\ref{lem:groth})}", "\sim"'] \ar[d, "\eqref{eqn:grpi-isom}"', "\wr"] &
\displaystyle\bigoplus_{\lambda \in (-p\bba) \cap \bX^+} \cN^{\bbf_\lambda}_\aff|_{v=1}\ar[d, "\theta^\IW_\aff\ \text{\normalfont (Lemma~\ref{lem:theta-iw})}", "\wr"'] \\
\Ko(\Parity_{\Gm,\IW}(\Gr^\varpi,\bk))|_{v= 1} \ar[r, "\ch", "\sim"'] &
\Z[\bX]^W
\end{tikzcd}
\]
\end{enumerate}
\end{lem}
\begin{proof}
For part~\eqref{it:gv-sph}, combining the formula from Lemma~\ref{lem:groth}\eqref{it:groth-sph} with the formula for $\theta^\sph_\aff$ from Lemma~\ref{lem:theta-sph}, we find that
\[
\theta^\sph_\aff(\ch([\cF])) =
\sum_{\lambda \in (-p\bba) \cap \bX^+} \sum_{\substack{w \in \oWaff^{\bbf_\lambda} \\ i \in \Z}}
\rank \coh^i(\cF|_{\Fl_{p,\bbf_\lambda,w}^\sph}) s(w_\bo(w \boxl \lambda)).
\]
Recall from Section~\ref{ss:fixed-pts} that $\Fl_{p,\bbf_\lambda,w} = \Gr^{\varpi}_{w_\bo(w \boxl \lambda)}$.  Thus, the formula above agrees with the one in the statement of the present lemma, and the diagram commutes.

The proof of part~\eqref{it:gv-iw} is similar: combine the formulas from Lemma~\ref{lem:groth}\eqref{it:groth-anti} and Lemma~\ref{lem:theta-iw}, and recall from Section~\ref{ss:fixed-pts} that for $w \in \spWaff^{\bbf_\lambda}$, $\Fl^+_{p,\bbf_\lambda,w} = \Gr^{\varpi+}_{w \boxl \lambda}$.
\end{proof}

\begin{prop}\label{prop:groth-varpi-total}
The following diagram commutes:
\[
\begin{tikzcd}
\Ko(\Parity_{\Gm,\sph}(\Gr^\varpi,\bk))|_{v=1} \ar[r, "\ch", "\sim"'] \ar[d, "j^{p+}_{p\varsigma!}\cL^{p\varsigma}_\AS{[\langle \varsigma, 2\check\rho\rangle]} \star ({-})"'] & \Z[\bX]^W \ar[d, "\text{\normalfont multiply by $\chi((p-1)\varsigma)$}"] \\
\Ko(\Parity_{\Gm,\IW}(\Gr^\varpi,\bk)) \ar[r, "\ch", "\sim"'] & \Z[\bX]^W
\end{tikzcd}
\]
\end{prop}
\begin{proof}
Combine the commutative diagrams from Lemmas~\ref{lem:theta-compare}, \ref{lem:groth-whitav}, and~\ref{lem:groth-varpi}.
\end{proof}

\subsection{Split Grothendieck groups for Smith parity sheaves}

In this subsection, we will study the split Grothendieck groups
\[
\Ko(\SmParity_{\sph}(\Gr^\varpi,\bk)),
\qquad
\Ko(\SmParity_{\IW}(\Gr^\varpi,\bk)).
\]
As in Section~\ref{ss:groth-parity}, these groups are a priori modules over $\Z[v,v^{-1}]$, with $v$ acting by $[1]$.  But since $[2]$ is isomorphic to the identity functor (see Section~\ref{ss:smith-defn}), the Grothendieck groups above are actually modules over
\[
\Z[v]/(v^2 - 1).
\]

\begin{lem}\label{lem:groth-psm}
The localization functor $\Psm$ induces isomorphisms
\begin{align*}
\Ko(\SmParity_{\sph}(\Gr^\varpi,\bk)) &\cong \Z[v]/(v^2-1) \otimes_{\Z[v,v^{-1}]}
\Ko(\Parity_{\Gm,\sph}(\Gr,\bk)), \\
\Ko(\SmParity_{\IW}(\Gr^\varpi,\bk)) &\cong \Z[v]/(v^2-1) \otimes_{\Z[v,v^{-1}]}
\Ko(\Parity_{\Gm,\IW}(\Gr,\bk)).
\end{align*}
\end{lem}
\begin{proof}
We will explain the details for the spherical case; the Iwahori--Whittaker case is very similar. The functor $\Psm$ at least induces a map
\begin{equation}\label{eqn:groth-smith-map}
\Z[v]/(v^2-1) \otimes_{\Z[v,v^{-1}]}
\Ko(\Parity_{\Gm,\sph}(\Gr,\bk)) \to \Ko(\SmParity_{\sph}(\Gr^\varpi,\bk)).
\end{equation}
The collection $\{\cE^\sph_\lambda : \lambda \in \bX^+ \}$ is a $\Z[v,v^{-1}]$-basis for $\Ko(\Parity_{\Gm,\sph}(\Gr,\bk))$.  Since this is supported on the variety
\[
\overline{\Gr_\lambda} =
\Gr_\lambda \cup \bigcup
\left(
\text{various orbits $\Gr_\mu$ with $\dim \Gr_\mu < \dim \Gr_\lambda$}
\right),
\]
we see that
\[
\Psm(\cE^\sph_\lambda) \cong \cE^{\varpi,\sph}_{\Sm,\lambda} \oplus \bigoplus \left(
\text{various $\cE^{\varpi,\sph}_{\Sm,\mu}[n]$ with $\dim \Gr_\mu < \dim \Gr_\lambda$}
\right),
\]
This shows that the map~\eqref{eqn:groth-smith-map} sends a $\Z[v]/(v^2-1)$-basis of its domain to a basis of its codomain, so it is an isomorphism.
\end{proof}

\begin{lem}\label{lem:groth-q}
The quotient functor $\rQ$ induces isomorphisms
\begin{align*}
\Ko(\SmParity_{\sph}(\Gr^\varpi,\bk)) &\cong \Z[v]/(v^2-1) \otimes_{\Z[v,v^{-1}]}
\Ko(\Parity_{\Gm,\sph}(\Gr^\varpi,\bk)), \\
\Ko(\SmParity_{\IW}(\Gr^\varpi,\bk)) &\cong \Z[v]/(v^2-1) \otimes_{\Z[v,v^{-1}]}
\Ko(\Parity_{\Gm,\IW}(\Gr^\varpi,\bk)).
\end{align*}
\end{lem}
\begin{proof}
The proof is similar to that of Lemma~\ref{lem:groth-psm} (but somewhat simpler, as $\rQ$ sends indecomposable parity sheaves to indecomposable Smith parity sheaves, by Corollary~\ref{cor:parity-full}).
\end{proof}

\begin{lem}\label{lem:groth-smith-v1}
Let $\cF \in \SmParity_{\Gm,\IW}(\Gr^\varpi,\bk)$.  We have $\cF = 0$ if and only if the class of $\cF$ in the specialized split Grothendieck group
\begin{multline*}
\Ko(\SmParity_{\Gm,\IW}(\Gr^\varpi,\bk))|_{v=1} ={} \\
 \Z[v]/(v-1) \otimes_{\Z[v]/(v^2-1)} \Ko(\SmParity_{\Gm,\IW}(\Gr^\varpi,\bk)).
\end{multline*}
vanishes.
\end{lem}
\begin{proof}
The isomorphism classes of indecomposable Smith parity sheaves form a $\Z$-basis for $\Ko(\SmParity_{\Gm,\IW}(\Gr^\varpi,\bk))$.  Specializing at $v= 1$, we obtain a free abelian group with a basis consisting of isomorphism classes of indecomposable Smith parity sheaves \emph{up to shift}.  It is clear that any nonzero Smith parity object is a nonnegative linear combination of basis elements in this basis.
\end{proof}

\section{Main results}
\label{sec:geom}

\subsection{Compatibility of convolution and Smith--Treumann localization}
\label{ss:compat}

In this subsection, we will study how Smith--Treumann localization interacts with convolution, and we will prove the main geometric theorem of the paper.  Up through Proposition~\ref{prop:main-summand}, we assume that $p$ is not a torsion prime for $\Gv$.  Later, we will impose the stronger assumption that $p$ is good for $G$.

Recall that $\Gr_p$ is identified with a union of connected components of $\Gr^\varpi$.  The restriction of $\bi$ to $\Gr_p$ is denoted by
\[
\bi_\thin: \Gr_p \hookrightarrow \Gr.
\]
There is a convolution operation for $\Gr_p$ and $\Gr^\varpi$ that is defined using the following counterpart of~\eqref{eqn:conv-diag}:
\[
\Gr_p \times \Gr^\varpi
\xleftarrow{p_p} \Loop_p\Gv \times \Gr^\varpi \xrightarrow{q_p} \Loop_p \times^{\Loop_p^+\Gv} \Gr^\varpi \xrightarrow{m_p} \Gr^\varpi.
\]
For $\cF \in \Dbc(\Gr_p)$ and $\cG \in \Db_{\Loop^+_p\Gv}(\Gr^\varpi)$, we can form the objects
\[
\cF \tboxtimes \cG
\qquad\text{and}\qquad
\cF \star \cG = m_{p*}(\cF \tboxtimes \cG).
\]

\begin{lem}\label{lem:twistprod-pull}
Let $\bit: \Loop_p\Gv \times^{\Loop^+_p\Gv} \Gr^\varpi \hookrightarrow \Loop\Gv \times^{\Loop^+\Gv} \Gr$ be the inclusion map.  For $\cF \in \Db_\Gm(\Gr,\bk)$ and $\cG \in \Db_{\Gm \ltimes \Loop^+\Gv}(\Gr,\bk)$, there are natural isomorphisms
\begin{align*}
\bit^!(\cF \tboxtimes \cG) \cong (\bi_\thin^!\cF) \tboxtimes (\bi^!\cG), \\
\bit^*(\cF \tboxtimes \cG) \cong (\bi_\thin^*\cF) \tboxtimes (\bi^*\cG).
\end{align*}
\end{lem}
\begin{proof}
Let $Z \subset \Gr$ be a closed union of finitely many $\Loop^+\Gv$-orbits that contains the support of $\cG$.  Then the $\Loop^+\Gv$-action on $Z$ factors through some finite-dimensional quotient group $J$.  Let $K$ be the kernel of $\Loop^+\Gv \twoheadrightarrow J$.  (Thus, $K$ acts trivially on $Z$.). By substituting $z^p$ for $z$, we get an analogous sequence of subgroups $K_p \hookrightarrow \Loop^+_p\Gv \twoheadrightarrow J_p$.  

Consider the fixed-point set $Z^\varpi = Z \cap \Gr^\varpi$, which carries an action of $\Loop^+_p\Gv$.  The subgroup $K_p$ acts trivially on $Z^\varpi$, so the $\Loop^+_p\Gv$-action on $Z^\varpi$ factors through $\Loop^+\Gv \twoheadrightarrow J_p$.

We now consider the following diagram (whose top line comes from~\eqref{eqn:conv-diag}):
\begin{equation}\label{eqn:tppull1}
\begin{tikzcd}
\Gr \times Z \ar[d, equal] & \Loop\Gv \times Z \ar[l, "p"'] \ar[r, "q"] \ar[d] & \Loop\Gv \times^{\Loop^+\Gv} Z \ar[d, equal] \\
\Gr \times Z & (\Loop\Gv/K) \times \Gr \ar[l, "\bar p"'] \ar[r, "\bar q"] & (\Loop\Gv/K) \times^J Z
\end{tikzcd}
\end{equation}
For $\cF$ and $\cG$ as in the statement of the lemma, the object $\cF \tboxtimes \cG$ was characterized in~\eqref{eqn:tboxtimes-defn}, but the diagram above shows that it is also the unique object such that
\[
\bar q^*(\cF \tboxtimes \cG) \cong \bar p^*(\cF \boxtimes \cG).
\]
The maps $\bar p$ and $\bar q$ are both $J$-torsors, and hence smooth of the same relative dimension (namely, $\dim J$).  It follows that $\cF \tboxtimes \cG$ is also characterized as the unique object such that
\[
\bar q^!(\cF \tboxtimes \cG) \cong \bar p^!(\cF \boxtimes \cG).
\]

Next, consider the commutative diagram
\[
\begin{tikzcd}
\Gr_p \times Z^\varpi \ar[d, "\bi_\thin \times \bi"'] & (\Loop_p\Gv/K_p) \times Z^\varpi \ar[l, "\bar p_p"'] \ar[r, "\bar q_p"] \ar[d, "h"] & (\Loop_p\Gv/K_p) \times^{J_p} Z \ar[d, "\bit"] \\
\Gr \times Z & (\Loop\Gv/K) \times \Gr \ar[l, "\bar p"'] \ar[r, "\bar q"] & (\Loop\Gv/K) \times^J Z
\end{tikzcd}
\]
in which the bottom line is the same as that in~\eqref{eqn:tppull1}, and the top row is its $p$-scaled version.  We clearly have
\begin{gather*}
\bar q_p^!(\bit^!(\cF \tboxtimes \cG)) \cong h^! \bar p^!(\cF \boxtimes \cG)
\cong \bar p_p^!(\bi_\thin^!\cF \boxtimes \bi^!\cG), \\
\bar q_p^*(\bit^*(\cF \tboxtimes \cG)) \cong h^* \bar p^*(\cF \boxtimes \cG)
\cong \bar p_p^*(\bi_\thin^*\cF \boxtimes \bi^*\cG),
\end{gather*}
and the result follows.
\end{proof}

\begin{lem}\label{lem:grpi-square}
For $\cF \in \Db_\Gm(\Gr,\bk)$ and $\cG \in \Db_{\Gm \ltimes \Loop^+\Gv}(\Gr,\bk)$, there is a natural commutative diagram
\[
\begin{tikzcd}
(\bi_\thin^!\cF) \star (\bi^!\cG) \ar[r] \ar[d] & (\bi_\thin^*\cF) \star (\bi^*\cG)   \\
\bi^!(\cF \star \cG) \ar[r] & \bi^*(\cF \star \cG) \ar[u]
\end{tikzcd}
\]
\end{lem}
\begin{proof}
Let $\bit: \Loop_p\Gv \times^{\Loop^+_p\Gv} \Gr^\varpi \hookrightarrow \Loop\Gv \times^{\Loop^+\Gv} \Gr$ be as in Lemma~\ref{lem:twistprod-pull}.  Let $m$ and $m_p$ be the multiplication maps involved in convolution (see~\eqref{eqn:conv-diag}), and let $Y = m^{-1}(\Gr^\varpi)$.  We have the following diagram, which the bottom square is cartesian, and all vertical maps are closed immersions:
\[
\begin{tikzcd}[row sep=small]
\Loop_p\Gv \times^{\Loop^+_p\Gv} \Gr^\varpi \ar[d, "h"] \ar[dd, bend right=30, "\bit"']  \ar[r, "m_p"] & \Gr^\varpi \ar[d, equal] \\
Y \ar[d, "k"] \ar[r, "m"] & \Gr^\varpi \ar[d, "\bi"] \\
\Loop\Gv \times^{\Loop^+\Gv} \Gr \ar[r, "m"] & \Gr
\end{tikzcd}
\]
By adjunction, we have natural maps
\[
h_!h^!k^!(\cF \tboxtimes \cG) \to k^!(\cF \tboxtimes \cG) \to k^*(\cF \tboxtimes \cG) \to h_*h^*k^*(\cF \tboxtimes \cG).
\]
Substitute $h^!k^! = \bit^!$ and $h^*k^* = \bit^*$, and also apply $m_*$ to obtain
\[
m_*h_!\bit^!(\cF \tboxtimes \cG) \to m_*k^!(\cF \tboxtimes \cG) \to m_*k^*(\cF \tboxtimes \cG) \to m_*h_*\bit^*(\cF \tboxtimes \cG).
\]
For the first and last terms, we have $m_*h_! = m_*h_* = m_{p*}$.  For the middle two terms, we have $m_*k^! \cong \bi^!m_*$ and $m_*k^* \cong \bi^* m_*$.  Combining these observations with Lemma~\ref{lem:twistprod-pull} yields the desired result.
\end{proof}

\begin{lem}\label{lem:conv-psm-compat}
Let $p\nu \in p\bX^{++}$, and suppose it is IW-clean. For any object $\cF \in  \Parity_{\Gm \ltimes \Loop^+\Gv}(\Gr,\bk)$, there is a natural isomorphism
\[
\rQ_\IW(j^{p+}_{p\nu!}\cL^p_\AS \star \bi^!\cF) \cong j^{p+}_{p\nu!}\cL^p_\AS \star \Psm_{p\nu}(\cF)
\]
in $\SmParity_\IW(\Gr^\varpi,\bk)$.
\end{lem}
\begin{proof}
It is enough to prove this lemma when $\cF$ is even (or odd): for a general parity sheaf, the desired isomorphism is then the sum of those for its even and odd components.

Assume for the rest of the proof that $\cF$ is even.  Then $\bi^!\cF$ is $!$-even, and $\bi^*\cF$ is $*$-even, so $\Psm_{p\nu}(\cF) \cong \rQ_{p\nu}(\bi^!\cF) \cong \rQ_{p\nu}(\bi^*\cF)$ is Smith-even. By Proposition~\ref{prop:smeven-lift}, there exists an even complex
\[
\cE \in \Parity_{\Gm \ltimes_{\widehat{p\nu}} \Iwul}(\Gr^\varpi,\bk)
\]
and a map
\[
\phi: \cE \to \bi^!\cF
\qquad\text{in}\qquad
\Db_{\Gm \ltimes_{\widehat{p\nu}} \Iwul}(\Gr^\varpi,\bk)
\]
such that $\rQ_{p\nu}(\cE) \to \rQ_{p\nu}(\bi^!\cF)$ is an isomorphism.

We claim that $\cE$ actually lies in $\Parity_{\Gm,\sph}(\Gr^\varpi,\bk)$.  If not, then $\cE$ would have an indecomposable summand $\cE'$ that is not spherical, i.e., whose support is not the closure of a $\Loop^+_p\Gv$-orbit.  But then $\rQ_{p\nu}(\cE)$ would also not be spherical, contradicting the fact that $\rQ_{p\nu}(\bi^!\cF)$ is spherical.

Since $\cE$ is spherical, we can upgrade it to an object
\[
\cE \in \Parity_{\Gm \ltimes \Loop^+_p\Gv}(\Gr^\varpi,\bk).
\]
Moreover, by~\cite[Lemma~A.11]{ar:agsr}, $\bHom_{\Db_{\Gm \ltimes \Loop^+_p\Gv}}(\cE,\bi^!\cF)$ is free over $\coh^\bullet_{\Gm \times \Gv}(\pt,\bk)$, and there is a natural isomorphism
\begin{multline*}
\bHom_{\Db_{\Gm \times_{\widehat{p\nu}} \Iwul}(\Gr^\varpi,\bk) }(\cE,\bi^!\cF)
\cong \\
\coh^\bullet_\Gm(\pt,\bk)_{\rS(p\nu)} \otimes_{\coh^\bullet_{\Gm \times \Gv}(\pt,\bk)} \bHom_{\Db_{\Gm \ltimes \Loop^+_p\Gv}(\Gr^\varpi,\bk)}(\cE,\bi^!\cF).
\end{multline*}
The map $\phi$ belongs to the left-hand side of this isomorphism.  Since the map $\rS(\widehat{p\nu})$ from~\eqref{eqn:s-hatnu-defn} is surjective, the natural map
\begin{equation}\label{eqn:cpc-bhom}
\bHom_{\Db_{\Gm \ltimes \Loop^+_p\Gv}(\Gr^\varpi,\bk)}(\cE,\bi^!\cF) \to
\bHom_{\Db_{\Gm \times_{\widehat{p\nu}} \Iwul}(\Gr^\varpi,\bk) }(\cE,\bi^!\cF)
\end{equation}
is surjective.  In particular, there exists a map
\[
\tilde\phi: \cE \to \bi^!\cF
\qquad\text{in}\qquad
\Db_{\Gm \ltimes \Loop^+_p\Gv}(\Gr^\varpi,\bk)
\]
whose image under~\eqref{eqn:cpc-bhom} is $\phi$.  Since the cone of $\phi$ is $\varpi$-perfect, so is the cone of $\tilde\phi$.

Since $\tilde\phi$ is a $\Gm \ltimes \Loop^+_p\Gv$-equivariant map, it makes sense to apply $j^{p+}_{p\nu!}\cL^p_\AS \star ({-})$ to it: we get a map
\[
j^{p+}_{p\nu!}\cL^p_\AS \star \tilde\phi: j^{p+}_{p\nu!}\cL^p_\AS \star \cE \to j^{p+}_{p\nu!}\cL^p_\AS \star \bi^!\cF.
\]
The cone of this map is $\varpi$-perfect (by Lemma~\ref{lem:conv-perfect}), so the map
\[
\rQ_\IW(j^{p+}_{p\nu!}\cL^p_\AS \star \tilde\phi): \rQ(j^{p+}_{p\nu!}\cL^p_\AS \star \cE) \simto \rQ(j^{p+}_{p\nu!}\cL^p_\AS \star \bi^!\cF)
\]
is an isomorphism.  By Lemma~\ref{lem:nuconv-smith}, the first term is naturally identified with $j^{p+}_{p\nu!}\cL^p_\AS \star \rQ(\cE)$.  To summarize, we have a diagram
\[
\begin{tikzcd}[column sep=huge]
\rQ(j^{p+}_{p\nu!}\cL^p_\AS \star \bi^!\cF)
 &
j^{p+}_{p\nu!}\cL^p_\AS \star \rQ(\cE) \ar[r, "j^{p+}_{p\nu!}\cL^p_\AS \star \rQ(\tilde\phi)", "\sim"'] \ar[l, "\rQ(j^{p+}_{p\nu!}\cL^p_\AS \star \tilde\phi)"', "\sim"] &
j^{p+}_{p\nu!}\cL^p_\AS \star \rQ(\bi^!\cF)
\end{tikzcd}
\]
The composition of these maps is the desired isomorphism in the statement of the lemma.  The claim that this isomorphism is well defined (i.e., independent of the choices of $\cE$, $\phi$, and $\tilde\phi$) and natural in $\cF$ follows by a routine argument using Proposition~\ref{prop:smeven-nat}.
\end{proof}

\begin{prop}\label{prop:main-summand}
Let $p\nu \in p\bX^{++}$, and suppose it is IW-clean.  For any object $\cF \in \Parity_{\Gm \ltimes \Loop^+\Gv}(\Gr,\bk)$, there is a natural commutative triangle
\[
\begin{tikzcd}[column sep=0pt,row sep=small]
j^{p+}_{p\nu!}\cL^p_\AS \star \Psm_{p\nu}(\cF) \ar[dr] \ar[rr, "\id"] && j^{p+}_{p\nu!}\cL^p_\AS \star \Psm_{p\nu}(\cF) \\
& \Psm(j^+_{p\nu!}\cL_\AS \star \cF) \ar[ur]
\end{tikzcd}
\]
in $\SmParity_\IW(\Gr^\varpi,\bk)$.
\end{prop}
In other words, this proposition says that $j^{p+}_{p\nu!}\cL^p_\AS \star \Psm_{p\nu}(\cF)$ can naturally be identified with a direct summand of $\Psm(j^+_{p\nu!}\cL_\AS \star \cF)$.
\begin{proof}
By Lemma~\ref{lem:grpi-square}, we have a commutative square
\[
\begin{tikzcd}
(\bi_\thin^! j^+_{p\nu!}\cL_\AS) \star (\bi^!\cF) \ar[r] \ar[d] &
  (\bi_\thin^* j^+_{p\nu!}\cL_\AS) \star (\bi^!\cF) \ar[r] &  
  (\bi_\thin^* j^+_{p\nu!}\cL_\AS) \star (\bi^*\cF) \\
\bi^!(j^+_{p\nu!}\cL_\AS \star \cF) \ar[rr] &&
  \bi^!(j^+_{p\nu!}\cL_\AS \star \cF) \ar[u]
\end{tikzcd}
\]
After applying $\rQ$, all horizontal maps become isomorphisms.  The objects in the bottom row are identified with $\Psm(j^+_{p\nu!}\cL_\AS \star \cF)$.  For the top row, since $(\Iwup \cdot \rL_{p\nu}) \cap \Gr^\varpi = \Iwupl \cdot \rL_{p\nu}$, we see by proper base change that
\[
\bi_\thin^* j^+_{p\nu!}\cL_\AS \cong j^{p+}_{p\nu!}\cL^p_\AS.
\]
Then Lemma~\ref{lem:conv-psm-compat} lets us identify the objects in the top row with $j^{p+}_{p\nu!}\cL^p_\AS \star \Psm_{p\nu}(\cF)$. 
\end{proof}

In the following statement, which makes the ``approximate'' statement Theorem~\ref{thm:main-intro} precise, we assume that $p$ is good for $G$.  By~\cite[\S4.1]{ss:cc}, this implies that $p$ is not a torsion prime for $\Gv$, so we may freely apply results from earlier in the paper.

\begin{thm}\label{thm:main}
Assume that $p$ is good for $G$.  The following diagram commutes up to natural isomorphism:
\[
\begin{tikzcd}[column sep=small]
\Parity_{\Gm \ltimes \Loop^+\Gv}(\Gr,\bk) \ar[r] \ar[d, "\cE^\sph_{(p-1)\varsigma} \star ({-})"'] &
  \Parity_{\Gm,\sph|p\varsigma}(\Gr,\bk) \ar[rr, "\Psm_{p\varsigma}"]
  &&
  \SmParity_{\sph|p\varsigma}(\Gr^\varpi,\bk) \ar[dd, "j^{p+}_{p\varsigma!}\cL^p_\AS \star ({-})"] \\
\Parity_{\Gm \ltimes \Loop^+\Gv}(\Gr,\bk) \ar[d, "j^+_{\varsigma!}\cL_\AS \star ({-})"'] \\
 \Parity_{\Gm,\IW}(\Gr,\bk) \ar[rrr, "\Psm_\IW"] &&&
  \SmParity_{\IW}(\Gr^\varpi,\bk)
\end{tikzcd}
\]
\end{thm}
\begin{proof}
By Lemma~\ref{lem:steinberg-clean}, for any $\cF \in \Parity_{\Gm \ltimes \Loop^+\Gv}(\Gr,\bk)$, we have
\[
j^+_{\varsigma!}\cL_\AS \star \cE^\sph_{(p-1)\varsigma} \star \cF \cong j^+_{p\varsigma!}\cL_\AS \star \cF,
\]
so to prove the theorem above, it is enough to show that the diagram 
\[
\begin{tikzcd}
\Parity_{\Gm,\sph|p\varsigma}(\Gr,\bk) \ar[r, "\Psm_{p\varsigma}"]
  \ar[d, "j^+_{p\varsigma!}\cL_\AS \star ({-})"']
  &
  \SmParity_{\sph|p\varsigma}(\Gr^\varpi,\bk) \ar[d, "j^{p+}_{p\varsigma!}\cL^p_\AS \star ({-})"] \\
 \Parity_{\Gm,\IW}(\Gr,\bk) \ar[r, "\Psm_\IW"] &
  \SmParity_{\IW}(\Gr^\varpi,\bk)
\end{tikzcd}
\]
commutes up to natural isomorphism.

By Proposition~\ref{prop:main-summand} (combined with Lemma~\ref{lem:steinberg-clean}), $j^{p+}_{p\varsigma!}\cL_\AS \star \Psm_{p\varsigma}(\cF)$ is naturally a direct summand of $\Psm(j^+_{p\varsigma!}\cL_\AS \star \cF)$, so there is some isomorphism
\[
\Psm(j^+_{p\varsigma!}\cL_\AS \star \cF)
\cong j^{p+}_{p\varsigma!}\cL_\AS \star \Psm_{p\varsigma}(\cF) \oplus \cG,
\]
and in the split Grothendieck group $\Ko(\SmParity_\IW(\Gr^\varpi,\bk))$, we have
\[
[\Psm(j^+_{p\varsigma!}\cL_\AS \star \cF)]
= [j^{p+}_{p\varsigma!}\cL_\AS \star \Psm_{p\varsigma}(\cF)] + [\cG].
\]
We wish to prove that $\cG = 0$.  By Lemma~\ref{lem:groth-smith-v1}, it is enough to show that the class $[\cG]$ vanishes after specializing the Grothendieck group at $v = 1$.  That is, we have to show that
\[
[\Psm(j^+_{p\varsigma!}\cL_\AS \star \cF)]|_{v=1}
= [j^{p+}_{p\varsigma!}\cL_\AS \star \Psm_{p\varsigma}(\cF)]|_{v=1},
\]
or in other words, that the diagram of (split, specialized) Grothendieck groups shown in Figure~\ref{fig:main-groth} commutes.  The rest of the proof is devoted to this.
Note that in the top back row of Figure~\ref{fig:main-groth}, we have implicitly identified
\[
\Ko(\Parity_{\Gm \ltimes \Loop^+\Gv}(\Gr,\bk))|_{v=1} \cong \Ko(\Parity_{\Gm, \sph|p\varsigma}(\Gr,\bk))|_{v=1},
\]
using an analogue of~\eqref{eqn:groth-isom}.

The ``front'' face of Figure~\eqref{fig:main-groth}, involving only $\Z[\bX]^W$ and multiplication by $\chi((p-1)\varsigma)$, obviously commutes.

The ``top'' face, involving $\Psm_{p\varsigma}$, can be seen to commute by comparing the definitions of the ``$\ch$'' map from Lemma~\ref{lem:groth-gr}\eqref{it:ggr-sph} with that in Lemma~\ref{lem:groth-varpi}\eqref{it:gv-sph}.  Similarly, ``bottom'' face commutes by comparing the ``$\ch$'' map from Lemma~\ref{lem:groth-gr}\eqref{it:ggr-iw} with that in Lemma~\ref{lem:groth-varpi}\eqref{it:gv-iw}.

The ``right-hand'' face commutes by Proposition~\ref{prop:groth-varpi-total} (together with Lemma~\ref{lem:groth-q}).

We now turn to the ``left-hand'' face.  Because we have assumed that $p$ is good for $G$, a result of Mautner--Riche~\cite[Corollary~1.6]{mr:etsps} says that every $\cE^\sph_\lambda$ is perverse, and hence that $\Ko(\Parity_{\Gm \ltimes \Loop^+\Gv}(\Gr,\bk))|_{v=1}$ admits a basis consisting of the classes of perverse parity sheaves.  For perverse parity sheaves, Lemma~\ref{lem:groth-gr}\eqref{it:ggr-sph} includes a compatibility with the geometric Satake equivalence, and this compatibility implies that the upper-left square in Figure~\ref{fig:main-groth} commutes.  Finally, the main results of~\cite{bgmrr} imply that the lower left square, involving $j^+_{\varsigma!}\cL_\AS \star ({-})$, commutes.

Since every arrow joining the back face of Figure~\ref{fig:main-groth} to the front face is an isomorphism, the preceding paragraphs imply that the back face also commutes, as desired.
\end{proof}

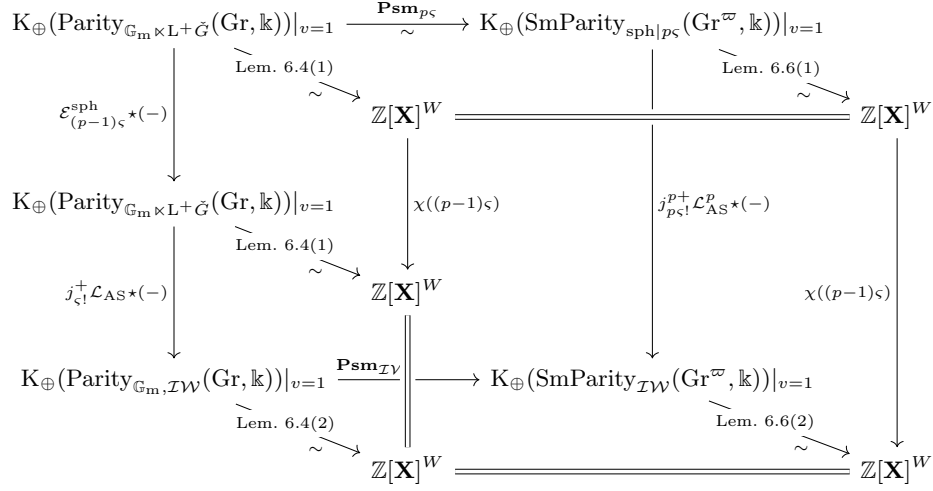
\begin{figure}
\[
\hbox{\small$
\begin{tikzcd}[column sep=tiny]
\Ko(\Parity_{\Gm \ltimes \Loop^+\Gv}(\Gr,\bk))|_{v=1} \ar[rr, "\Psm_{p\varsigma}", "\sim"'] \ar[dd, "\cE^\sph_{(p-1)\varsigma} \star ({-})"'] \ar[dr, "\text{Lem.~\ref{lem:groth-gr}\eqref{it:ggr-sph}}" {description,pos=0.4}, "\sim"' near end] && \Ko(\SmParity_{\sph|p\varsigma}(\Gr^\varpi,\bk))|_{v=1} \ar[dr, "\text{Lem.~\ref{lem:groth-varpi}\eqref{it:gv-sph}}" {description,pos=0.4}, "\sim"' near end]\ar[dddd, "j^{p+}_{p\varsigma!}\cL^p_\AS \star ({-})"] \\ 
& \Z[\bX]^W \ar[rr, equal, crossing over] \ar[dd, "\chi((p-1)\varsigma)"] && \Z[\bX]^W \ar[dddd, "\chi((p-1)\varsigma)"'] \\
\Ko(\Parity_{\Gm \ltimes \Loop^+\Gv}(\Gr,\bk))|_{v=1} \ar[dd, "j^+_{\varsigma!}\cL_\AS \star ({-})"'] \ar[dr, "\text{Lem.~\ref{lem:groth-gr}\eqref{it:gv-sph}}" {description,pos=0.4}, "\sim"' near end]\\ 
& \Z[\bX]^W \\
\Ko(\Parity_{\Gm,\IW}(\Gr,\bk))|_{v=1} \ar[rr, "\Psm_\IW" near start] \ar[dr, "\text{Lem.~\ref{lem:groth-gr}\eqref{it:ggr-iw}}" {description,pos=0.4}, "\sim"' near end] && \Ko(\SmParity_\IW(\Gr^\varpi,\bk))|_{v=1} \ar[dr, "\text{Lem.~\ref{lem:groth-varpi}\eqref{it:gv-iw}}" {description,pos=0.4}, "\sim"' near end] \\
& \Z[\bX]^W \ar[uu, equal, crossing over] \ar[rr, equal] && \Z[\bX]^W
\end{tikzcd}$}
\]
\caption{Diagram for the proof of Theorem~\ref{thm:main}}\label{fig:main-groth}
\end{figure}

\subsection{Applications}
\label{ss:applications}

For $\lambda \in \bX^+$, we define the \emph{Steinberg quotient character} $t(\lambda) \in \Z[\bX]^W$ by
\[
t(\lambda) = \frac{ \ch T(\lambda + (p-1)\varsigma)}{ \ch T((p-1)\varsigma)}.
\]
Let us check that this definition is independent of the choice of $\varsigma$.  If $\varsigma'$ is another weight satisfying~\eqref{eqn:varsig-exist}, then the weight $(p-1)(\varsigma - \varsigma')$ is orthogonal to all coroots, and so $G$ admits $1$-dimensional tilting representation $\bk_{(p-1)(\varsigma - \varsigma')}$.  For any $\lambda \in \bX^+$, we have
\[
T(\lambda + (p-1)\varsigma) \cong T(\lambda + (p-1)\varsigma') \otimes \bk_{(p-1)(\varsigma - \varsigma')},
\]
and it follows from this that $t(\lambda)$ is independent of $\varsigma$.  In particular, the definition above agrees with those in Section~\ref{ss:over-stquot} and in~\cite{sob1,sob2}.  The following proposition is one of the main motivations for the present paper.

\begin{thm}\label{thm:stquot}
Assume that $p$ is good for $G$. For $\lambda \in \bX^+$, we have
\[
\ch ([\cE^{\varpi,\sph}_\lambda]) = t(\lambda),
\]
where the left-hand side is defined as in Lemma~\ref{lem:groth-varpi}\eqref{it:gv-sph}.
\end{thm}
\begin{proof}
By Corollary~\ref{cor:gv-indecomp}, we have $j^{p+}_{p\varsigma!}\cL^p_\AS[\langle \varsigma, 2\check\rho\rangle] \star \cE^{\varpi,\sph}_\lambda \cong \cE^{\varpi,\IW,\bo}_{\lambda + p \varsigma}$.  By the right-hand face of Figure~\ref{fig:main-groth}, we see that
\[
\chi((p-1)\varsigma) \ch ([\cE^{\varpi,\sph}_\lambda]) = \ch([\cE^{\varpi,\IW,\bo}_{\lambda + p \varsigma}]).
\]
On the other hand, by~\cite[Theorem~1.2]{rw:st}, we have
\[
\Psm_\IW(\cE^{\IW,\bo}_{\lambda + p\varsigma}) \cong \cE^{\varpi, \IW,\bo}_{\lambda + p\varsigma},
\]
while Lemma~~\ref{lem:groth-gr}\eqref{it:ggr-iw} tells us that
\[
\ch([\cE^{\IW,\bo}_{\lambda + p\varsigma}]) = \ch T(\lambda + (p-1)\varsigma).
\]
Combining this with the bottom face of Figure~\ref{fig:main-groth}, we see that
\[
\ch ([\cE^{\varpi,\sph}_\lambda]) = \ch T(\lambda + (p-1)\varsigma)/ \chi((p-1)\varsigma) = t(\lambda).\qedhere
\]
\end{proof}

Theorem~\ref{thm:stquot} is closely related to the monotonicity property of Steinberg quotients that was mentioned in Section~\ref{ss:over-stquot}.  Let us review what this monotonicity property says.  Write
\begin{equation}\label{eqn:t-expand}
t(\lambda) = \sum b_{\mu,\lambda}s(\mu),
\end{equation}
where $b_{\mu,\lambda} \in \Z$.  In~\cite{sob1,sob2}, the second author established the following:

\begin{thm}[\cite{sob1,sob2}]\label{thm:mono}
For $\lambda, \mu \in \bX^+$, we have:
\begin{enumerate}
\item $b_{\lambda,\lambda}=1$.\label{it:mono-open}
\item If $b_{\mu,\lambda} \ge 1$, then $(\mu-\rho) \uparrow (\lambda-\rho)$.\label{it:mono-closure}
\item If $(\mu-\rho) \uparrow (\mu^{\prime}-\rho)$, then $b_{\mu,\lambda} \ge b_{\mu^{\prime},\lambda}$.\label{it:mono-mono}
\end{enumerate}
\end{thm}

(Here, $\uparrow$ is the partial order on $\bX$ defined in~\cite[\S II.6.4]{jan:rag}.)  When $p$ is good, Theorem~\ref{thm:mono} can be deduced as a consequence of Theorem~\ref{thm:stquot}, as follows

\begin{proof}[Proof sketch for Theorem~\ref{thm:mono} when $p$ is good]
Comparing~\eqref{eqn:t-expand} with Theorem~\ref{thm:stquot}, we have
\[
b_{\mu, \lambda} = \sum_{i \in \Z} \rank \coh^i(\cE^{\varpi,\sph}_\lambda|_{\Gr^{\varpi}_\mu}).
\]
Part~\eqref{it:mono-open} of Theorem~\ref{thm:mono} is immediate from the definition of $\cE^{\varpi,\sph}_\lambda$.

Next, it can be shown using the methods of~\cite{rw:st} that $\Gr^{\varpi}_\mu \subset \overline{\Gr^{\varpi}_\lambda}$ if and only if $(\mu - \varsigma) \uparrow (\lambda - \varsigma)$.  Part~\eqref{it:mono-closure} of Theorem~\ref{thm:mono} then holds because $\cE^{\varpi,\sph}_\lambda$ is supported on $\overline{\Gr^{\varpi}_\lambda}$.

Finally, part~\eqref{it:mono-mono} of Theorem~\ref{thm:mono} is a general property of  parity sheaves that are equivariant with respect to a torus action satisfying certain assumptions, explained in~\cite{fw:psmg}.  In more detail, by~\cite[Proposition~7.1]{fw:psmg}, the action of the torus $\Tv$ on $\Gr^\varpi$ satisfies the appropriate assumptions, so by~\cite[Theorems~5.11 and 6.10]{fw:psmg}, objects in $\Parity_{\Loop^+_p\Gv}(\Gr^\varpi,\bk)$ can be modeled by the theory of ``sheaves on moment graphs.''  Then the argument of~\cite[Theorem~3.6]{bm:mgic} can be adapted to show the inequality in Theorem~\ref{thm:mono}\eqref{it:mono-mono}.
\end{proof}

Finally, we will describe decomposition multiplicities for the top arrow in Theorem~\ref{thm:main} in representation-theoretic terms.

\begin{lem}\label{lem:proj-tilt}
For $\lambda \in \bX^+$, every indecomposable summand of $T(\lambda) \otimes T((p-1)\varsigma)$ is of the form $T(\mu + (p-1)\varsigma)$ for some $\mu \in \bX^+$.
\end{lem}
\begin{proof}
The tensor product of two tilting modules for $G$ is a tilting module.  Since the Steinberg module $T((p-1)\varsigma)$ is projective over the first Frobenius kernel $G_1$, the tensor product $T(\lambda) \otimes T((p-1)\varsigma)$ is also projective over $G_1$, and therefore its summands are as well.  The tilting modules that are projective over $G_1$ are those with highest weights of the form $\mu + (p-1)\varsigma$.
\end{proof}

Lemma~\ref{lem:proj-tilt} implies that there are nonnegative integers $d_{\mu,\lambda}$ characterized by the property that
\begin{equation}\label{eqn:dmu-defn}
T(\lambda) \otimes T((p-1)\varsigma) \cong \bigoplus_{\mu \in \bX^+} T(\mu + (p-1)\varsigma)^{\oplus d_{\mu,\lambda}}.
\end{equation}

\begin{thm}\label{thm:dmu}
Assume that $p$ is good for $G$, and let $d_{\mu,\lambda}$ be as in~\eqref{eqn:dmu-defn}.  We have
\[
\Psm_{p\varsigma}(\cE^\sph_\lambda) \cong \bigoplus_{\mu \in \bX^+} (\cE^{\varpi,\sph}_{\Sm,\mu})^{\oplus d_{\mu,\lambda}}.
\]
\end{thm}
\begin{proof}
Since the right-hand vertical arrow in Theorem~\ref{thm:main} sends distinct indecomposable objects to distinct indecomposable objects (Corollary~\ref{cor:gv-indecomp}), to prove the theorem, it is enough to show that
\begin{equation}\label{eqn:dmu-comp1}
j^{p+}_{p\varsigma!}\cL_\AS[\langle \varsigma, 2\check\rho\rangle] \star \Psm_{p\varsigma}(\cE^\sph_\lambda) \cong
\bigoplus_{\mu \in \bX^+} (\cE^{\varpi,\IW,\bo}_{\Sm,\mu+p\varsigma})^{\oplus d_{\mu,\lambda}}.
\end{equation}

By the geometric Satake equivalence applied to~\eqref{eqn:dmu-defn}, we have
\begin{equation}\label{eqn:dmu-comp2}
\cE^\sph_{(p-1)\varsigma} \star \cE^\sph_\lambda \cong \bigoplus_{\mu \in \bX^+} (\cE^\sph_{\mu + (p-1)\varsigma})^{\oplus d_{\mu,\lambda}}.
\end{equation}
Recall that the lower-left and bottom arrows in Theorem~\ref{thm:main} are both fully faithful on perverse parity sheaves (by~\cite{bgmrr, rw:st}).  In particular, we have
\[
\Psm_{\IW}(j^+_{\varsigma!}\cL_\AS[\langle \varsigma, 2\check\rho\rangle] \star \cE^\sph_\nu) = \Psm_\IW(\cE^{\IW,\bo}_{\nu+\varsigma}) \cong \cE^{\varpi,\IW,\bo}_{\Sm,\nu+\varsigma}.
\]
Applying this to~\eqref{eqn:dmu-comp2}, we obtain
\[
\Psm_{\IW}(j^+_{\varsigma!}\cL_\AS[\langle \varsigma, 2\check\rho\rangle] \star \cE^\sph_{(p-1)\varsigma} \star \cE^\sph_\lambda) \cong
\bigoplus_{\mu \in \bX^+} (\cE^{\varpi,\IW,\bo}_{\Sm,\mu + p\varsigma})^{\oplus d_{\mu,\lambda}}.
\]
By Theorem~\ref{thm:main}, this implies~\eqref{eqn:dmu-comp1}.
\end{proof}

\begin{rmk}
Theorems~\ref{thm:stquot}, \ref{thm:mono}, and~\ref{thm:dmu} do not actually require the full strength of Theorem~\ref{thm:main}: they can all be deduced from the Grothendieck-group shadow of Theorem~\ref{thm:main}, i.e., from the commutative diagram in Figure~\ref{fig:main-groth}.  Moreover, the proof of that commutative diagram relies only on the results in Section~\ref{sec:groth}; in particular, it does not require Proposition~\ref{prop:main-summand} or other results from Section~\ref{ss:compat}.
\end{rmk}

\appendix
\section{Background on Smith--Treumann theory}
\label{app:st}

Let $\F$ be an algebraically closed field of characteristic $\ell \ne p$, and let
\[
\varpi = \{ \zeta \in \F^\times \mid \zeta ^{p} = 1 \}
\]
be the group of $p$-th roots of unity in $\F$.  We can also regard $\varpi$ as a (discrete) algebraic subgroup of the multiplicative group $\Gm$ over $\F$.  In this section, we review the foundations of \emph{Smith theory} for sheaves on varieties over $\F$ with a $\Gm$-action.  

\subsection{Preliminaries on \texorpdfstring{$\varpi$}{varpi}-modules}
\label{ss:pi-prelim}

For $g \in \varpi$, let $t_g$ denote the corresponding basis element in the group ring $\bk[\varpi]$.  Thus, a typical element of $\bk[\varpi]$ can be written as $\sum_{g \in \varpi} a_g t_g$, with $a_g \in \bk$.  Consider the $\bk$-linear maps
\begin{align*}
\epsilon &: \bk[\varpi] \to \bk &  \epsilon(t_g) &= 1 \quad\text{for all $g \in \varpi$,} \\
\iota&: \bk \to \bk[\varpi] & \iota(1) &= \textstyle\sum_{g \in \varpi} t_g.
\end{align*}
If we regard $\bk$ as the trivial $\bk[\varpi]$-module, then these are maps of $\bk[\varpi]$-modules.  Choose a generator $\zeta \in \varpi$.  Then we can form the four-term exact sequence
\[
0 \to \bk \xrightarrow{\iota} \bk[\varpi] \xrightarrow{\zeta - 1} \bk[\varpi] \xrightarrow{\epsilon} \bk \to 0,
\]
and this determines an element $c \in \Ext^2_{\bk[\varpi]}(\bk,\bk)$.  

We will now reformulate the preceding construction in a way that makes it independent of the choice of generator $\zeta$.  The group $\varpi$ is, of course, a cyclic group of order $p$ and thus a free module over $\Z/p\Z$, but it does not have a preferred generator.  We define the \emph{Tate module} over $\bk$ to be the $\bk$-vector space
\[
\bk(1) = \bk \otimes_{\Z/p\Z} \varpi.
\]
Again, this is a $1$-dimensional $\bk$-vector space, but with no preferred basis element.  Nevertheless, if we choose a generator $\zeta \in \varpi$, then $1 \otimes \zeta$ is a $\bk$-basis element for $\bk(1)$.  Define a map
\[
\iota_\zeta: \bk(1) \to \bk[\varpi] \qquad\text{by}\qquad \iota_\zeta(1 \otimes \zeta) = \textstyle\sum_{g \in \varpi} t_g.
\]
Now suppose $\xi$ is another generator of $\varpi$.  We must have $\xi = \zeta^r$ for some $r \in \{1, 2, \ldots, p-1\}$.  We have the following commutative diagram with exact rows:
\[
\begin{tikzcd}[column sep=small]
0 \ar[r] & \bk(1) \ar[d, equal] \ar[r, "\iota_\xi"] &
\bk[\varpi] \ar[rrr, "\xi - 1"] \ar[d, "1 + \zeta + \cdots + \zeta^{r-1}"] &&& \bk[\varpi] \ar[d, equal] \ar[r, "\epsilon"] & \bk \ar[d, equal] \ar[r] & 0 \\
0 \ar[r] & \bk(1) \ar[r, "\iota_\zeta"] & \bk[\varpi] \ar[rrr, "\zeta - 1"] &&& \bk[\varpi] \ar[r, "\epsilon"] & \bk \ar[r] & 0
\end{tikzcd}
\]
Either row of this diagram determines an element
\begin{equation}\label{eqn:can2-defn}
\can^2 \in \Ext^2_{\bk[\varpi]}(\bk, \bk(1)).
\end{equation}
that is canonical in the sense that it is independent of the choice of generator $\zeta \in \varpi$.  (The superscript ``$2$'' in ``$\can^2$'' is just a reminder that this element lies in an $\Ext^2$ group.  It is included for consistency of notation with~\cite{rw:st}.)

\subsection{Equivariant cohomology interpretation}

We will now describe another construction of $\can^2$ following~\cite{rw:st}.  Let
\[
\pt = \Spec \F.
\]
It is well known that the $\Gm$-equivariant cohomology of $\pt$ is given by
\[
\coh^n_\Gm(\pt,\bk) \cong
\begin{cases}
\bk(-n/2) & \text{if $n \ge 0$ and $n$ is even,} \\
0 & \text{otherwise.}
\end{cases}
\]
In particular, we have a canonical identification $\coh^2_\Gm(\pt,\bk)(1) = \bk$.  The element corresponding to $1 \in \bk$ is denoted by
\begin{equation}\label{eqn:can2-defn2}
\can^2 \in \coh^2_\Gm(\pt,\bk)(1).
\end{equation}
Let us justify this notation.  By restricting to the subgroup $\varpi \subset \Gm$, we get a map $\coh^\bullet_\Gm(\pt,\bk) \to \coh^\bullet_\varpi(\pt,\bk)$.  The latter is identified with $\Ext^\bullet_{\bk[\varpi]}(\bk,\bk)$ (cf.~\cite[Eq.~(3.6)]{rw:st}).  According to~\cite[Remark~3.10]{rw:st}, the image of~\eqref{eqn:can2-defn2} under the map
\[
\coh^2_\Gm(\pt,\bk)(1) \to \coh^2_\varpi(\pt,\bk)(1) \cong \Ext^2_{\bk[\varpi]}(\bk,\bk(1))
\]
is equal to~\eqref{eqn:can2-defn}.

\subsection{Definition of the Smith category}
\label{ss:smith-defn}

The derived category $\Db_\varpi(\pt,\bk)$ of $\pi$-equivariant constructible sheaves on a point contains as a subcategory the abelian category $\Sh_\varpi(\Spec \F,\bk)$ of constructible $\varpi$-equivariant $\bk$-sheaves, or equivalently of finite-dimensional $\bk[\varpi]$-modules.  In particular, one can consider the free module $\bk[\varpi]$ as an object of $\Sh_\varpi(\Spec \F,\bk)$.  Denote by
\[
\Db_\varpi(\Spec \F,\bk)_\perf \subset \Db_\varpi(\Spec \F,\bk)
\]
the full subcategory generated by the free module $\bk[\varpi]$.  Objects of the category $\Db_\varpi(\Spec \F,\bk)_\perf$ are called \emph{$\varpi$-perfect complexes}.  As an example, the cone of
\begin{equation}\label{eqn:can2-point}
\can^2: \bk[-2](-1) \to \bk
\end{equation}
is $\varpi$-perfect, as it is represented by the chain complex $\cdots \to \bk[\varpi] \xrightarrow{\zeta - 1} \bk[\varpi] \to \cdots$.

Suppose now that $X$ is an (ind-)scheme over $\F$ equipped with an action of $\Gm$, so that we may consider the $\Gm$-equivariant derived category $\Db_\Gm(X,\bk)$.  Assume that the subgroup $\varpi \subset \Gm$ acts trivially on $X$.  Then, for any object $\cF \in \Db_\Gm(X,\bk)$ and any geometric point $\bar z$ of $X$, the stalk $\cF_{\bar z}$ can be regarded as an object of the equivariant derived category $\Db_\varpi(\bar z, \bk)$.  We say that $\cF$ is \emph{$\varpi$-perfect} if every stalk $\cF_{\bar z}$ is $\varpi$-perfect in the sense of the preceding paragraph.  The $\varpi$-perfect objects form a full triangulated subcategory
\[
\Db_\Gm(X,\bk)_\perf \subset \Db_\Gm(X,\bk).
\]
We define the \emph{Smith category} of $X$ to be the Verdier quotient
\[
\Sm(X,\bk) = \Db_\Gm(X,\bk)/ \Db_\Gm(X,\bk)_\perf.
\]
The quotient functor is denoted
\[
\rQ: \Db_\Gm(X,\bk) \to \Sm(X,\bk).
\]

The map~\eqref{eqn:can2-point} gives rise to a map of constant sheaves $\can^2: \underline{\bk}_X[-2](-1) \to \underline{\bk}_X$ whose cone is $\varpi$-perfect.  Therefore, in the Smith category, we have an \emph{isomorphism}
\begin{equation}\label{eqn:can2-sheaf}
\can^2: \underline{\bk}_X[-2](-1) \cong \underline{\bk}_X.
\end{equation}
In the case where $X = \Spec \F$, the following lemma expands on this observation. For a proof, see~\cite[Lemma~3.7]{rw:st}.

\begin{lem}\label{lem:sm-point-endo}
In $\Sm(\pt,\bk)$, we have
\[
\Hom(\underline{\bk}_\pt, \underline{\bk}_\pt[n])
\cong
\begin{cases}
\bk(-n/2) & \text{if $n$ is even,} \\
0 & \text{if $n$ is odd.}
\end{cases}
\]
\end{lem}

The next lemma says that all the usual sheaf functors send $\varpi$-perfect complexes to $\varpi$-perfect complexes.

\begin{lem}\label{lem:sheaf-perfect}
Let $X$ and $Y$ be schemes over $\F$ equipped with $\Gm$-actions such that $\varpi \subset \Gm$ acts trivially.
\begin{enumerate}
\item The Verdier duality functor $\D$ on $X$ sends $\varpi$-perfect complexes to $\varpi$-perfect complexes.
\item If at least one of $\cF$ or $\cG$ is $\varpi$-perfect, then $\cF \otimes^L \cG$ and $\cRHom(\cF,\cG)$ are $\varpi$-perfect.
\item Let $f: X \to Y$ be a $\Gm$-equivariant map of schemes.  Then the functors $f^*$, $f^!$, $f_*$, and $f_!$ send $\varpi$-perfect complexes to $\varpi$-perfect complexes.
\end{enumerate}
\end{lem}
\begin{proof}
For $f^*$, this claim is obvious from the definitions, and for $f_*$, this is~\cite[Lemma~3.6]{rw:st}. For the remaining functors, we will treat some special cases first.

\textit{Case 1.  $f^!$ for $f$ a closed immersion.}  Let $j: U \hookrightarrow Y$ be the complementary open immersion.  For any $\varpi$-perfect complex $\cF$, we have the distinguished triangle $f_*f^!\cF \to \cF \to j_*j^*\cF \to$.  Since the second and third terms are $\varpi$-perfect, so is $f_*f^!\cF$.  Since the stalks of $f^!\cF$ coincide with (some of) those of $f_*f^!\cF$, $f^!\cF$ is $\varpi$-perfect as well.

\textit{Case 2. $\D$ for $X = \Spec \F$.}
The Verdier dual of the free module $\bk[\varpi]$ is $\Hom_\bk(\bk[\varpi],\bk)$, which is again isomorphic (as a $\bk[\varpi]$-module) to $\bk[\varpi]$.

\textit{Case 3. $\D$ in general.}
Let $\cF$ be a $\varpi$-perfect complex on $X$.  We must show that $\D_X\cF$ is $\varpi$-perfect, i.e., that for every geometric point $i: \bar x \to X$, the object $i^*\D_X\cF$ is $\varpi$-perfect.  But the latter is isomorphic to $\D_{\Spec F} (i^!\cF)$, and this is $\varpi$-perfect by Cases~1 and~2.

\textit{Case 4. $f^!$ and $f_!$ in general.} These follow from Case~3 and the formulas $f^! \cong \D \circ f^* \circ \D$ and $f_! \cong \D \circ f_! \circ \D$.

\textit{Case 5. $\otimes^L$ for $X = \Spec \F$.}
If $M$ is any finitely-generated $\bk[\varpi]$-module, the tensor product $\bk[\varpi] \otimes_\bk M$ is a perfect $\bk[\varpi]$-module.  The claim follows.

\textit{Case 6. $\otimes^L$ in general.}  If $\cF$ is $\varpi$-perfect, then for any geometric point $i: \bar x \to X$, the stalk $(\cF \otimes^L \cG)_{\bar x} \cong \cF_{\bar x} \otimes^L \cG_{\bar x}$ is perfect by Case~5.

\textit{Case 7. $\cRHom$ in general.}  This follows from Cases~3 and~6 using the formula $\cRHom(\cF,\cG) \cong \D(\cF \otimes^L \D(\cG))$.
\end{proof}

As a consequence of Lemma~\ref{lem:sheaf-perfect}, all the usual sheaf functors induce corresponding functors on the Smith category.  For instance, it makes sense to take a (derived) tensor product with the isomorphism~\eqref{eqn:can2-sheaf}.  We obtain a natural isomorphism of functors 
\[
\id \cong [2](1)
\qquad\text{in $\Sm(X,\bk)$.}
\]

\subsection{Smith--Treumann localization}
\label{ss:st-defn}

In this subsection, let $X$ be a scheme or ind-scheme over $\F$ with an action of $\Gm$, but we do \emph{not} assume that $\varpi \subset \Gm$ acts trivially.  Let
\[
X^\varpi
\]
be the closed subscheme of fixed points for the subgroup $\varpi \subset \Gm$, and let $i: X^\varpi \hookrightarrow X$ be the inclusion map.  By~\cite[Lemma~3.5]{rw:st}, for any $\cF \in \Db_\Gm(X,\bk)$, the cone of the natural map $i^!\cF \to i^*\cF$ is $\varpi$-perfect.  Thus, this natural map becomes an isomorphism after passage to $\Sm(X^\varpi,\bk)$.  The \emph{Smith--Treumann localization functor} is the functor
\[
\Psm: \Db_\Gm(X,\bk) \to \Sm(X^\varpi,\bk)
\]
induced by either $i^!$ or $i^*$.

\subsection{The equivariant Smith category}
\label{ss:smith-equiv}

Let $H$ be an algebraic group over $\F$ equipped with an action of $\Gm$ by group automorphisms, so that it makes sense to form the semidirect product $\Gm \ltimes H$.  If $X$ is a scheme over $\F$ with an action of $\Gm \ltimes H$ such that $\varpi \subset \Gm$ acts trivially, then we say that an object $\cF \in \Db_{\Gm \ltimes H}(X,\bk)$ is \emph{$\varpi$-perfect} if its image in $\Db_\Gm(X,\bk)$ (after forgetting $H$-equivariance) is $\varpi$-perfect.  The $\varpi$-perfect objects form a full subcategory $\Db_{\Gm \ltimes H}(X,\bk)_\perf \subset \Db_{\Gm \ltimes H}(X,\bk)$.  We define the \emph{equivariant Smith category} to be the quotient
\[
\Sm_H(X,\bk) = \Db_{\Gm \ltimes H}(X,\bk)/\Db_{\Gm \ltimes H}(X,\bk)_\perf.
\]

The reader should beware that if $H$ is a \emph{nonunipotent} group, then $\Sm_H(X,\bk)$ is rather unwieldy: for instance, the constant sheaf $\underline{\bk}_X$ typically has an infinite-dimensional endomorphism ring, and the whole category can fail to be Krull--Schmidt.  Thus, $\Sm_H(X,\bk)$ is not well suited for computations or classification questions, and indeed, this version of the Smith category does not appear in~\cite{treu:smgha, ll:psst, rw:st}.  Nevertheless, it is occasionally useful in the present paper as a technical tool.

\subsection{The (twisted-)equivariant Smith category for a unipotent group}
\label{ss:tweq}

The situation for a \emph{unipotent} group is much better.  In this subsection, we let  $V$ be a connected unipotent algebraic group over $\F$ equipped with an action of $\Gm$ by group automorphisms.  If we let $\mu: V \times V \to V$ be the group operation on $V$, then $\mu$ is $\Gm$-equivariant.

Let $\cX$ be a $\Gm$-equivariant multiplicative $\bk$-local system on $V$.  This means that $\cX$ is a $\Gm$-equivariant local system on $V$ together with an isomorphism
\[
\mu^*\cX \cong \cX \boxtimes \cX
\]
satisfying an appropriate associativity condition.  The constant sheaf $\underline{\bk}_V$ canonically has the structure of a $\Gm$-equivariant multiplicative local system, but there may be others as well.  

Now let $X$ be a scheme over $\F$, equipped with an action of $\Gm \ltimes V$.  Let $a: V \times X \to X$ be the action of $V$ on $X$.  Let
\[
\Db_{\Gm \ltimes V, \cX}(X,\bk)
\]
denote the $(\Gm \ltimes V, \cX)$-equivariant derived category of $X$.  By definition, an object of this category is a pair $(\cF, \beta)$, where $\cF \in \Db_\Gm(X,\bk)$, and $\beta$ is an isomorphism $\beta: a^*\cF \cong \cX \boxtimes \cF$ satisfying an appropriate cocycle condition.  There is a forgetful functor
\[
\For: \Db_{\Gm \ltimes V,\cX}(X,\bk) \to \Db_\Gm(X,\bk)
\]
that identifies $\Db_{\Gm \ltimes V, \cX}(X,\bk)$ with a full subcategory of $\Db_\Gm(X,\bk)$.  (See~\cite[Proposition~A.2]{ar} for a proof of this without the $\Gm$-action; the same arguments apply here.)  Here are some observations about this category:
\begin{itemize}
\item The full subcategory $\Db_{\Gm \ltimes V,\cX}(X,\bk) \subset \Db_\Gm(X,\bk)$ is stable under truncation with respect to both the natural and perverse $t$-structures.  (To see this, observe that an isomorphism $\beta: a^*\cF \cong \cX \boxtimes \cF$ induces a similar isomorphism after taking ordinary or perverse cohomology.)
\item The full subcategory $\Db_{\Gm \ltimes V,\cX}(X,\bk) \subset \Db_\Gm(X,\bk)$ is thick, i.e., stable under direct summands.  (This follows from the criterion in~\cite[Proposition~A.5]{ar}.)
\end{itemize}

We would now like to define the \emph{$(V,\cX)$-equivariant Smith category} of $X$.  There are two reasonable approaches to this.  The first is to define
\[
\Sm_{V,\cX}(X,\bk) =
\begin{array}{c}
\text{the full triangulated subcategory of $\Sm(X,\bk)$ generated by}\\
\text{the image of $\Db_{\Gm \ltimes V,\cX}(X,\bk) \to \Db_\Gm(X,\bk) \to \Sm(X,\bk)$}
\end{array}
\]

The second approach is this: let $\Db_{\Gm \ltimes V,\cX}(X,\bk)_\perf$ be the full subcategory of $\Db_{\Gm \ltimes V,\cX}$ consisting of objects that are $\varpi$-perfect when regarded as objects of $\Db_\Gm(X,\bk)$.  Then one can form the Verdier quotient category
\[
\Db_{\Gm \ltimes V,\cX}(X,\bk)/ \Db_{\Gm \ltimes V,\cX}(X,\bk)_\perf.
\]
In~\cite[\S6.1]{rw:st}, this quotient is taken to be the definition of the $(V,\cX)$-equivariant Smith category.  The following lemma says that the two approaches agree.

\begin{lem}\label{lem:smith-2defns}
Let $V$, $\cX$, and $X$ be as above.  The natural functor
\[
\Db_{\Gm \ltimes V,\cX}(X,\bk)/ \Db_{\Gm \ltimes V,\cX}(X,\bk)_\perf \to \Sm_{V,\cX}(X,\bk)
\]
is an equivalence of categories.
\end{lem}
\begin{proof}
The image of this functor generates $\Sm_{V,\cX}(X,\bk)$ by the definition of the latter, so it is enough to show that the functor is fully faithful.  To this, we will use the criterion from~\cite[Proposition~1.6.10]{ks}.  Suppose we have a morphism
\[
\phi: \cF \to \cH
\]
in $\Db_\Gm(X,\bk)$, where $\cF$ is $(\Gm \ltimes V,\cX)$-equivariant, and $\cH$ is $\varpi$-perfect.  To apply~\cite[Proposition~1.6.10]{ks}, we must show that $\phi$ factors through an object $\cH'$ that is both $(\Gm \ltimes V,\cX)$-equivariant and $\varpi$-perfect.

By~\cite[Lemma~A.3]{ar}, the forgetful functor $\For: \Db_{\Gm \ltimes V,\cX}(Y,\bk) \to \Db_\Gm(Y,\bk)$ has a right adjoint, which we denote by $\Av_*$.  By adjunction, $\phi$ must factor through $\Av_*\cH$.  Moreover, the explicit formula for $\Av_*$ in~\cite[Lemma~A.3]{ar}, together with Lemma~\ref{lem:sheaf-perfect}, shows that $\Av_*\cH$ is $\varpi$-perfect.
\end{proof}

\subsection{The Smith category of an orbit}
\label{ss:smith-orbit}

Let $V$ and $\cX$ be as in Section~\ref{ss:tweq}.  Let $X$ be a variety over $\F$ with an action of $\Gm \ltimes V$.  In this subsection, we assume that:
\begin{itemize}
\item The group $V$ acts transitively on $X$.
\item The group $\Gm$ has a fixed point on $X$.
\end{itemize}

\begin{lem}\label{lem:smith-orbit}
Let $V$, $\cX$, and $X$ be as above.  Then exactly one of the following holds:
\begin{enumerate}
\item There are no nonzero $(\Gm \ltimes V, \cX)$-equivariant complexes on $X$, and we have $\Db_{\Gm \ltimes V,\cX}(X,\bk) = \Sm_{V,\cX}(X,\bk) = 0$.\label{it:tweq-zero}
\item There is up to isomorphism a unique irreducible $(\Gm \ltimes V,\cX)$-equivariant local system $\cL_X$ on $X$.  Moreover, there are equivalences of categories\label{it:tweq-nonzero}
\begin{align*}
\Db_{\Gm \ltimes V,\cX}(X,\bk) &\cong \Db_{\Gm}(\pt,\bk), \\
\Sm_{V,\cX}(X,\bk) &\cong \Sm(\pt,\bk)
\end{align*}
that send $\cL_X$ to $\underline{\bk}_\pt$.
\end{enumerate}
\end{lem}
\begin{proof}
Let $x$ be a fixed point for the $\Gm$-action on $X$, and let $V_x$ be its (reduced) stabilizer in $V$.  Let $\cX'$ be the restriction of $\cX$ to $V_x$.  Then there is an equivalence of categories
\begin{equation}\label{eqn:tweq-point}
\Db_{\Gm \ltimes V, \cX}(X,\bk) \cong \Db_{Gm \ltimes V_x,\cX'}(\pt,\bk)
\end{equation}
that sends any object $\cF \in \Db_{\Gm \ltimes V, \cX}(X,\bk)$ to its stalk $\cF_x$.

The right-hand side of~\eqref{eqn:tweq-point} is a full triangulated subcategory of $\Db_\Gm(\pt,\bk)$.  Since $\Db_{Gm \ltimes V_x,\cX'}(\pt,\bk)$ is closed under truncation and direct summands, if it contains any nonzero object at all, it must contain the constant sheaf $\underline{\bk}_\pt$.  The lemma follows from this.
\end{proof}

The next statement follows by transferring~\cite[Lemma~3.9]{rw:st} across the equivalence of categories in Lemma~\ref{lem:smith-orbit}\eqref{it:tweq-nonzero}. 

\begin{lem}\label{lem:perfect-crit}
In the setting of Lemma~\ref{lem:smith-orbit}\eqref{it:tweq-nonzero}, an object $\cF \in \Db_{\Gm \ltimes V,\cX}(X,\bk)$ is $\varpi$-perfect if and only if the graded vector space $\bHom(\cL_X,\cF)$ is finite-dimensional.
\end{lem}

For $V$, $\cX$, and $X$ as above, we define the notion of \emph{Smith-even object} in the category $\Sm_{V,\cX}(X,\bk)$ as follows:
\begin{itemize}
\item In the situation of Lemma~\ref{lem:smith-orbit}\eqref{it:tweq-zero}, we declare all objects of $\Sm_{V,\cX}(X,\bk)$ to be \emph{Smith-even}.  (Of course, all objects are zero!)
\item In the situation of Lemma~\ref{lem:smith-orbit}\eqref{it:tweq-nonzero}, we declare an object $\cF \in \Sm_{V,\cX}(X,\bk)$ to be \emph{Smith-even} if it is isomorphic to a (finite) direct sum of copies of $\cL_X$.
\end{itemize}
In both cases, we say that $\cF \in \Sm_{V,\cX}(X,\bk)$ is \emph{Smith-odd} if $\cF[1]$ is odd.  A \emph{Smith parity object} is an object that can be written as a direct sum of a Smith-even object and a Smith-odd object.  The full subcategory of $\Sm_{V,\cX}(X,\bk)$ consisting of Smith parity objects is denoted by
\[
\SmParity_{V,\cX}(X,\bk).
\]
Of course, in the usual derived category $\Db_{\Gm\ltimes V,\cX}(X,\bk)$, we have the ``classical'' notions of \emph{even}, \emph{odd}, and \emph{parity} from~\cite{jmw:ps}.  (In the situation of Lemma~\ref{lem:smith-orbit}\eqref{it:tweq-nonzero}, an even object is a (finite) direct sum of objects of the form $\cL_X[2m]$.)  The full subcategory of parity sheaves is denoted by
\[
\Parity_{\Gm \ltimes V,\cX}(X,\bk) \subset \Db_{\Gm \ltimes V,\cX}(X,\bk).
\]
It is obvious that the quotient functor $\rQ$ sends even (resp.~odd) objects to Smith-even (resp.~Smith-odd) objects, so it restricts to a functor
\[
\rQ: \Parity_{\Gm \ltimes V,\cX}(X,\bk) \to \SmParity_{V,\cX}(X,\bk).
\]
However, it is possible for on object $\cF \in \Db_{\Gm \ltimes V,\cX}(X,\bk)$ to become Smith-even under $\rQ$ without being even.  The following lemma gives us some control over the behavior of such $\cF$.

\begin{lem}\label{lem:smeven-crit}
In the setting of Lemma~\ref{lem:smith-orbit}\eqref{it:tweq-nonzero}, an object $\cF \in \Db_{\Gm \ltimes V,\cX}(X,\bk)$ is Smith-even if and only if the graded vector space $\bHom(\cL_X,\cF)$ has only finitely many nonzero graded components in odd degrees.  Moreover, in this case, there exists an even complex $\cE \in \Db_{\Gm\ltimes V,\cX}(X,\bk)$ and a morphism
\[
\phi: \cE \to \cF
\]
such that $\rQ(\phi)$ is an isomorphism in $\Sm_{V,\cX}(X,\bk)$.
\end{lem}
\begin{proof}
Suppose $\cF$ is Smith-even, so that it isomorphic in $\Sm_{V,\cX}(X,\bk)$ to $\cL_X^{\oplus n}$.  By the definition of the Smith category, there exist maps
\[
\cF \xleftarrow{f} \cF' \xrightarrow{g} \cL_X^{\oplus n}
\]
such that the cones of $f$ and $g$ are both $\varpi$-perfect.  By Lemma~\ref{lem:perfect-crit}, the maps
\[
\bHom(\cL_X,\cF') \xrightarrow{f} \bHom(\cL_X,\cF)
\qquad\text{and}\qquad
\bHom(\cL_X,\cF') \xrightarrow{g} \bHom(\cL_X,\cL_X^{\oplus n})
\]
must both be isomorphisms in all but finitely many grading degrees.  Since the graded vector space $\bHom(\cL,\cL_X^{\oplus_n})$ is concentrated in even degrees, there can only be finitely many odd degrees in which $\bHom(\cL_X,\cF)$ is nonzero.

Conversely, suppose $\bHom(\cL_X,\cF)$ has only finitely many nonzero graded components in odd degrees.  For brevity, write $M = \bHom(\cL_X,\cF)$.  Then $M$ is a finitely generated graded module over the ring $\bEnd(\cL_X) \cong \coh^\bullet_\Gm(\pt,\bk)$, which is of course a polynomial ring on the generator $\can^2$, and thus a principal ideal domain.  

We can decompose $M$ as $M = M_{\mathrm{tor}} \oplus M'$, where $M_{\mathrm{tor}}$ is the graded submodule of torsion elements, and $M'$ is a free graded $\coh^\bullet_\Gm(\pt,\bk)$-module.  Since $M_{\mathrm{tor}}$ is a finitely generated torsion module over $\coh^\bullet_{\Gm}(\pt,\bk)$, its total dimension is finite.  Our assumptions imply that all homogeneous elements in odd degrees are torsion, so $M'$ is concentrated in even degrees.  Choose a homogeneous basis $m_1, \ldots, m_r$ for $M'$, and suppose these elements live in degrees $2d_1, \ldots, 2d_r$.

Choose an even integer $2N$ such that $2N \ge 2d_1, \ldots, 2d_r$, and also such that $M_{\mathrm{tor}}$ is concentrated in degrees${}<2N$.  Then the graded component $M^{2N}$ is equal to $(M')^{2N}$, and is spanned by the elements
\[
(\can^2)^{N-d_1}\cdot m_1, \ldots, (\can^2)^{N-d_r} \cdot m_r.
\]
From the structure theory of $\coh^\bullet_\Gm(\pt,\bk)$-modules, one checks that the action map
\begin{equation}\label{eqn:smeven-hom1}
\coh^\bullet_\Gm(\pt,\bk) \otimes M^{2N} \to M
\end{equation}
is injective, and an isomorphism in degrees${}\ge 2N$.  Its cokernel lives in degrees${}<2N$ and is finite-dimensional.

Recall that $M^{2N} = \Hom(\cL_X[-2N], \cF)$, and let
\[
\cE = \cL_X[-2N] \otimes M^{2N} = \cL_X[-2N] \otimes \Hom(\cL_X[-2N], \cF).
\]
We have a canonical map $\phi: \cE \to \cF$.  The induced map
\begin{equation}\label{eqn:smeven-hom}
\bHom(\cL_X,\cE) \to \bHom(\cL_X,\cF)
\end{equation}
can be identified with~\eqref{eqn:smeven-hom1}.  In particular,~\eqref{eqn:smeven-hom} is injective, and its cokernel is finite-dimensional.  That cokernel is identified with $\bHom(\cL_X, \text{cone of $\phi$})$.  By Lemma~\ref{lem:perfect-crit}, we deduce that $\phi$ becomes an isomorphism after passing to the Smith category.  Since $\cE$ is even, we conclude that $\cF$ is Smith-even.
\end{proof}

For $\cF, \cG \in \Parity_{\Gm \ltimes V,\cX}(X,\bk)$, consider the graded vector space $\bHom(\cF,\cG)$ and its even part $\bHomev(\cF,\cG)$.  As usual for parity sheaves, the first equivalence in Lemma~\ref{lem:smith-orbit}\eqref{it:tweq-nonzero} implies that $\bHomev(\cF,\cG)$ is a free module over $\coh^\bullet_\Gm(\pt,\bk)$ of finite rank.

For the following statement, observe that the functor $\rQ: \Db_{\Gm \ltimes V,\cX}(X,\bk) \to \Sm_{V,\cX}(X,\bk)$ sends parity sheaves to parity sheaves.

\begin{lem}\label{lem:sphom-orbit}
For $V$, $\cX$, and $X$ as above, the functor
\[
\rQ : \Parity_{\Gm \ltimes V,\cX}(X,\bk) \to \SmParity_{V,\cX}(X,\bk)
\]
is essentially surjective.  Moreover, the natural map
\begin{equation}\label{eqn:spho1}
\bHomev(\cF, \cG) \to \Hom(\rQ(\cF), \rQ(\cG))
\end{equation}
induces an isomorphism
\[
\coh^\bullet_\Gm(\pt,\bk)/(\can^2 - 1) \otimes_{\coh^\bullet_\Gm(\pt,\bk)} \bHomev(\cF,\cG) \simto \Hom(\rQ(\cF), \rQ(\cG)).
\]
\end{lem}
\begin{proof}
In the situation of Lemma~\ref{lem:smith-orbit}\eqref{it:tweq-zero}, there is nothing to prove, so let us assume that we are in the situation of Lemma~\ref{lem:smith-orbit}\eqref{it:tweq-nonzero}.  After passing through the equivalences of that lemma, we may instead study
\[
\rQ: \Parity_\Gm(\pt,\bk) \to \SmParity(\pt,\bk).
\]
Essential surjectivity is clear from the definitions.  For the description of $\Hom$-groups, it is enough to treat the case where $\cF$ and $\cG$ are indecomposable.  If one of these objects is even and the other odd, then all $\Hom$-groups in the lemma vanish.

Suppose now that $\cF$ and $\cG$ are both even, say $\cF = \underline{\bk}_\pt[2n]$ and $\cG = \underline{\bk}_\pt[2m]$.  (The case where they are both odd can be treated identically.)  Up to isomorphism, none of the $\Hom$-groups in the lemma are changed by even shifts, so we may assume without loss of generality that $m = n = 0$.  In this case,~\eqref{eqn:spho1} can be rewritten as
\begin{equation}\label{eqn:spho2}
\coh^\bullet_\Gm(\pt,\bk) \to \End_{\Sm_(\pt,\bk)}(\underline{\bk}_\pt).
\end{equation}
This is a ring homomorphism, and the right-hand side is just the field $\bk$ (by Lemma~\ref{lem:sm-point-endo}), so the map is certainly surjective.  On the other hand, $\can^2 - 1$ is in the kernel of~\eqref{eqn:spho2} by construction, and the quotient $\coh^\bullet_\Gm(\pt,\bk)/(\can^2 - 1)$ is $1$-dimensional, so we obtain our desired isomorphism
\[
\coh^\bullet_\Gm(\pt,\bk)/(\can^2 - 1) \simto \End_{\Sm_(\pt,\bk)}(\underline{\bk}_\pt).\qedhere
\]
\end{proof}

\subsection{Smith parity sheaves in general}
\label{ss:smithpar}

Let $V$ and $\cX$ be as in Section~\ref{ss:tweq}, and let $X$ be a variety over $\F$ with an action of $\Gm \ltimes V$.  In this subsection, we assume that:
\begin{itemize}
\item The group $V$ acts on $X$ with finitely many orbits.
\item Each $V$-orbit is stable under $\Gm$.
\item The group $\Gm$ has a fixed point on each $V$-orbit in $X$.
\end{itemize}
An object $\cF \in \Sm_{V,\cX}(X,\bk)$ is said to be \emph{Smith *-even} (resp.~\emph{Smith $!$-even}) if its $*$-pullback (resp.~$!$-pullback) to each $V$-orbit is Smith-even in the sense of Section~\ref{ss:smith-orbit}.  The object $\cF$ is \emph{Smith-even} it is both Smith $*$-even and Smith $!$-even.  It is \emph{Smith-odd} if $\cF[1]$ is Smith-even.  Finally, a \emph{Smith parity object} in $\Sm_{V,\cX}(X,\bk)$ is an object that is a direct sum of a Smith-even object and a Smith-odd object.  The full additive subcategory of $\Sm_{V,\cX}(X,\bk)$ consisting of Smith parity objects is denoted~by
\[
\SmParity_{V,\cX}(X,\bk).
\]

As in Section~\ref{ss:smith-orbit}, we also have the ``classical'' theory of parity sheaves, which form a full subcategory
\[
\Parity_{\Gm \ltimes V,\cX}(X,\bk) \subset \Db_{\Gm \ltimes V,\cX}(X,\bk).
\]

\begin{lem}
The categories $\Parity_{\Gm \ltimes V,\cX}(X,\bk)$ and $\SmParity_{V,\cX}(X,\bk)$ are both Krull--Schmidt.
\end{lem}
\begin{proof}[Proof sketch]
For $\Parity_{\Gm \ltimes V,\cX}(X,\bk)$, this is part of the general theory in~\cite[\S2.1]{jmw:ps}.  For $\SmParity_{V,\cX}(X,\bk)$, a proof for a specific example of $V$, $\cX$, and $X$ can be found in~\cite[Lemma~7.1]{rw:st}, but in fact this proof applies in the generality of the assumptions at the beginning of Section~\ref{ss:smithpar}.
\end{proof}

\begin{prop}
In either $\Parity_{\Gm \ltimes V,\cX}(X,\bk)$ or $\SmParity_{V,\cX}(X,\bk)$:
\begin{itemize}
\item The support of any indecomposable object is the closure of a single $V$-orbit.
\item For any $V$-orbit $Y$, there is up to isomorphism and shift at most one indecomposable object $\cE$ supported on $\overline{Y}$ with $\cE|_Y \ne 0$.  Moreover, $\cE|_Y$ is (up to shift) an irreducible local system.
\end{itemize}
\end{prop}
\begin{proof}[Proof sketch]
For $\Parity_{\Gm \ltimes V,\cX}(X,\bk)$, this statement is~\cite[Theorem~2.12]{jmw:ps}.  For $\SmParity_{V,\cX}(X,\bk)$, a proof for a specific example of $V$, $\cX$, and $X$ is outlined in~\cite[Proposition~7.3]{rw:st}, and the same arguments apply in the present generality.
\end{proof}

Let $Y$ be a $V$-orbit on $X$.  For there to exist an indecomposable parity sheaf or Smith parity sheaf whose support is $\overline{Y}$, we must have $\Db_{\Gm \ltimes V,\cX}(Y,\bk) \ne 0$ (cf.~Lemma~\ref{lem:smith-orbit}).  When they exist, we denote by
\[
\cE^{V,\cX}(Y,\bk),
\qquad\text{resp.}\qquad
\cE^{V,\cX}_\Sm(Y,\bk)
\]
the unique indecomposable parity sheaf, resp.~Smith parity sheaf, supported on $\overline{Y}$ and normalized so that
\[
\cE^{V,\cX}(Y,\bk)|_Y \cong \cL_Y[\dim Y],
\qquad\text{resp.}\qquad
\cE_\Sm^{V,\cX}(Y,\bk)|_Y \cong \cL_Y[\dim Y],
\]
where $\cL_Y$ is the unique irreducible $(V,\cX)$-equivariant local system on $Y$.

We say that $X$ \emph{has enough parity sheaves} if $\cE^{V,\cX}(Y,\bk)$ exists for every $V$-orbit $Y$ that fits the situation of Lemma~\ref{lem:smith-orbit}\eqref{it:tweq-nonzero}.  Similarly, $X$ \emph{has enough Smith parity sheaves} if $\cE_\Sm^{V,\cX}(Y,\bk)$ exists for each such $Y$.

\begin{lem}\label{lem:sphom-gen}
Let $V$, $\cX$, and $X$ be as above.  For $\cF, \cG \in \Parity_{\Gm \ltimes V,\cX}(V,\bk)$, the natural map
\[
\bHomev(\cF, \cG) \to \Hom(\rQ(\cF), \rQ(\cG))
\]
induces an isomorphism
\[
\coh^\bullet_\Gm(\pt,\bk)/(\can^2 - 1) \otimes_{\coh^\bullet_\Gm(\pt,\bk)} \bHomev(\cF,\cG) \simto \Hom(\rQ(\cF), \rQ(\cG)).
\]
\end{lem}
\begin{proof}[Proof sketch]
This statement can be deduced from Lemma~\ref{lem:sphom-orbit} by induction on the number of $V$-orbits.  See~\cite[Proposition~7.7]{rw:st} for a similar statement.
\end{proof}

\begin{cor}\label{cor:parity-full}
Let $V$, $\cX$, and $X$ be as above.  The functor
\[
\rQ: \Parity_{\Gm \ltimes V,\cX}(X,\bk) \to \SmParity_{V,\cX}(X,\bk)
\]
sends indecomposable parity sheaves to indecomposable Smith parity sheaves.  As a consequence, if $X$ has enough parity sheaves, then it also has enough Smith parity sheaves, and the functor above is essentially surjective.
\end{cor}

\subsection{A lifting lemma}
\label{ss:smithlift}

Supp $V$, $\cX$, and $X$ satisfy the assumptions of Section~\ref{ss:smithpar}.  For each $V$-orbit $Y \subset X$, let
\[
j_Y: Y \hookrightarrow X
\]
denote the inclusion map.  If $Y$ falls into case~\eqref{it:tweq-nonzero} of Lemma~\ref{lem:smith-orbit}, let $\cL_Y$ denote the unique irreducible $(\Gm \ltimes V,\cX)$-equivariant local system on $Y$.

\begin{defn}
An object $\cF \in \Db_{\Gm \ltimes V,\cX}(X,\bk)$ is said to be \emph{cohomologically $!$-even} if for every $V$-orbit $Y \subset X$ that admits nonzero $(\Gm \ltimes V,\cX)$-equivariant sheaves, the graded vector space
\[
\bHom(\cL_Y, j_Y^!\cF)
\]
is concentrated in even degrees.
\end{defn}

Any object that is $!$-even in the sense of~\cite{jmw:ps} is also cohomologically $!$-even, but a cohomologically $!$-even object can fail to be $!$-even.

Let $h: U \hookrightarrow X$ be the inclusion of a $V$-stable locally closed subvariety.  It is clear that if $\cF \in \Db_{\Gm\ltimes V,\cX}(X,\bk)$ is cohomologically $!$-even, then $h^!\cF$ is a cohomologically $!$-even object in $\Db_{\Gm \ltimes V,\cX}(U,\bk)$.  With this observation in mind, it is easy to check that several key facts from~\cite{jmw:ps} that involve $!$-even objects actually generalize to cohomologically $!$-even objects.  The following three statements correspond to~\cite[Proposition~2.7, Corollary~2.8, and Corollary~2.9]{jmw:ps}.

\begin{prop}\label{prop:bhom-directsum}
Let $\cF, \cG \in \Db_{\Gm \ltimes V,\cX}(X,\bk)$.  If $\cF$ is $*$-even and $\cG$ is cohomologically $!$-even, then there is a (noncanonical) isomorphism
\[
\bHom(\cF,\cG) \cong \bigoplus_{\text{$Y \subset X$ a $V$-orbit}} \bHom(j_Y^*\cF, j_Y^!\cG).
\]
\end{prop}

\begin{cor}
If $\cF$ is $*$-even and $\cG$ is cohomologically $!$-even, then the graded vector space $\bHom(\cF,\cG)$ is concentrated in even degrees.
\end{cor}

\begin{cor}\label{cor:cse-surj}
Let $U \subset X$ be a $V$-stable open subset.   If $\cF$ is $*$-even and $\cG$ is cohomologically $!$-even, then restriction to $U$ gives a surjection
\[
\bHom(\cF,\cG) \twoheadrightarrow \bHom(\cF|_U, \cG|_U).
\]
\end{cor}

\begin{prop}\label{prop:smeven-lift}
Assume that $X$ has enough parity sheaves.  Suppose $\cF \in \Db_{\Gm \ltimes V,\cX}(X,\bk)$ is cohomologically $!$-even, and that $\rQ(\cF)$ is Smith-even.  There exists an even complex $\cE \in \Parity_{\Gm \ltimes V,\cX}(X,\bk)$ and a morphism
\[
\phi: \cE \to \cF
\]
such that $\rQ(\phi): \rQ(\cE) \to \rQ(\cF)$ is an isomorphism.
\end{prop}
\begin{proof}
We proceed by induction on the number of $V$-orbits in the support of $\rQ(\cF)$.  If $\rQ(\cF) = 0$, there is nothing to prove.  Otherwise, let $Y$ be a $V$-orbit that is open in the support of $\rQ(\cF)$.  Since $\rQ(j_Y^!\cF) \cong j_Y^!\rQ(\cF)$ is even, Lemma~\ref{lem:smeven-crit} gives us a morphism
\[
\phi_0: \cE_0 \to j_Y^!\cF
\]
in $\Db_{\Gm \ltimes V,\cX}(Y,\bk)$, where $\cE_0$ is even, and $\rQ(\phi_0)$ is an isomorphism.  

Since $\cE_0$ is even, it can be written as $\cE_0 \cong \oplus_{i=1}^k \cL_Y[2d_i]$ for various integers $d_1, \ldots, d_k$.  Let
\[
\cE = \bigoplus_{i=1}^k \cE^{V,X}(Y,\bk)[2d_i - \dim Y].
\]
By construction, $\cE$ is an even complex supported on $\overline{Y}$ and satisfying $\cE|_Y \cong \cE_0$.  

Now let $\bar\jmath_Y: \overline{Y} \hookrightarrow X$ be the inclusion of the closure of $Y$.  Since $\bar\jmath_Y^!\cF$ is cohomologically $!$-even,  Corollary~\ref{cor:cse-surj} tells us that the map
\[
\Hom(\cE, \bar\jmath_{Y!}\bar\jmath_Y^!\cF) \to \Hom(\cE|_Y, (\bar\jmath_{Y!}\bar\jmath_Y^!\cF)|_Y) \cong \Hom(\cE_0,j_Y^!\cF)
\]
is surjective, so there exists a map $\phi_1: \cE \to \bar\jmath_{Y!}\bar\jmath_Y^!\cF$ whose restriction to $Y$ is $\phi_0$.  Let $\phi: \cE \to \cF$ denote the composition
\[
\cE \xrightarrow{\phi_1} \bar\jmath_{Y!}\bar\jmath_Y^!\cF \to \cF.
\]
Then $j_Y^!\phi$ is identified with $\phi_0$.

Complete $\phi$ to a distinguished triangle
\begin{equation}\label{eqn:smeven-dt}
\cE \xrightarrow{\phi} \cF \to \cF' \to.
\end{equation}
The support of $\cE$ (namely, $\overline{Y}$) is contained in that of $\rQ(\cF)$, so the support of $\rQ(\cF')$ is also contained in that of $\rQ(\cF)$.  Moreover, $Y$ is open in this support, and $j_Y^*\rQ(\phi) = j_Y^!\rQ(\phi) = \rQ(\phi_0)$ is an isomorphism, so in fact $\rQ(\cF')|_Y = 0$.  That is, $\rQ(\cF')$ has fewer strata in its support than $\cF$.

Next, we claim that $\rQ(\phi): \rQ(\cE) \to \rQ(\cF)$ admits a retraction.  Indeed, since $\rQ(\cE)$ and $\rQ(\cF)$ are both Smith-parity, and since $Y$ is open in their support, the map
\[
\Hom(\rQ(\cF),\rQ(\cE)) \twoheadrightarrow \Hom(\rQ(\cF)|_Y, \rQ(\cE)|_Y)
\]
is surjective, by the Smith counterpart of~\cite[Corollary~2.9]{jmw:ps}.  In particular, there exists a morphism $\xi: \rQ(\cF) \to \rQ(\cE)$ such that $\xi|_Y = \rQ(\phi_0)^{-1}$.  One can show that $\xi \circ \rQ(\phi)$ is an automorphism of $\rQ(\cE)$, and our claim follows.

Since $\rQ(\phi)$ admits a retraction, we see from~\eqref{eqn:smeven-dt} that $\rQ(\cF) \cong \rQ(\cE) \oplus \rQ(\cF')$.  In particular, $\rQ(\cF')$ is Smith-even.

We now claim that $\cF'$ is cohomologically $!$-even.  Let $Z \subset X$ be a $V$-orbit.  The distinguished triangle $j_Z^!\cE \to j_Z^!\cF \to j_Z^!\cF' \to$ gives rise to a long exact sequence that we can informally write as
\[
\bHom(\cL_Z,j_Z^!\cE) \xrightarrow{\phi_*} \bHom(\cL_Z,j_Z^!\cE) \to \bHom(\cL_Z, j_Z^!\cF') \to.
\]
where we write $\phi_*$ for the map induced by $\phi$. The first two terms are concentrated in even degrees (because $\cE$ is parity and $\cF$ is cohomologically $!$-even), so from the long exact sequence, we see that
\begin{align*}
\text{even-degree components of $\bHom(\cL_Z,j_Z^!\cF')$} &\cong \cok \phi_*, \\
\text{odd-degree components of $\bHom(\cL_Z,j_Z^!\cF')$} &\cong \ker \phi_*.
\end{align*}
Now, $\rQ(j_Z^!\cF')$ is Smith-even (because $\rQ(\cF')$ is Smith-even), so by Lemma~\ref{lem:smeven-crit}, $\bHom(\cL_Z,j_Z^!\cF')$ has only finitely many nonzero components in odd degrees.  Thus, the kernel of $\phi_*$ is finite-dimensional.  But $\bHom(\cL_Z,j_Z^!\cE)$ is a finitely generated free module over the PID $\coh^\bullet_\Gm(\pt,\bk)$, so any submodule is also free.  Since $\ker \phi_*$ is finite-dimensional and free, it must be $0$.  We conclude that $\bHom(\cL_Z,j_Z^!\cF')$ is concentrated in even degrees, as desired.

To summarize, $\cF'$ is cohomologically $!$-even; $\rQ(\cF')$ is Smith-even; and its support contains fewer orbits than that of $\cF$.  By induction, there exists an even complex $\cE'$ and a morphism
\[
\phi': \cE' \to \cF'
\]
that becomes an isomorphism after applying $\rQ$.  Since $\cE'$ is even, the composition $\cE' \to \cF' \to \cE[1]$ must vanish, so $\phi'$ factors through a morphism $\tilde\phi': \cE' \to \cF$.  Consider the following commutative diagram whose rows are distinguished triangles:
\[
\begin{tikzcd}
\cE \ar[d, equal] \ar[r] & \cE \oplus \cE' \ar[d, "{(\phi, \tilde \phi')}"] \ar[r] & \cE' \ar[d, "\phi'"] \ar[r, "0"] & {} \\
\cE \ar[r, "\phi"] & \cF \ar[r] & \cF' \ar[r] & {}
\end{tikzcd}
\]
After applying $\rQ$, the first and third vertical maps become isomorphisms, so the middle one does as well.
\end{proof}

\begin{prop}\label{prop:smeven-nat}
Suppose $\cF_1, \cF_2 \in \Db_{\Gm \ltimes V,\cX}(X,\bk)$ are cohomologically $!$-even, and that $\rQ(\cF_1)$ and $\rQ(\cF_2)$ are Smith-even.  Suppose also that (as in Proposition~\ref{prop:smeven-lift}) we have two even complexes $\cE_1, \cE_2 \in \Parity_{\Gm \ltimes V,\cX}(X,\bk)$ and two maps
\[
\phi_1: \cE_1 \to \cF_1,
\qquad
\phi_2: \cE_2 \to \cF_2
\]
that become isomorphisms in the Smith category.  Given any morphism
\[
f: \cF_1 \to \cF_2
\]
in $\Db_{\Gm \ltimes V,\cX}(X,\bk)$, there exists an integer $N \ge 0$ and a morphism $f': \cE_1[-2N] \to \cE_2$ such that the following diagram commutes:
\[
\begin{tikzcd}
\cE_1[-2N] \ar[r, "(\can^2)^N"] \ar[d, "f'"'] & \cE_1 \ar[r, "\phi_1"] & \cF_1 \ar[d, "f"] \\
\cE_2 \ar[rr, "\phi_2"] && \cF_2
\end{tikzcd}
\]
\end{prop}
\begin{proof}
Let $\cG_2$ be the cone of $\phi_2$.  This object is $\varpi$-perfect, so by Lemma~\ref{lem:perfect-crit}, for each orbit $Y$, the vector space $\bHom(\cL_Y,j_Y^!\cG_2)$ is finite-dimensional.  In particular, there is an integer $n_Y \ge 0$ such that $\bHom(\cL_Y, j_Y^!\cG_2)$ is concentrated in degrees${}\le 2n_Y$.  

Next, since $j_Y^*\cE_1$ is even, it can be written as a direct sum
\[
j_Y^*\cE_1 \cong \bigoplus_{i=1}^k \cL_Y[2d_i].
\]
Choose an integer $m_Y \ge \max \{0, d_1, \ldots, d_k \}$.  Observe that $\bHom(\cL_Y[2d_i],j_Y^!\cG_2)$ is concentrated in degrees${}\le 2n_Y + 2d_i \le 2n_Y + 2m_Y$, so $\bHom(j_Y^*\cE_1, j_Y^!\cG_2)$ is concentrated in degrees${}\le 2n_Y + 2m_Y$.  

Now choose an integer $N$ such that $-N < -n_Y - m_Y$ for all $Y$ such that $j_Y^*\cE_Y \ne 0$.  Since $n_Y, m_Y \ge 0$, we have $N > 0$.  By the previous paragraph, $\bHom(j_Y^*\cE_1[-2N], j_Y^!\cG_2)$ is concentrated in degrees${}\le 2n_Y + 2m_Y - 2N < 0$, so by Proposition~\ref{prop:bhom-directsum}, $\bHom(\cE_1[-2N], \cG_2)$ is also concentrated in degrees${}< 0$.  In particular, $\Hom(\cE_1[-2N], \cG_2) = 0$.

We see that the composition
\[
\cE_1[-2N] \xrightarrow{f \circ \phi_1 \circ (\can^2)^N} \cF_2 \to \cG_2
\]
vanishes, so $f \circ \phi_1 \circ (\can^2)^N$ factors through $\phi_2: \cE_2 \to \cF_2$.
\end{proof}


\end{document}